\documentclass[reqno,centertags,11pt]{amsart}

\usepackage{amssymb,amsmath,amsfonts,amssymb}%numbysec}
\usepackage{hyperref}
\usepackage{enumerate}
\usepackage{dsfont}

-\marginparwidth \oddsidemargin -9mm \evensidemargin \oddsidemargin
\usepackage{latexsym}
\advance\hoffset by 5mm

\def\@abssec#1{\vspace{.05in}\footnotesize \parindent .2in
{\bf #1. }\ignorespaces}

\newtheorem{theorem}{Theorem}[section]

\newtheorem{lemma}[theorem]{Lemma}
\newtheorem{proposition}[theorem]{Proposition}

\newtheorem{definition}[theorem]{Definition}
\newtheorem{remark}[theorem]{Remark}

\DeclareMathOperator{\divg}{div}
\DeclareMathOperator{\Div}{div}

\newcommand{\R}{\ensuremath{\mathbb{R}}}
\newcommand{\RR}{\ensuremath{\mathcal{R}}}
\newcommand{\WW}{\ensuremath{\mathcal{W}}}
\newcommand{\BB}{\ensuremath{\widetilde{\mathcal{B}}}}
\newcommand{\bb}{\ensuremath{\mathcal{B}}}

\newcommand{\Z}{\ensuremath{\mathbb{Z}}}
\newcommand{\N}{\ensuremath{\mathbb{N}}}

\newcommand{\Id}{\ensuremath{\mathrm{Id}}}

\newcommand{\dd}{\mathrm{d}}
\allowdisplaybreaks \numberwithin{equation}{section}

\begin{document}

\title[3D Boussinesq temperature front problem]{On the regularity of temperature fronts for the 3D viscous Boussinesq system}
\author{Omar Lazar}
\address{}
\email{}
\author{Yatao Li}

\author{Liutang Xue}

\date{\today}
\maketitle

\begin{abstract}
  We study the temperature front problem for the 3D viscous Boussinesq equation. We prove that the $C^{k,\gamma}$ ($k\geq 1$, $0<\gamma< 1$) and
$W^{2,\infty}$ regularity of a temperature front is locally preserved along the evolution
as well as globally preserved under a smallness condition in a critical space. In particular, beside giving another proof of the main result  in \cite{GGJ20}, we also extend it to a more general class of regular patch.

\end{abstract}

\section{Introduction}
In this paper, we study the 3D incompressible  Boussinesq system with viscous dissipation.
It is a well-known evolution equation which models the natural convection phenomena in geophysical flows (\cite{Ped87,Maj03}) and reads as follows, for any $ (t,x) \in \R_+\times\R^d$ where $d\geq 2$,
\begin{equation}\label{eq:BousEq}
\begin{split}
\begin{cases}
  \partial_t v + v\cdot\nabla v - \nu\Delta v + \nabla p = \theta e_d, \\
   \mathrm{div}\,v=0, \\
   \partial_t \theta + v\cdot \nabla \theta = 0, \\
  (\theta, v)|_{t=0}(x) = (\theta_0,v_0)(x),
\end{cases}
\end{split}
\end{equation}
Where  $e_d=(0,\cdots,0,1)^t$ and $\nu>0$ is the viscosity coefficient (without loss of generality, we  assume that $\nu=1$ for simplicity).
The unknowns are the scalar temperature $\theta$, the velocity field $v=(v^1,\cdots,v^d)^t$ and the scalar pressure $p$.

The Boussinesq system may be viewed as a generalization of very important models from incompressible fluid mechanics (\cite{Maj86,CD}). Indeed, when $\theta\equiv 0$, the Boussinesq system is nothing but the incompressible Navier-Stokes equation. The force models the so-called vortex-stretching phenomena which is observed in either 2D or 3D. In the 2D case, the Boussinesq equation is analogous to the 3D axisymmetric Euler equations with swirl \cite{MB02}.

In the last decades, the Boussinesq system has been widely studied and many results have been obtained.
Regarding the 2D viscous Boussinesq system \eqref{eq:BousEq}, the global regularity issue for smooth solutions has been first mentioned in a paper by Moffatt
\cite{Mof01} and was rigorously  proved by Chae \cite{Cha} and also by  Hou, Li \cite{HLi05}. We can also mention a work by Abidi and Hmidi \cite{AbH07} who proved the global well-posedness of strong solution with rough initial data.

 As far as weak solutions are concerned, several results have been proved. In particular,
Hmidi and Keraani \cite{HK07} studied the Cauchy problem for the 2D Boussinesq system \eqref{eq:BousEq} associated to data in the energy space $(\theta_0,v_0)\in L^2\times L^2$. The uniqueness of such a weak solutions has been proved by Danchin and Pa\"icu in \cite{DanPa08}. The regularity issue in Sobolev spaces (based on $L^2$) has been also studied in a work by Hu, Kukavica and Ziane \cite{HKuZ15}.

Some authors have also considered the case where the Boussinesq system has some anisotropic term or some partial dissipation. These cases are physically  relevant to consider when one wants to take into account large scale atmospheric and oceanic flows; we refer the interested reader to the following works \cite{AdCaoWu11,CaoW13,HKR10,LLT13,LiT16}.

Regarding  the Boussinesq equation \eqref{eq:BousEq} in higher dimension, {\it{i.e.}}  $d\geq 3$, Danchin and Pa\"icu \cite{DanPa08B}
studied the Cauchy problem associated to sufficiently regular initial data and proved that the system is globally well-posed under the smallness of the critical quantity

$\|v_0\|_{L^{d,\infty}}+ \nu^{-1}\|\theta_0\|_{L^{\frac{d}{3}}}$.

The latter result may be viewed as a generalization of the classical Fujita's and Kato's result on the incompressible 3D Navier-Stokes equations \cite{FujK64,Kato84}.
Later on, the same authors \cite{DanPa08} have been able to not only weakened the above Besov space assumption but also the smallness condition by considering Lorentz spaces.

As far as the inviscid Boussinesq system ({\it{i.e.}} $\nu=0$ in \eqref{eq:BousEq}) is concerned, it is worth recalling that, unlike the incompressible 2D Navier-Stokes equation, the global regularity issue for the 2D Boussinesq system is still an outstanding open problem.
One expects global regularity according to numerical simulations presented in \cite{ES94} at least when the domain is periodic,
however, in \cite{LH14b} the author proposed a potential scenario of finite-time blowup using a numerical approach if the domain is bounded and smooth (see also \cite{DFeff}).
The latter work \cite{LH14b} has opened the door to the introduction of several new 1D and 2D modified Boussinesq model 
 \cite{CKY15,CHKLSY,KT18, CCL19} and different physical scenarios in 3D \cite{Wid}.

One of the main goal of this paper is to study the
{Boussinesq temperature patch problem}.
This problem deals with the propagation of discontinuity of the temperature along a free interface which makes the study physically relevant \cite{Maj03}.
More precisely, we study the evolution of a initial temperature which is defined as the characteristic function of a bounded domain  $D_0\subset \R^d$ which is further assumed to be simply connected. The patch structure is preserved along the evolution so that $\theta(x, t) = \mathds{1}_{D(t)}(x)$ where
$D(t) = \psi_t(D_0)$, and where the particle trajectory $\psi_t(x)$ is given by
\begin{equation}\label{eq:flowmap}
  \frac{\partial \psi_t(x)}{\partial t} = v (t, \psi_t(x)),\quad \psi_t(x)|_{t=0}=x.
\end{equation}
This special type of patch solutions and the study of the propagation of their regularity along the evolution have been initiated some decades ago for the 2D incompressible Euler equations with initial vorticity $\omega_0 = \mathds{1}_{D_0}$. The global propagation of the $C^{k,\gamma}$ regularity (where $(k,\gamma) \in \mathbb{Z}_+ \times (0,1)$) of the initial patch in the case of the 2D incompressible Euler equation goes back to a work of Chemin \cite{Chem91}
using the para-differential calculus. Another approach - using geometric cancellations in some singular integral operators - has allowed Bertozzi and Constantin \cite{BerC} to give a new proof of the persistance of the regularity  of the patch. Some years after, Gamblin and Saint-Raymond \cite{GSR95} studied the 3D case. They were able to prove the local well-posedness for any initial patch with $C^{1,\gamma}$ regularity. We refer also to a work of Danchin \cite{Dan99} for similar regularity persistence result in higher dimension.

As we have recalled, the main question usually raised when studying patch-type solutions is whether the initial regularity of the patch is preserved along the evolution or not. The temperature patch problem of Boussinesq system \eqref{eq:BousEq} was initiated in a work of Danchin and Zhang \cite{DanZ17}. By using  para-differential calculus
they were able to prove that for sufficiently regular data $\theta_0\in B^{\frac 2q-1}_{q,1}(\R^2)$, $q\in(1, 2)$ with temperature patch initial data $\mathds{1}_{D_0}$,
the evolved temperature patch is globally well-posed and persists the $C^{1,\gamma}$-regularity for all time;
while in higher dimension for $\theta_0\in B^0_{d,1}\cap L^{\frac{d}{3}}(\R^d)$, $d\ge 3$,
which also contains $\mathds{1}_{D_0}$,
the same global $C^{1,\gamma}$-regularity persistence result is obtained under a smallness assumption of
$\|v_0\|_{L^{d,\infty}}+ \nu^{-1}\|\theta_0\|_{L^{\frac{d}{3}}}$.
More recently, Gancedo and Garcia-Ju$\mathrm{\acute{a}}$rez \cite{GGJ17}
in the 2D case gave a different proof of the global the $C^{1,\gamma}$ propagation of temperature patch,
and furthermore they obtained the $W^{2,\infty}$
and $C^{2,\gamma}$ regularity persistence of the patch boundary. They were able to use some hidden cancellations
in the time-dependent Calder$\mathrm{\acute{o}}$n-Zygmund operators and in the tangential derivative along the patch boundary.
Recently, in \cite{ChMX21}, Chae, Miao and Xue were able to
prove that the $C^{k,\gamma}$  (with $k\ge 1$ and $\gamma\in (0,1)$) regularity of the temperature patch is preserved along the evolution  of the 2D Boussinesq system
\eqref{eq:BousEq}.

Regarding the 3D case, the best known results deal with temperature patch with regularity $W^{2,\infty}$ and $C^{2,\gamma}$. Indeed, these cases have been studied independently by Danchin and Zhang \cite{DanZ17} and Gancedo and Garc\'ia-Ju\'arez \cite{GGJ20} in a recent work. They
considered the initial patch of non-constant values (e.g. $\theta_0 = \bar{\theta}_0 \mathds{1}_{D_0}$ 
with $\bar{\theta}_0$ defined on $\overline{D}_0$) which is usually called the temperature front initial data (\cite{Maj03}),
and proved that the  $C^{1,\gamma}$, $W^{2,\infty}$ and $C^{2,\gamma}$-regularity is preserved along the evolution if the critical quantity $\|v_0\|_{\dot H^{\frac{1}{2}}} + \|\theta_0\|_{L^1}$ is small enough.
They also proved that the temperature front preserved its regularity locally in time (without any smallness assumption).

Our goal in this paper is not only to improve the regularity of temperature front but also to give an alternative proof than the one in \cite{GGJ20}.
Roughly speaking, our generalization may be stated as follows: assume that $\partial D_0\in C^{k,\gamma}, \textrm{where} \  (k, \gamma) \in \mathbb{Z}_+ \times (0, 1) \  \textrm{then\,\,}\,\,\partial D(t)\in C^{k,\gamma}\,\, \textrm{for\ any\,\,} t\in [0,T]$ ($T=\infty$ if a smallness condition is assumed in some critical space). More precisely, we are going to prove the following theorem.
\begin{theorem}\label{thm:exi-reg}
 Let $v_0\in H^{\frac{1}{2}}(\R^3)$ be a divergence-free vector field and assume that  $\theta_0 \in L^1\cap L^s(\R^3)$, $s>3$.
Then, there exists a positive time $T$, which depends on $v_0,\theta_0$ and $s$, such that the Boussinesq system \eqref{eq:BousEq}
has a unique solution $(v,\theta,\nabla p)$ on $[0,T]$ which satisfies, for any $q\in(3,6)$, that
\begin{equation}\label{eq:v-the-es}
  v \in C([0,T], H^{\frac{1}{2}}(\R^3)) \cap L^2([0,T], H^{\frac{3}{2}}\cap L^\infty) \cap \widetilde{L}^1([0,T], \dot B^{1+\frac{3}{q}}_{q,\infty}),
  \quad \theta \in L^\infty([0,T],L^1\cap L^s(\R^3)).
\end{equation}
Let $T^*$ be the maximal time of existence of such a solution,

then $T^* = \infty$ provided that there exists a constant $c_*>0$ such that
\begin{equation}\label{eq:data-cond}
  \|v_0\|_{L^{3,\infty}(\R^3)} + \|\theta_0\|_{L^1(\R^3)}\leq c_*.
\end{equation}

Moreover, if one starts initially with the following temperature front
\begin{align}\label{eq.th0}
  \theta_0 (x)= \theta^*_1(x) \mathds{1}_{D_0}(x) + \theta_2^*(x) \mathds{1}_{D_0^c}(x),
\end{align}
where $D_0\subset \R^3$ is a bounded simply connected domain,
then one has the following regularity persistence results for all $t<T^*$.

\begin{enumerate}[(1)]
\item If $\partial D_0 \in C^{1,\gamma}(\R^3)$, $\gamma\in(0,1)$, $\theta^*_1\in L^\infty(\overline{D_0})$,
$\theta^*_2 \in L^1\cap L^\infty(\overline{D_0^c})$ and $v_0\in H^1\cap W^{1,3}(\R^3)$, then we have
\begin{align}
  \theta(x,t) = \theta^*_1(\psi^{-1}_t(x)) \mathds{1}_{D(t)}(x) + \theta^*_2(\psi_t^{-1}(x)) \mathds{1}_{D(t)^c}(x),
\end{align}
and
\begin{equation*}
  \partial D(t) = \psi_t(D_0) \in L^\infty([0,T], C^{1,\gamma}(\R^3)).
\end{equation*}

\item If additionally, $\partial D_0 \in W^{2,\infty}(\R^3)$, $\theta^*_1\in C^{\mu_1}(\overline{D_0})$,
$\theta^*_2\in C^{\mu_2}\cap L^1(\overline{D_0^c})$ where $\mu_1,\mu_2\in (0,1)$, and $v_0\in  W^{1,p}(\R^3)$, $p\in(3,\infty)$.
Then, $$ v\in L^\infty([0,T],W^{1,p})\cap L^\gamma([0,T], W^{2,\infty}))$$ with $1\leq \gamma < \frac{2p}{p+3}$ and
\begin{equation*}
  \partial D(t) \in L^\infty([0,T], W^{2,\infty}(\R^3)).
\end{equation*}

\item If additionally,\, $ \partial D_0 \in C^{k,\gamma}(\R^3)$,\,  $k \geq 2$, $\gamma\in(0, 1)$, $\theta^*_1\in C^{k-2,\gamma}(\overline{D_0})$,
$\theta_2^*\in C^{k-2,\gamma}\cap L^1 (\overline{D_0^c})$ and $v_0 \in H^1\cap W^{k,p}(\R^3)$\,
where $p\in (3, \infty)$. Then we have
\begin{equation}
  \partial D(t) \in L^\infty([0,T], C^{k,\gamma}(\R^3)).
\end{equation}
\end{enumerate}
\end{theorem}

Our main theorem clearly covers the cases  $C^{1,\gamma}$, $W^{2,\infty}$ as well as  $C^{2,\gamma}$
which have been obtained in \cite{GGJ20}. Beside allowing more general temperature front initial data \eqref{eq.th0}
than \cite{GGJ20} we are proving the persistence of a bigger class of regular temperature patch.

The existence and uniqueness part of Theorem \ref{thm:exi-reg}
is mainly based on \textit{a priori} estimates for the velocity $v$ in $\dot H^{\frac{1}{2}}$.  The estimate in $\dot H^{\frac{1}{2}}$
will be useful to prove higher regularity results (see {\it{e.g.}} Lemma \ref{lem:L2Linf}) and in particular the persistence results in Theorem \ref{thm:exi-reg}.
Note that the uniqueness issue is not that standard. This is mainly due to the fact that the velocity field $v$ may not be in $L^1_T(Lip)$. In order to avoid this difficulty, to prove the uniqueness we rather work in the Chemin-Lerner space in time and homogeneous Besov in space, that is  $\widetilde{L}^1_T(\dot B^{1+\frac{3}{q}}_{q,\infty})$
together with the use of Lemma \ref{lem:uniq} (which may be viewed as a slight modification of the uniqueness result in \cite{DanPa08}).
It is worth mentioning that the global existence is already proved by Danchin and Pa\"icu \cite{DanPa08}.

Based on the obtained existence and uniqueness result in critical spaces, the remaining part of Theorem \ref{thm:exi-reg} deals with the regularity persistence results of temperature front.
The main goal is to measure the regularity of an initial temperature front $\partial D_0$  evolving according to the Boussinesq sytem.
The geometric quantity that one has to study is therefore $\partial D(t) = \psi_t(\partial D_0)$ where $\psi_t$ stands for the particule trajectory which is given by the ODE \eqref{eq:flowmap}. The latter quantity is strongly related to the striated regularity of the system $\mathcal{W}(t)= \{W^i(t)\}_{1\leq i\leq 5}$ (see Lemma \ref{lem:sr-cond}), where $\mathcal{W}(t) $ is the system of evolved tangential divergence-free vector fields satisfying \eqref{eq:Wi}. We want to emphasize that, compared to \cite{GGJ20}, we introduce a new quantity $\Gamma: = \Omega - (\mathcal{R}_{-1,2}\theta, - \mathcal{R}_{-1,1}\theta,0)^t$, which satisfies the equation
\begin{equation}\label{eq.Gamm0}
  \partial_t \Gamma + v\cdot\nabla \Gamma -\Delta \Gamma =\Omega\cdot \nabla v+ \big([\mathcal{R}_{-1,2}, v\cdot\nabla]\theta,  -[\mathcal{R}_{-1,1}, v\cdot\nabla]\theta,0\big)^t
\end{equation}
with $\mathcal{R}_{-1,j} := \partial_j \Lambda^{-2}$, $\Lambda=(-\Delta)^{\frac12}$, $j=1,2$,
and $\Omega=\nabla\wedge v$ the vorticity field. Such a good unknown $\Gamma$ plays a central role in the proof of the persistence of the regularity in the class $C^{1,\gamma}$, $W^{2,\infty}$ and in several striated estimates for the velocity $v$ in the sequel.

In order to show the $C^{1,\gamma}$ and $ W^{2,\infty}$  persistence of regularity of the patch,
according to Lemma \ref{lem:flow},
it suffices to show that the velocity field $v$ belongs to $L^1_T(C^{1,\gamma})$ and $L^1_T(W^{2,\infty})$ respectively.
{\it{Via}} the Biot-Savart law, we may decompose $\nabla v$ into the sum of controlled terms as follows
\begin{align}\label{eq.v.GamThe0}
  \nabla v = (-\Delta)^{-1}\nabla\nabla \wedge \Omega = \underbrace{\Lambda^{-2}\nabla\nabla\wedge \Gamma}_{under \ control} + \nabla^2\partial_3  \Lambda^{-4} \theta
  + \Lambda^{-2}\nabla \theta\otimes e_3 .
\end{align}
By exploiting the commutator estimate and making use of the smoothing effect, the quantity $\Gamma$
indeed can be well controlled.
Since $\theta$ belongs to $L^1\cap L^\infty$ uniformly in time,
one can directly prove that $\nabla^2\partial_3  \Lambda^{-4} \theta$ and $\Lambda^{-2} \nabla\theta $
belong to $L^1_T(C^\gamma)$ for all $\gamma\in (0, 1)$,
which then ensures the wanted $L^1_T(C^{1,\gamma})$ estimate of $v$. 
Then, in order to control the quantity $\|\nabla^2v\|_{L^1_T(L^\infty)}$,
it suffices to prove that $\nabla^3\partial_3  \Lambda^{-4} \theta$ and $\nabla^2\Lambda^{-2} \theta e_3$  belong to $L^\infty_T(L^\infty)$.

It is worth mentioning that the situation is quite analogous with that in the vorticity patch problem of 2D Euler equations
(where one needs to control $L^1_T(L^\infty)$ of $\nabla v = \nabla \nabla^\perp (-\Delta)^{-1} \mathds{1}_{D(t)}$),
we here adopt the method of striated estimates initiated by Chemin \cite{Chem88,Chem91} and its generalization to 3D Euler by \cite{GSR95}.
Although $\nabla^3 \partial_3 \Lambda^{-4}$ is a fourth-order Riesz transform, by using Lemmas \ref{lem:stra-exp} and \ref{lem:par-jk-Cgam},
we can conclude that $\nabla^3 \partial_3 \Lambda^{-4} \theta$ and moreover $\nabla^2 v$ belong to $L^1([0,T], L^\infty)$ as desired.

To prove the persistence of the $C^{2,\gamma}$ regularity of the patch, instead of using the complicated contour dynamics method as in \cite{GGJ20},
we essentially consider the regularity of a series of evolved tangential vector fields $\mathcal{W}(t)=\{W^i(t)\}_{1\leq i\leq 5}$ for the temperature front.
According to Lemma \ref{lem:sr-cond}, it suffices to show $\mathcal{W}\in L^\infty_T(C^{1,\gamma})$.
In estimating the $C^{\gamma}$  norm  of $\nabla\mathcal{W}$,  we mainly need to treat the striated term $\partial_{\mathcal{W}}\nabla v:=\mathcal{W}\cdot\nabla^2 v$
in $L^1_T (C^\gamma)$. By using \eqref{eq.v.GamThe0}, we treat the $\Gamma$ term and $\theta$ term separately,
and together with striated estimates in Lemma \ref{lem:str-es1} and commutator estimate in Lemma \ref{lem:comm-es}, 
we can give the striated estimate $\|\partial_{\mathcal{W}}\nabla v\|_{L^1_T (C^\gamma)}$.
Then, Gr\"onwall's inequality ensures the desired control in $L^1_T(C^{1,\gamma})$  of $\mathcal{W}$.

In order to obtain the regularity persistence in $C^{k,\gamma}$, with $k\geq 3$, according to Lemma \ref{lem:sr-cond},
we only need to show the striated regularity $\partial^{k-1}_\mathcal{W}\mathcal{W}\in L^\infty_T(C^\gamma)$. To do so, we use the induction method. Moreover, the function spaces $\bb^{s,\ell}_{r,\mathcal{W}}$ (see Definition \ref{def:Besov-str}) and the high-order striated estimates established in Lemma \ref{lem:prod-es2} will be essential throughout the proof of the main theorem.

Suppose that we already have controlled the quantities $\mathcal{W},\nabla v$ and $\Gamma$ in some well-chosen striated
spaces $\bb^{s,\ell}_{r,\mathcal{W}}$ with $\ell\in\{1,..., k-2\}$ (see \eqref{assup.el} below),
we aim at showing the corresponding estimates by replacing $\ell$ with  $\ell+1$.

To control of  the $C^{-1,\gamma}$ norm of $\partial^\ell _\mathcal{W}\nabla^2\mathcal{W}$, one needs to estimate the quantity
$\partial^{\ell+1}_\mathcal{W}\nabla^2v$ in $L^1_T (C^{-1,\gamma})$ and other (low order) striated estimates. In view of \eqref{eq.v.GamThe0} again, we treat the term in $\Gamma$ and the term in $\theta$ separately. More precisely, by making use of the smoothing estimate of transport-diffusion equation and the induction assumptions,
the $L^1_T (C^{-1,\gamma} )$ norm of $\partial^{\ell+1}_\mathcal{W}\nabla^2\Lambda^{-2}\nabla\wedge \Gamma$ can be controlled in terms of
$\Gamma$ itself and $\mathcal{W}$ in some $\bb^{s,\ell}_{r,\mathcal{W}}$ norms.
Then, by using Lemma \ref{lem:str-reg} which deals with the striated regularity of $\theta$, we can
also control the  $L^1_T (C^{-1,\gamma})$ norms of $\partial^{\ell+1}_\mathcal{W}\nabla^3\partial_3 \Lambda^{-4} \theta$ and
$\partial^{\ell+1}_\mathcal{W}\nabla^2\Lambda^{-2} \theta$.

Collecting all these estimates and applying Gr\"onwall inequality we get  the wanted estimates at the rank $\ell+1$,
so that the induction process guarantees the desired estimate \eqref{eq:Targ7}, which implies the  $L^\infty_T(C^\gamma)$ regularity of $\partial_\mathcal{W}^{k-1}\mathcal{W}$.

\vskip1.5mm
The paper is organized as follows. In Section \ref{sec:pril},
we introduce the notion of admissible conormal system adapted to the temperature patch and the definition of special type of Besov spaces $\bb^{s,\ell}_{p,r,\mathcal{W}}$
which are known in the literature as striated Besov spaces.
In the latter section, we also include a series of useful results on estimates in the striated setting, and we also present some auxiliary lemmas.
Then, in Section \ref{sec:exi-uni} we give a detailled proof of the existence and uniqueness part in Theorem \ref{thm:exi-reg}.
The sections \ref{sec:C1gam} and \ref{sec:Ck-gam} are dedicated to the proof of the regularity persistence results stated in Theorem \ref{thm:exi-reg}, namely,
we prove that the regularity of the patch in the space $C^{1,\gamma}$, $W^{2,\infty}$, $C^{2,\gamma}$ 
are preserved along the evolution of the temperature front in Section \ref{sec:C1gam}. The persistence in the $C^{k,\gamma}$  regularity  for any $k\geq 3$ in presented in Section \ref{sec:Ck-gam}.
Finally, we present a detailed proof of some technical lemmas in the appendix, namely, Lemmas \ref{lem:parW-mDf} and \ref{lem:L2Linf}.

\section{Preliminaries and auxiliary lemmas}\label{sec:pril}

\subsection{The admissible conormal vector system}

In order to obtain the (higher-order) regularity persistence of patches,  we need to introduce the striated (i.e. conormal) vector fields.
 
Let us first present the definition of an admissible system (see \cite{GSR95}).
\begin{definition}\label{definition-W}
A system $\mathcal W = (W^1 , W^2 ,..., W^N )$ of $N$ continuous
vector fields is said to be admissible if the function
\begin{align*}
  [\mathcal W]^{-1}\stackrel{\mathrm{def}}{=}\Big(\frac 2{N(N-1)}\sum_{\mu<\nu}|W^\mu\wedge W^\nu|^2\Big)^{-\frac {1}{4}} < \infty,
\end{align*}
where the wedge product $X\wedge Y$ is defined as
$X \wedge Y= (X_2Y_3-X_3Y_2,X_3Y_1-X_1Y_3,X_1Y_2-X_2Y_1)^t $ for any two vector fields $X = (X_1 ,X_2 ,X_3 )^t$ and
$Y = (Y_1 ,Y_2 ,Y_3)^t$.
\end{definition}

Now let $D_0\subset \R^3$ be a bounded simply connected domain with $\partial D_0\in C^{k,\gamma}$ with $k\geq 1$, $0<\gamma<1$.
Then there exists a function $F\in C^{k,\gamma}$ such that $F\equiv 0$ on $\partial D_0$ and $\nabla F|_{\partial D_0} \neq 0$.
The following result deals with the existence of an admissible system of divergence-free tangential vector fields of $\partial D_0$.
\begin{proposition}\label{prop:5vec}
  For any two-dimensional compact $C^{k,\gamma}$-submanifold $\partial D_0$ of $\R^3$,
we can find an admissible system consisting of five divergence-free and $C^{k-1,\gamma}$-vector fields tangent to $\partial D_0$.
\end{proposition}

\begin{proof}[Proof of Proposition \ref{prop:5vec}]
  The proof follows an idea from Gamblin and Saint-Raymond (see \cite{GSR95}), we include it for the sake of completeness.
By continuity, there exists some $\epsilon >0$ such that $\nabla F \neq 0$ on $\Sigma_\epsilon = \{x\in \R^3: \mathrm{dist}(x,\partial D_0)\leq \epsilon \}$.
Let $\chi\in C^\infty$ be a bump function so that $\chi =1$ on $\R^3\setminus \Sigma_\epsilon$ and $\chi =0$ on $\Sigma_{\epsilon/2}$. Set
\begin{eqnarray} \label{eq:Wi0}
   W^1_0 &=& (0,-\partial_3 F,\partial_2 F)^t, \nonumber \\
   W^2_0 &=&  (\partial_3 F,0, -\partial_1 F)^t, \nonumber \\
   W^3_0 &=& (-\partial_2 F, \partial_1 F, 0)^t, \nonumber \\
   W^4_0 &=& (\partial_3(\chi x_3),0, - \partial_1 (\chi x_3))^t, \nonumber\\
   W^5_0 &=& (- \partial_2(\chi x_1), \partial_1(\chi x_1),0)^t. \nonumber
\end{eqnarray}
Clearly, $\{W^1_0, \cdots, W^5_0\}$ is composed of $C^{k-1,\gamma}$ divergence-free vector fields that are all tangent to $\Sigma$.
Moreover, this system is admissible because of the fact that $|W^4_0\wedge W^5_0|\equiv 1$ on $\R^3\setminus \Sigma_\epsilon$
and $|W^1_0\wedge W^2_0|^2 + |W^1_0\wedge W^3_0|^2 + |W^2_0\wedge W^3_0|^2 = |\nabla F|^2\geq \delta >0$ on $\Sigma_\epsilon$.
\end{proof}

Consider an initial temperature patch $\theta_0$ which is non-constant and satisfies \eqref{eq.th0}.
Since the patch boundary $\partial D_0$ is a
two-dimensional compact submanifold of  $C^{k,\gamma}(\R^3)$ regularity, where $k\ge1$ and $\gamma\in (0,1)$,
thanks to Proposition \ref{prop:5vec}, we can find an admissible system $\mathcal W_0 = \{W^1_0, \cdots, W^5_0\}$
such that
\begin{equation}\label{W-i-0}
\begin{split}
  \divg W^i_0=0 \,\,\mathrm{and}\,\, W^i_0\in C^{k-1,\gamma}(\R^3)\, \mathrm{are \,\,tangent\,\, to \,\,}\partial D_0,  \quad i = 1,\cdots,5.
\end{split}
\end{equation}

As in the vorticity patch problem for the Euler equation (see {\it{e.g.}} \cite{Chem88,Chem91,GSR95}),
we consider the evolution of the vectors $\mathcal{W}(t) = \{W^1(t), \cdots, W^5(t)\}$ where $W^i(t)$ is a solution of

\begin{equation}\label{eq:Wi}
  \partial_t W^i + v \cdot\nabla W^i = W^i\cdot\nabla v = \partial_{W^i}v, \quad\quad W^i|_{t=0}(x)=W^i_0(x),
\end{equation}
for all $i=1,\cdots,5$ and all initial data $W^i_0$ verifying \eqref{eq:Wi0}.
Since $\divg W^i$ satisfies the transport equation
\begin{equation}\label{eq.divX}
  \partial_t(\divg W^i) + v\cdot\nabla (\divg W^i)=0,\qquad \divg W^i|_{t=0}=\divg W^i_0 =0,
\end{equation}
we see that $W^i(t,x)$ is still divergence-free.
According to \cite{MB02}, we also have
\begin{align}\label{exp:Wit}
  W^i(t,x) = (\partial_{W^i_0}\psi_t)\big(\psi_t^{-1}(x)\big),\quad i=1,\cdots,5,
\end{align}
where $\psi_t:\R^3\rightarrow \R^3$ is the particle-trajectory map (that is a solution to the ODE \eqref{eq:flowmap}) with inverse $\psi_t^{-1}$.

The following result is fundamental as it shows the deep relationship between the (higher-order) boundary regularity of $\partial D(t)$
with the striated regularity of the system $\mathcal{W}(t)=\{W^i(t)\}_{1\leq i\leq 5}$. 
\begin{lemma}\label{lem:sr-cond}
 Let $T>0$,  $k\ge 2$ and $\gamma\in(0,1)$.
Let $\psi_t(\cdot): \R^3\rightarrow \R^3$ defined by \eqref{eq:flowmap} be the measure-preserving bi-Lipschitz particle-trajectory map on $[0,T]$.
Then, the temperature patch boundary $\partial D(t)=\psi_t( \partial D_0)$ preserves its $C^{k,\gamma}(\R^3)$ regularity on the time interval $[0,T]$, provided that
\begin{align}\label{eq:targ-sr}
  \textrm{$\partial^\ell_{\mathcal{W}} \mathcal{W} \in L^\infty([0,T], C^\gamma(\R^3))$ \quad for all \quad  $0\leq \ell\leq k-1$.}
\end{align}
\end{lemma}

\begin{proof}[Proof of Lemma \ref{lem:sr-cond}]
  To prove this lemma, we follow the same approach as \cite{GSR95,LL19}. More precisely, we notice that
  since $D_0$ is a simply-connected domain with its boundary $\partial D_0\in C^{k,\gamma}(\R^3)$, according to the finite covering theorem,
there exists a finite number of charts $\{V_\beta, \varphi_\beta\}_{1\leq \beta\leq m}$ covering the two-dimensional compact
$C^{k,\gamma}$-submanifold $\partial D_0$ with
\begin{equation*}
\begin{split}
  \varphi_\beta:\quad  U_\beta&\longmapsto V_\beta \subset \R^3, \\
  (s_1,s_2)&\longmapsto \varphi_\beta(s_1,s_2)= (\varphi_\beta^1,\varphi_\beta^2,\varphi_\beta^3)(s_1,s_2)\in V_\beta,
\end{split}
\end{equation*}
where $U_\beta$ is an open set of $\R^2$, $V_\beta$ is an open set of $\R^3$ near neighborhood of $\partial D_0$,
$\varphi_\beta\in C^{k,\gamma}$, $\beta=1,\cdots,m$.

In order to show that $\partial D(t)=\psi_t(\partial D_0)\in C^{k,\gamma}$, it suffices to prove
\begin{align}\label{eq:targ1}
  \partial^{k_1}_{s_1} \partial^{k_2}_{s_2}( \psi_t(\varphi_\beta(s_1,s_2)))\in L^\infty_T(C^\gamma(U_\beta)),\ \text{for all} \  0\leq k_1+k_2\leq k.
\end{align}
Define the tangential vector fields $Y_i(\varphi_\beta(s_1,s_2)) : = \partial_{s_i} \varphi_\beta(s_1,s_2)$ for $i=1,2$, then $Y_i \in C^{k-1,\gamma}(V_\beta)$, and by the chain rule we get that
\begin{equation*}
\begin{split}
  \partial_{s_i} (\psi_t(\varphi_\beta(s_1,s_2))) & = \partial_{s_i} \varphi_\beta(s_1,s_2) \cdot \nabla \psi_t(\varphi_\beta(s_1,s_2)) \\
  & =  Y_i(\varphi_\beta(s_1,s_2)) \cdot \nabla \psi_t(\varphi_\beta(s_1,s_2)) = (\partial_{Y_i} \psi_t)(\varphi_\beta(s_1,s_2)), \quad i=1,2.
\end{split}
\end{equation*}
By induction we see that
\begin{align*}
  \partial^{k_1}_{s_1} \partial^{k_2}_{s_2}( \psi_t(\varphi_\beta(s_1,s_2)))=
  (\partial^{k_1}_{Y_1} \partial^{k_2}_{Y_2} \psi_t)(\varphi_\beta(s_1,s_2)),
\end{align*}
Then, using the fact that $\varphi_\beta\in C^{k,\gamma}$,  it suffices to prove that
\begin{align}\label{eq:targ2}
  (\partial^{k_1}_{Y_1} \partial^{k_2}_{Y_2} \psi_t)(\cdot) \in L^\infty_T(C^\gamma(V_\beta)), \ \text{for all} \ 0\leq k_1+k_2\leq k.
\end{align}

Since $\mathcal{W}_0 =\{W^i_0\}_{1\leq i\leq 5}$ is an admissible system (see Definition \ref{definition-W}),
for $(s_1,s_2)\in U_\beta$, without loss of generality we may assume that for $i_1\neq i_2\in\{1,\cdots,5\}$,
\begin{equation*}
\begin{split}
  |W^{i_1}_0 \wedge W^{i_2}_0|(\varphi_\beta(s_1,s_2))>0.
\end{split}
\end{equation*}
Then, the fact that $W^i_0\in C^{k-1,\gamma}$ guarantees that there exists an open set $\widetilde{V}_\beta\subset V_\beta$ such that
\begin{equation}\label{base-condition}
\begin{split}
  \varphi_\beta(s_1,s_2)\in \widetilde{V}_\beta,  \quad \inf_{x\in \widetilde{V}_\beta}|W^{i_1}_0\wedge W^{i_2}_0|(x)>0.
\end{split}
\end{equation}
This means that $\{W^{i_1}_0, W^{i_2}_0\}$ can be seen as the base of the tangent vector fields of $\partial D_0$ on $\widetilde{V}_\beta$.
Hence, $Y_1, Y_2$ on $\widetilde{V}_\beta$ can be expressed as a linear combination of $W^{i_1}_0$ and $W^{i_2}_0$:
\begin{equation}\label{exp:Y1Y2}
\begin{split}
  Y_1=\mu_{11}W^{i_1}_0+\mu_{12}W^{i_2}_0,\quad Y_2=\mu_{21}W^{i_1}_0+\mu_{22}W^{i_2}_0,
\end{split}
\end{equation}
where the coefficients are determined by
\begin{equation*}
\begin{split}
  \mu_{jk}= \frac{\langle Y_j,W^{i_k}_0 \rangle
  |W^{i_{\bar k}}_0|^2 - \langle Y_j,W^{i_{\bar k}}_0 \rangle \langle W^{i_1}_0,W^{i_2}_0 \rangle}{|W^{i_1}_0\wedge W^{i_2}_0|^2},
  \quad \text{for} \quad j,k=1,2 \quad \text{and} \quad (k,\bar{k})\in\{(1,2),(2,1)\},
\end{split}
\end{equation*}
where $\langle\cdot,\cdot\rangle$ denotes the inner product of $\R^3$.
Then, the fact  $W_0^i, Y_j \in C^{k-1,\gamma}(V_\beta)$ together with \eqref{base-condition},
 allows us to find that the coefficients $\mu_{ij} \in C^{k-1,\gamma}(\widetilde{V}_\beta)$, for all $k\ge 2$ and $i,j=1,2$.
This property combined with \eqref{exp:Y1Y2} imply that in order to show \eqref{eq:targ2}, it suffices to prove that for $i\neq j \in \{1,\cdots,5\}$ one has
\begin{equation}\label{eq:targ3}
\begin{split}
  \partial_{W^j_0}^{k_1} \partial_{W^i_0}^{k_2} \psi_t \in L^\infty_{T}(C^\gamma\big(\widetilde{V}_\beta)\big),\quad \forall \ 0\leq k_1 + k_2 \leq k .
\end{split}
\end{equation}

On the other hand, by using the identity \eqref{exp:Wit} and its equivalent form, that is, $W^i(t,\psi_t(x)) = \partial_{W^i_0}\psi_t (x) $ one finds, {\it{via}} a direct computation that for all $i,j\in \{1,\cdots,5\}$,
\begin{equation*}
\begin{split}
  \partial_{W^j_0}\partial_{W^i_0}\psi_t(x) = \partial_{W^j_0} (W^i(t,\psi_t(x))) & = \partial_{W^j_0}\psi_t(x) \cdot (\nabla W^i)(t,\psi_t(x))
  = (\partial_{W^j}W^i)(t,\psi_t(x)).
\end{split}
\end{equation*}
By induction, we see that for all $k_1,k_2\in \N$,
\begin{equation*}
\begin{split}
  \partial_{W^j_0}^{k_1} \partial_{W^i_0}^{k_2+1} \psi_t(x) = (\partial_{W^j}^{k_1} \partial_{W^i}^{k_2} W^i)(t,\psi_t(x)).
\end{split}
\end{equation*}
Hence in order to show \eqref{eq:targ3}, since $\psi_t\in L^\infty_T(W^{1,\infty}) \subset L^\infty_T(C^\gamma)$,
one only needs to prove that
\begin{align*}
  \partial_{W^j}^{k_1} \partial_{W^i}^{k_2} (W^i,W^j)(t,\psi_t(x)) \in L^\infty_T(C^\gamma(\widetilde{V}_\beta)), \quad \text{for all}  \ 0\leq k_1+k_2 \leq k-1,
\end{align*}
but the latter is a direct consequence  of
\begin{align}\label{eq:targ4}
  \partial_{W^j}^{k_1} \partial_{W^i}^{k_2} (W^i,W^j)(t,\cdot) \in L^\infty_T(C^\gamma(\R^3 )), \quad \text{for all} \ 0\leq k_1+k_2 \leq k-1.
\end{align}

Since obviously \eqref{eq:targ-sr} implies \eqref{eq:targ4} then the proof of Lemma \ref{lem:sr-cond} is done.
\end{proof}

The lemma below deals with the striated estimate of the initial temperature front.
\begin{lemma}\label{lem:str-reg}
  Let $k\geq 2$, and $0<\gamma<1$. Assume that $D_0 \subset \R^3$ is a bounded simply connected domain with boundary
$\partial D_0\in C^{k,\gamma}(\R^3)$,
and $\theta_0(x) = \theta^*_1(x) \mathds{1}_{D_0}(x) + \theta_2^*(x) \mathds{1}_{D_0^c}(x)$ satisfies $\theta^*_1 \in C^{k-2,\gamma}(\overline{D_0})$
and $\theta_2^* \in L^1 \cap C^{k-2,\gamma}(\overline{D_0^c})$.
Let $\mathcal{W}_0 = \{W^i_0\}_{1\leq i\leq 5}$ be defined as in \eqref{eq:Wi0}. Then for all $\ell\in \{1,2,\cdots, k-1\}$ we have
\begin{align}\label{eq:str-reg}
  \partial_{\mathcal{W}_0}^\ell \theta_0 \in C^{-1,\gamma}(\R^3).
\end{align}
Besides, if $\partial D_0\in C^{1,\gamma}$, and $\theta_0(x) = \theta^*_1(x) \mathds{1}_{D_0}(x) + \theta^*_2(x) \mathds{1}_{D_0^c}(x)$
with $\theta_1^* \in C^{\mu_1}(\overline{D_0})$
and $\theta_2^*\in L^1\cap C^{\mu_2}(\overline{D_0^c})$, $0<\mu_1,\mu_2<1$, we have
\begin{align}\label{eq:str-reg2}
  \partial_{\mathcal{W}_0} \theta_0 \in C^{-1,\mu}(\R^3),\quad \textrm{with\;\;\;} \mu=\min\{\mu_1,\mu_2\}.
\end{align}
\end{lemma}

\begin{proof}[Proof of Lemma \ref{lem:str-reg}]
We start with the proof of \eqref{eq:str-reg}. We follow an idea of F. Sueur \cite{Sue15}.
More precisely,
by using Rychkov's extension theorem (\cite{Ryc99}),
there exist two functions $\widetilde{\theta}^*_{i}(\R^3) \in C^{k-2,\gamma}(\R^3), i=1,2$ such that
$\widetilde{\theta}^*_1|_{D_0}=\theta^*_1(x) $  and $\widetilde{\theta}^*_2|_{D^c_0}=\theta_2^*(x) $.
Hence $\theta_0=\widetilde{\theta}^*_1 \cdot \mathds{1}_{D_0} + \widetilde{\theta}^*_2 \cdot \mathds{1}_{D^c_0}$.
Therefore, it is sufficient to show that for all $1\leq \ell\leq k-1$, one has
\begin{align*}
  \partial_{\mathcal{W}_0}^\ell(\widetilde{\theta}^*_1\cdot \mathds{1}_{D_0}), \,
  \partial_{\mathcal{W}_0}^\ell(\widetilde{\theta}^*_2\cdot \mathds{1}_{D^c_0})\in {C^{-1,\gamma}}(\R^3).
\end{align*}
Since the divergence-free vector fields $W^j_0 \in {\mathcal{W}_0}$ $(1\leq j\leq 5)$
are tangential to the patch boundary $\partial D_0$,
the operator $\partial_{\mathcal{W}_0}^\ell$ commutes with the characteristic functions $\mathds{1}_{D_0}$ and $\mathds{1}_{D^c_0}$.
Moreover, since the characteristic functions
$\mathds{1}_{D_0}$ and $\mathds{1}_{D^c_0}$ are pointwise multipliers in the H\"older space ${C^{-1,\gamma}}(\R^3)$ (see {\it{e.g}} Theorems 1-2 in the book of Runst and Sickel \cite{RunSi96}).

Therefore,  one only needs to show $\partial_{\mathcal{W}_0}^\ell \widetilde{\theta}^*_i \in {C^{-1,\gamma}}(\R^3)$, $i=1,2$.
Using the product estimate \eqref{eq:prod-es} as many times as needed, we deduce that for $i=1,2$ and for all $\ell\leq k-1$,
\begin{eqnarray}\label{ext.thet}
  \|\partial_{\mathcal{W}_0}^\ell \widetilde{\theta}^*_{i} \|_{C^{-1,\gamma}}
  &\lesssim& \|\mathcal{W}_0\|_{L^\infty} \|\nabla \partial_{\mathcal{W}_0}^{\ell-1}
  \widetilde{\theta}^*_{i} \|_{C^{-1,\gamma}} \nonumber \\
  & \lesssim& \|\nabla \partial_{\mathcal{W}_0}^{\ell-2}
  \widetilde{\theta}^*_{i} \|_{C^{-1,\gamma}}
  + \|\nabla^2 \partial_{\mathcal{W}_0}^{\ell-2}  \widetilde{\theta}^*_{i} \|_{C^{-1,\gamma}} \nonumber \\
  & \lesssim& \| \nabla \widetilde{\theta}^*_{i} \|_{C^{-1,\gamma}}  + \cdots
  + \|\nabla^{\ell}  \widetilde{\theta}^*_{i} \|_{C^{-1,\gamma}} \nonumber \\
  & \lesssim& \| \widetilde{\theta}^*_{i} \|_{C^{\ell-1,\gamma}} <\infty ,
\end{eqnarray}
where the constants in the above depend only on ${\|\mathcal{W}_0\|_{W^{1,\infty}}}$ and $\|\mathcal{W}_0\|_{W^{\ell-1,\infty}}$, which is the desired result.

The proof of \eqref{eq:str-reg2} is analogous. Indeed,  by using Rychkov's extension theorem, we have
$\theta_0 = \widetilde{\theta}^*_1 \, \mathds{1}_{D_0} + \widetilde{\theta}^*_2\, \mathds{1}_{D_0^c}$
with $\widetilde{\theta}^*_i \in C^\mu(\R^3)$, $i=1,2$. Since we have the control given by \eqref{eq:prod-es} and we have $\mathcal{W}_0\in L^\infty(\R^3)$, then one immediately gets that
\begin{align*}
  \|\partial_{\mathcal{W}_0} \theta_0\|_{C^{-1,\mu}} \lesssim \|\partial_{\mathcal{W}_0} \widetilde{\theta}^*_1\|_{C^{-1,\mu}}
  + \|\partial_{\mathcal{W}_0} \widetilde{\theta}^*_2\|_{C^{-1,\mu}} \lesssim \|\mathcal{W}_0\|_{L^\infty}
  \|(\widetilde{\theta}^*_1,\widetilde{\theta}^*_2)\|_{C^\mu} <\infty.
\end{align*}

\end{proof}

\subsection{Besov spaces}
Let $\chi$ and $\varphi\in C_c^\infty(\mathbb{R}^d)$  be two nonnegative radial functions which are supported respectively in the ball
$\{\xi\in \mathbb{R}^d:|\xi|\leq \frac{4}{3} \}$ and the annulus $\{\xi\in \mathbb{R}^d: \frac{3}{4}\leq |\xi|\leq  \frac{8}{3} \}$. We assume that they satisfy the following identity
\begin{align*}
  \chi(\xi)+\sum\limits_{j\ge0}\varphi(2^{-j}\xi)=1,\quad \textrm{for all}\;\;\xi\in\R^d.
\end{align*}
We define the frequency localization operator $\Delta_j$ and the low-frequency cut-off operator $ S_j$  as follows
\begin{align}\label{eq:Del-Sj}
  \Delta_j f = \varphi(2^{-j}D)f=2^{jd} h(2^j\cdot)*f,\quad  S_j f= \chi(2^{-j}D)f =2^{jd} \widetilde h(2^j\cdot)*f,
\end{align}
for all $j\geq -1$. We also use the homogenous frequency localization operator $\dot\Delta_j$ and
\begin{align}
  \dot{\Delta}_j f=\varphi(2^{-j}D)f=2^{jd}\int_{\R^d}h(2^jy)f(x-y)\,\mathrm{d}y,
\end{align}
where $h=\mathcal{F}^{-1}{\varphi}$, $\widetilde {h}=\mathcal{F}^{-1}{\chi}$ and $\mathcal{F}^{-1}$ is the inverse Fourier transform. \\

For all $f,g\in\mathcal{S}'(\R^d)$, we have the following Bony's decomposition
\begin{equation*}
  f\,g = T_f g + T_g f + R(f,g),
\end{equation*}
with 
\begin{equation}
  T_f g:= \sum_{q\in \N} S_{q-1}f \Delta_q g,\quad R(f,g)=\sum_{q\geq -1}\Delta_q f \widetilde{\Delta}_q g,\quad \widetilde{\Delta}_q := \Delta_{q-1} + \Delta_q + \Delta_{q+1}.
\end{equation}

\begin{definition}\label{def:Besov-str}
Let $s\in \R$, $p,r\in[1,\infty]$, $T>0$, $\mathcal{S}'(\R^d)$ be the space of tempered distributions. %and $\mathcal{P}(\R^d)$ lens des polynomes.
Let $\mathcal{W} = \{W^i\}_{1\leq i\leq N}$ be a family of regular vector fields $W^i:\R^d \rightarrow \R^d$.
We define $B^{s}_{p,r}(\R^d)$ (or $B^{s}_{p,r}$ in short)  to be the set of all $f\in \mathcal{S}'(\R^d)$ satisfying
\begin{align*}
  \|f\|_{B^{s}_{p,r}} :=  \big\| \big\{2^{qs}  \|\Delta_q f\|_{L^p} \big\}_{q\geq -1} \big\|_{\ell^r}  < \infty,
\end{align*}
and the space-time Chemin-Lerner's  space $\widetilde{L}^{\rho}_T( B^{s}_{p,r})$ as the set of tempered distributions $f$ such that
 \begin{align*}
  \|f\|_{\widetilde{L}^{\rho}_T( B^{s}_{p,r})} := \big\| \big\{2^{qs} \|\Delta_q f\|_{L^\rho_T(L^p)} \big\}_{q\geq -1} \big\|_{\ell^r}  < \infty .
\end{align*}
For all  $\ell\in\N$, we denote by $\bb^{s,\ell}_{p,r,\mathcal{W}}$ the set of all $f\in B^{s}_{p,r}$ such that
\begin{align}\label{norm:BBsln2}
  \|f\|_{\bb^{s,\ell}_{p,r,\mathcal{W}}} := \sum_{\lambda=0}^\ell \|\partial_{\mathcal{W}}^\lambda f\|_{B^{s}_{p,r}}
  = \sum_{\lambda=0}^\ell \sum_{\lambda_1+\cdots+\lambda_N=\lambda} \|\partial_{W_1}^{\lambda_1}\cdots \partial_{W_N}^{\lambda_N}f \|_{B^{s }_{p,r}} < \infty ;
\end{align}
and we define by $\BB^{s,\ell}_{p,r,\mathcal{W}}$ the set of  $f\in B^{s}_{p,r}$ such that
\begin{align}\label{norm:BBsln}
  \|f\|_{\BB^{s,\ell}_{p,r, \mathcal{W}}} := \sum_{\lambda=0}^\ell \|(T_{\mathcal{W}\cdot\nabla})^\lambda f\|_{B^{s}_{p,r}}
  = \sum_{\lambda=0}^\ell \sum_{\lambda_1+\cdots +\lambda_N =\lambda} \|(T_{W_1\cdot\nabla})^{\lambda_1}\cdots(T_{W_N\cdot\nabla})^{\lambda_N} f\|_{B^{s }_{p,r}} < \infty .
\end{align}
In particular, when $p=\infty$, we shall use the following short notations
\begin{equation}\label{eq:abbr}
\begin{split}
  \bb^{s,\ell}_{r,\mathcal{W}} := \bb^{s,\ell }_{\infty,r,\mathcal{W}},\quad \BB^{s,\ell}_{r,\mathcal{W}} := \BB^{s,\ell}_{\infty,r,\mathcal{W}},\\
  \bb^{s,\ell}_{\mathcal{W}} := \bb^{s,\ell}_{\infty,1,\mathcal{W}}, \quad \BB^{s,\ell}_{\mathcal{W}} := \BB^{s,\ell}_{\infty,1,\mathcal{W}},
\end{split}
\end{equation}
where we used the following notations  $\partial_{W^i} f= W^i\cdot \nabla f$ ($i=1,\cdots,N$) 
and $\partial_{\mathcal{W}}^\lambda f
  =  \big\{\partial_{W^1}^{\lambda_1}\cdots \partial_{W^N}^{\lambda_N}f: \lambda_1+\cdots+\lambda_N=\lambda\big\}$.
\end{definition}

We shall use the following basic properties of ${\bb^{s,\ell}_{p,r,\WW}}$.
\begin{lemma}\label{lemp:Bes-prop}
  Let $s,\widetilde{s}\in \R$, $\ell,\widetilde{\ell}$, $r, \widetilde{r}\in\N_+$ and $p\in[1,\infty]$.
The space $\bb^{s,\ell}_{p,r,\WW}$ satisfies that
\begin{equation*}
\begin{split}
  & \bb^{s,\ell}_{p,r,\WW} \subset \bb^{\widetilde{s},\ell }_{p,r,\WW}\,\,\, \mathrm{for}\,\,\, s\geq\widetilde{s},\;\;\;\;\;
  \bb^{s,\ell}_{p,r,\WW} \subset \bb^{s,\widetilde{\ell} }_{p,r,\WW}\,\,\, \mathrm{for}\,\,\, \ell \geq \widetilde{\ell},\;\;\;\;\;
  \bb^{s,\ell }_{p,r,\WW} \subset \bb^{s,\ell }_{p,\widetilde{r},\WW}\,\,\, \mathrm{for}\,\,\, r \leq \widetilde{r},\\
  &\quad \|f\|_{\bb^{s,\ell+1}_{p,r,\WW} }= \|\partial_\WW^{\ell+1} f\|_{B^{s}_{p,r}} + \|f\|_{\bb^{s,\ell}_{p,r,\WW}}\quad \mathrm{and}\quad
  \|f\|_{\bb^{s,\ell+1}_{p,r,\WW} } = \|\partial_\WW f\|_{\bb^{s,\ell}_{p,r,\WW}} + \|f\|_{B^{s }_{p,r}} .
\end{split}
\end{equation*}
\end{lemma}

We shall also need to use some product and commutator estimates in $\bb^{s,\ell}_{p,r,\mathcal{W}}$, which are stated below (for the proof see {\it{e.g.}} \cite{ChMX21}).
\begin{lemma}\label{lem:prod-es2}
  Let $k\in \N$, $\sigma\in(0,1)$ and let $\mathcal{W}= \{W^i\}_{1\leq i\leq N}$ be a family of regular divergence-free vector fields on $\R^d$ satisfying that
\begin{align}\label{norm:W}
  \|\mathcal{W}\|_{\bb^{1+\sigma,k-1}_{\infty,\mathcal{W}}}:= \sum_{\ell=0}^{k-1} \|\partial_\mathcal{W}^\ell \mathcal{W}\|_{B^{1+\sigma}_{\infty,\infty}}
  = \sum_{\ell=0}^{k-1} \sum_{\ell_1 +\cdots + \ell_N =\ell} \|\partial_{W_1}^{\ell_1} \cdots \partial_{W_N}^{\ell_N} \mathcal{W}\|_{B^{1+\sigma}_{\infty,\infty}} < \infty.
\end{align}
Let $m(D)$ be a zero-order pseudo-differential operator defined as a Fourier multiplier with symbol $m(\xi)\in C^\infty(\R^d\backslash\{0\})$.  Assume that $v$ is a smooth divergence-free vector field of $\R^d$ and $f:\R^d\rightarrow \R$ is a smooth function.
Then, for all $\epsilon\in(0,1)$ and all $p,r\in[1,\infty]$, there exists a constant $C>0$ depending only on $d,k,\epsilon$ and $ \|\mathcal{W}\|_{\bb^{1+\sigma,k-1}_{\infty,\mathcal{W}}}$
such that the following estimates hold:
\begin{align}\label{eq:prod-es2-1}
  \|v\cdot \nabla f\|_{\bb^{-\epsilon,k}_{p,r,\mathcal{W}} } \leq C\min \bigg\{ \sum_{\mu=0}^k \|v\|_{\bb^{0,\mu}_\mathcal{W}} \|\nabla f\|_{\bb^{-\epsilon,k-\mu}_{p,r,\mathcal{W}} },
  \sum_{\mu=0}^k \|v\|_{\bb^{-\epsilon,\mu}_{p,r,\mathcal{W}} } \|\nabla f\|_{\bb^{0,k-\mu}_\mathcal{W}} \bigg\},
\end{align}

\begin{align}\label{eq:prod-es2-2}
  \|m(D)f\|_{\bb^{-\epsilon, k+1}_{p,r,\mathcal{W}}} \leq C  \|f\|_{\bb^{-\epsilon, k+1}_{p,r,\mathcal{W}}}+ C\big(1+\|\mathcal{W}\|_{\bb^{1,k}_\mathcal{W}}\big)
  \big( \|f\|_{\bb^{-\epsilon, k}_{p,r,\mathcal{W}}}+  \|\Delta_{-1}m(D)f\|_{L^p}\big),
\end{align}
and
\begin{align}\label{eq:prod-es4-1}
  \|[m(D),v\cdot \nabla] f\|_{\bb^{-\epsilon,k}_{p,r,\mathcal{W}} } \leq C \Big( \|\nabla v\|_{\bb^{0,k}_\mathcal{W}} + \|v\|_{L^\infty}\Big)\| f\|_{\bb^{-\epsilon,k}_{p,r,\mathcal{W}}}.
\end{align}
\end{lemma}

The following result deals with the relation between the norms of $\bb^{s,\ell}_{p,r,\WW}$ and $\BB^{s,\ell}_{p,r,\WW}$ (see {\it e.g.} Lemmas 5.1, 5.2 of \cite{ChMX21}).
\begin{lemma}\label{lem:Bes-equ}
Under the assumptions of Lemma \ref{lem:prod-es2}, we have that there exists a constant $C>0$ depending only on $d,k,s$ and $ \|\mathcal{W}\|_{\bb^{1+\sigma,k-1}_{\infty,\mathcal{W}}}$ such that, for any $\ell\leq k$,
\begin{align}\label{es.bb.TNa}
  \|\nabla f\|_{\widetilde{\bb}^{s,\ell}_{p,r,\mathcal{W}} } \leq C \| f\|_{\widetilde{\bb}^{s+1,\ell}_{p,r,\mathcal{W}} },\qquad \mathrm{for}\;\,s>-1,
\end{align}
\begin{align}\label{es.bb.TNa.eqv}
  C^{-1} \| f\|_{{\bb}^{s,\ell}_{p,r,\mathcal{W}} } \leq  \| f\|_{\widetilde{\bb}^{s,\ell}_{p,r,\mathcal{W}} }\leq C \| f\|_{{\bb}^{s,\ell}_{p,r,\mathcal{W}} }\qquad\mathrm{for}\,\,s\in (-1,1),
\end{align}
\begin{align}\label{es.bb1.TNa.eqv}
  C^{-1}\| f\|_{{\bb}^{1,\ell}_{\mathcal{W}} } \leq \| f\|_{\widetilde{\bb}^{1,\ell}_{\mathcal{W}} }\leq C \| f\|_{{\bb}^{1,\ell}_{\mathcal{W}} },
\end{align}
\begin{align}\label{es.bb.TNafW}
  \|f\|_{\widetilde{\bb}^{s,\ell}_{p,r,\mathcal{W}} } \leq C \| f\|_{{\bb}^{s,\ell}_{p,r,\mathcal{W}} }+C\| f\|_{{\bb}^{1,\ell}_{\mathcal{W}} }\| W\|_{{\bb}^{s,\ell}_{p,r,\mathcal{W}} }\qquad\mathrm{for}\,\,s\ge 1,
\end{align}
\begin{align}\label{es.bb.TNaW}
  \|\mathcal{W}\|_{\widetilde{\bb}^{s,\ell}_{p,r,\mathcal{W}} } \leq C \| \mathcal{W}\|_{{\bb}^{s,\ell}_{p,r,\mathcal{W}} },\qquad \mathrm{for}\;\,s>-1.
\end{align}
\end{lemma}

Note that for the particular case $k=0,1$, the dependence of $\WW$ in the constant $C$ in the Lemma \ref{lem:prod-es2}
can be computed explicitly (see {\it e.g. \cite{ChMX21}}).
\begin{lemma}\label{lem:str-es1}
Let $v$ be a smooth divergence-free vector field of $\R^d$, and let $f:\R^d\rightarrow \R$ be a smooth function.
Then for any  $\epsilon\in (0,1)$ and $(p,r)\in [1,\infty]^2$, the following statements hold true for an absolute constant $C$.
\begin{enumerate}[(1)]
\item We have
\begin{equation}\label{eq:prod-es}
  \|v\cdot \nabla f\|_{B^{-\epsilon}_{p,r}} \leq C \min\Big\{\|v\|_{B^{-\epsilon}_{p,r}} \|\nabla f\|_{L^\infty}, \|v\|_{L^\infty} \|\nabla f\|_{B^{-\epsilon}_{p,r}} \Big\}.
\end{equation}
\item  We have
\begin{align}\label{eq:prod-es4}
  \|\partial_\mathcal{W} m(D)f\|_{B^{-\epsilon}_{p,r}} \leq& C  \|\mathcal{W}\|_{W^{1,\infty}} \|f\|_{B^{-\epsilon}_{p,r}} +C \|\partial_\mathcal{W} f\|_{B^{-\epsilon}_{p,r}} ,
\end{align}
and
\begin{align}\label{eq:prod-es5}
  \|[m(D),v\cdot \nabla] f\|_{B^{-\epsilon}_{p,r} } \leq& C \|v\|_{W^{1,\infty}}\| f\|_{B^{-\epsilon}_{p,r} }.
\end{align}
\end{enumerate}
\end{lemma}

If the divergence-free vector fields $W^i$ in $\mathcal{W}$ only belongs to $C^\gamma(\R^d)$ with $0<\gamma <1$,
then we have the following estimate \eqref{eq:prod-es4} whose proof is placed in the appendix section.
\begin{lemma}\label{lem:parW-mDf}
  Let $\mathcal{W}=\{W^i\}_{1\leq i\leq N}$ be a family  of divergence-free vector fields of $\R^d$
and $f:\R^d\rightarrow \R$ be a smooth function.
Then, for all $(p,r)\in [1,\infty]^2$, $0<\epsilon,\gamma <1$, there exists a nonnegative constant $C=C(\epsilon,\gamma,d)$ such that
\begin{align}\label{eq:par-W-mDf}
  \|\partial_\mathcal{W} m(D) f\|_{B^{-\epsilon}_{p,r}} \leq C \|\mathcal{W}\|_{C^\gamma}  \|f\|_{B^{1-\gamma-\epsilon}_{p,r}}
  + C\|\partial_\mathcal{W} f\|_{B^{-\epsilon}_{p,r}} .
\end{align}
\end{lemma}

\subsection{Useful lemmas}

We now state a useful commutator estimate.
\begin{lemma}\label{lem:comm-es}
Let $m(D)$ be a  pseudo differential operator defined as a Fourier multiplier with symbol $m(\xi)\in C^\infty(\R^d\backslash\{0\})$.
Let $p\in [2,\infty]$, $\widetilde{\mathcal{R}}_{-1}:= m(D) \Lambda^{-1}$ with $\Lambda=({-\Delta})^{1/2}$.
Assume that $v$ is a smooth divergence-free vector field of $\R^d$, and $\theta$ is a smooth scalar function. Then, we have
\begin{equation}\label{eq:comm-es1}
  \|[\widetilde{\mathcal{R}}_{-1}, v\cdot\nabla] \theta\|_{B^1_{p,\infty}} \leq C \big( \|\nabla v\|_{L^p} \|\theta\|_{B^0_{\infty,\infty}}
  + \|v\|_{L^2} \|\theta\|_{L^2} \big),
\end{equation}
with $C>0$ a constant depending on $p$ and $d$.
\end{lemma}

Some basic facts on the particle-trajectory map are collected below (see {\it{e.g.}} Proposition 3.10 in \cite{BCD11}).
\begin{lemma}\label{lem:flow}
Assume that $v(x,t)$ is a velocity field belonging to $L^1([0,T], \dot W^{1,\infty}(\R^d))$. Let $\psi_t(x)$ be the the particle-trajectory generated by velocity $v$ which solves that
\begin{equation}\label{eq:flow1}
    \frac{\partial \psi_t(x)}{\partial t} = v ( \psi_t(x),t),\quad \psi_t(x)|_{t=0}=x,
\end{equation}
that is
\begin{equation}\label{eq:flow2}
  \psi_t(x) = x + \int_0^t v(\psi_\tau(x),\tau) \dd \tau.
\end{equation}
Then the system \eqref{eq:flow1} has a unique solution $\psi_t(\cdot):\R^d\mapsto \R^d$ on $[0,T]$ which is a volume-preserving bi-Lipschitzian homeomorphism and satisfies that
$\nabla \psi_t$ and its inverse $\nabla \psi^{-1}_t$
belong to $L^\infty([0,T]\times \R^d)$ with
\begin{equation}\label{DXest}
  \|\nabla \psi^{\pm1}_t\|_{L^\infty(\R^d)} \leq e^{\int_0^t \|\nabla v(\tau)\|_{L^\infty}\dd \tau}.
\end{equation}
In addition, the following statements hold true.
\begin{enumerate}[(1)]
\item If  $v\in L^1([0,T], C^{1,\gamma}(\R^d))$, then $\psi^{\pm1}_t \in L^\infty([0,T], C^{1,\gamma}(\R^d))$ with
\begin{equation}\label{X-C1gam-es}
  \|\nabla\psi^{\pm 1}_t\|_{C^\gamma} \lesssim  e^{(2+\gamma)\int_0^t \|\nabla v\|_{L^\infty} \dd \tau} \Big( 1 + \int_0^t \|\nabla v(\tau)\|_{C^\gamma} \dd \tau \Big).
\end{equation}
\item If  $v\in L^1([0,T], W^{2,\infty}(\R^d))$, then $\psi^{\pm1}_t \in L^\infty([0,T],  W^{2,\infty}(\R^d))$ with
\begin{equation}\label{X-W2inf-es}
  \|\nabla^2 \psi^{\pm 1}_t\|_{L^\infty} \lesssim  e^{3\int_0^t \|\nabla v\|_{L^\infty} \dd \tau} \int_0^t \|\nabla^2 v(\tau)\|_{L^\infty} \dd \tau.
\end{equation}
\end{enumerate}
\end{lemma}

We have the following  estimates for  the transport and transport-diffusion equations
(one can see \cite{BCD11} for the proof of \eqref{es.sm1}-\eqref{eq:T-sm3}
and see \cite{ChMX21}, for the proof of \eqref{TD-sm-es}).
\begin{lemma}\label{lem:TD-sm2}
  Assume $(\rho, r,p)\in [1,\infty]^3$  and $-1< s < 1$. Let $u$ be a smooth divergence-free vector field and $\phi$ be a smooth function solving the following transport equation
\begin{equation}\label{eq:TD-eq2}
  \partial_t \phi + u\cdot\nabla \phi - \nu \Delta \phi = f,\quad \phi|_{t=0}(x)=\phi_0(x),\quad x\in \R^d.
\end{equation}
The following statements hold.
\begin{enumerate}[(1)]
\item If $\nu>0$, then there exists a constant $C$ which depends on $d$ and $s$ such that for any $t>0$,
\begin{equation}\label{es.sm1}
  \nu^{\frac{1}{\rho}} \|\phi\|_{\widetilde{L}^\rho_t (B^{s+ \frac{2}{\rho}}_{p,r})} \leq C (1+\nu t)^{\frac{1}{\rho}}
  \big( \|\phi_0\|_{B^{s}_{p,r}} + \|f\|_{\widetilde{L}^1_t (B^{s}_{p,r})}  + \int_0^t \|\nabla u(\tau)\|_{L^\infty} \|\phi(\tau)\|_{B^{s}_{p,r}} \dd \tau \big).
\end{equation}

\item
If $\nu=0$, then there exists a constant $C$ which depends on $d$ and $s$ such that for any $t>0$,
\begin{equation}\label{eq:T-sm2}
  \|\phi\|_{L^\infty_t (B^{s}_{p,r})} \leq C \big( \|\phi_0\|_{B^{s}_{p,r}} + \|f\|_{L^1_t B^{s}_{p,r}} + \int_0^t \|\nabla u(\tau)\|_{L^\infty} \|\phi(\tau)\|_{ B^{s}_{p,r}} \dd \tau \big),
\end{equation}
and
\begin{equation}\label{eq:T-sm3}
  \|\phi\|_{L^\infty_t (B^{s}_{p,r})} \leq C e^{C \int_0^t \|\nabla u(\tau)\|_{L^\infty} \dd \tau} \Big( \|\phi_0\|_{B^{s}_{p,r}} + \|f\|_{L^1_t (B^{s}_{p,r})} \Big).
\end{equation}
\item If $\nu>0$, then there exists a constant $C$ which depends on $d$ and $s$ such that for any $t>0$,
\begin{equation}\label{TD-sm-es}
  \sup_{q\in \N} 2^{\frac{2 q}{\rho}} \|\Delta_q \phi\|_{L^\rho_t (L^p)} \leq C \Big(\sup_{q\in \N}\|\Delta_q \phi_0\|_{L^p} + \int_0^t\|\nabla u\|_{L^p} \|\phi\|_{B^0_{\infty,\infty}} \dd \tau + \|f\|_{L^1_t (L^p)} \Big) .
\end{equation}

\end{enumerate}
\end{lemma}

We shall need to use the following smoothing estimates of nonhomogeneous heat equation (see {\it{e.g.}} \cite{BCD11} for the proof).

\begin{lemma}\label{lem:heat-Besov}
  Let $T>0$, $s\in \R$ and $1\leq \rho_1, p,r \leq \infty$. Assume that $v_0 \in \dot B^s_{p,r}(\R^d)$
and $f\in \widetilde{L}^{\rho_1}_T(\dot B^{s-2 + \frac{2}{\rho_1}}_{p,r})$.
Then the nonhomogeneous heat equation
\begin{equation}\label{eq:heat}
  \partial_t v - \Delta v =f,\qquad v|_{t=0} =v_0,
\end{equation}
has a unique solution $v$ in
$\widetilde{L}^{\rho_1}_T(\dot B^{s+\frac{2}{\rho_1}}_{p,r})\cap \widetilde{L}^\infty_T(\dot B^s_{p,r})$ and there exists a constant $C=C(d)>0$
such that for all $\rho \in [\rho_1,\infty]$,
\begin{align}\label{eq:heat-Besov}
  \|v\|_{\widetilde{L}^\rho_T(\dot B^{s+ \frac{2}{\rho}}_{p,r})}
  \leq C \Big(\|v_0\|_{\dot B^s_{p,r}} + \|f\|_{\widetilde{L}^{\rho_1}_T(\dot B^{s-2+\frac{2}{\rho_1}}_{p,r})}\Big).
\end{align}
In particular, if $v_0\in \dot H^s$ and $f \in L^2_T (\dot H^{s-1})$, we have
\begin{align}\label{eq:e-tDel-es}
  \|v\|_{L^\rho_T(\dot H^{s+ \frac{2}{\rho}})} & \leq \|v_0\|_{\dot H^s} + \|f\|_{L^2_T(\dot H^{s-1})},\quad \forall \rho\in [2,\infty].
\end{align}
\end{lemma}

Finally, we recall a product estimate used in the uniqueness part which may be found in \cite{DanPa08}.
\begin{lemma}\label{lem:ProdEs-uni}
  Let $(\alpha_q)_{q\in \mathbb{Z}}$ be a sequence of nonnegative functions over $[0,T]$, and let $s_1,s_2,p$ satisfy
\begin{align}\label{eq:s1s2-cd}
  1\leq p\leq \infty,\quad \frac{d}{p} + 1 >s_1,\quad \frac{d}{p} > s_2, \quad \textrm{and}\quad s_1+ s_2 > d \max\Big\{0,\frac{2}{p}-1 \Big\}.
\end{align}
Assume that for all $q'\geq q$ and $t\in [0,T]$, we have $0\leq \alpha_{q'}(t) -\alpha_q(t) \leq \frac{1}{2} \big(s_1+ s_2 + N \min\{0,1-\frac{2}{p}\}\big)(q'-q)$. Then for all $r\in [1,\infty]$, there exists a constant $C$ depending only on $s_1,s_2,N$ and
$p$ such that for all functions $b$ and solenoidal vector field $a$ over $\R^d$, the following estimate holds for all $t\in [0,T]$:
\begin{align*}
  \sup_{q\in \mathbb{Z}} \int_0^t 2^{q(s_1 + s_2 -1 -\frac{N}{p}) -\alpha_q(\tau)} \|\dot\Delta_q \mathrm{divg}(ab)\|_{L^p}\dd \tau
  \leq C \|b\|_{\widetilde{L}^r_t(\dot B^{s_1}_{p,\infty})} \sup_{q\in \mathbb{Z}}
  \big\|2^{q s_2 -\alpha_q} \|\dot{\Delta}_q a\|_{L^p}\big\|_{L^{r'}_t}.
\end{align*}
\end{lemma}

\section{Existence and uniqueness result of 3D Boussinesq system \eqref{eq:BousEq} with $v_0\in H^{1/2}$}\label{sec:exi-uni}

The proof is split into three classical steps. \\

$\bullet$ Step 1: local existence of solution. \\

First we prove the \textit{a priori} estimates.
From the transport equation $\eqref{eq:BousEq}$, we have that for all $s\in (3,\infty]$, the $L^1\cap L^s$ norm are preserved along the evolution, that is
\begin{equation}\label{es0.thet}
  \Vert \theta(t) \Vert_{L^1\cap L^s } \leq \Vert\theta_0\Vert_{L^1\cap L^s},\quad \forall t\geq 0.
\end{equation}
Now we focus on the estimates for the velocity $v$. The basic $L^2$-energy estimate for the second  equation in $\eqref{eq:BousEq}$ gives
\begin{align*}
  \frac{1}{2}\frac{\dd }{\dd t} \|v\|_{L^2}^2 + \|\nabla v\|_{L^2}^2  \leq \Big|\int_{\R^3} v^3\, \theta \dd x \Big| \leq \|v\|_{L^2} \|\theta\|_{L^2} \leq \|v\|_{L^2} \|\theta_0\|_{L^2},
\end{align*}
which implies that $\|v\|_{L^\infty_T(L^2)} \leq \|v_0\|_{L^2} + T \|\theta_0\|_{L^2}$ and
\begin{equation}\label{es.v.L2}
  \|v\|_{L^\infty_T(L^2)}^2 + \|\nabla v\|_{L^2_T(L^2)}^2 \leq 4 (1+T^2) (\|v_0\|_{L^2}^2 + \|\theta_0\|_{L^2}^2).
\end{equation}

Set $v= e^{t\Delta}v_0 + w$ where $e^{t\Delta}$ stand for the heat semigroup, then one has that
\begin{align}\label{eq:w}
  \partial_t w - \Delta w = - \mathbb{P}\big((w + e^{t\Delta} v_0)\cdot\nabla( w + e^{t\Delta} v_0 )\big) + \mathbb{P}(\theta e_3),\quad w|_{t=0} =0,
\end{align}
where $\mathbb{P} = \Id - \nabla \Delta^{-1} \divg$ is the Leray projection operator.
By multiplying both sides of equation \eqref{eq:w} by $\Lambda w$ and integrating in the space variable, one finds

 (as in \cite{GGJ20} or \eqref{eq:v-H-1/2} below)
\begin{align*}
  \frac{1}{2} \frac{\dd}{\dd t} \|w\|_{\dot H^{\frac{1}{2}}}^2 + \|w\|_{\dot H^{\frac{3}{2}}}^2
  & \leq \Big|\int_{\R^3} \Lambda w \cdot \mathbb{P}\big((w + e^{t\Delta} v_0)\cdot\nabla( w + e^{t\Delta} v_0 )\big) \dd x \Big|
  + \Big| \int_{\R^3} \Lambda w \cdot \mathbb{P}(\theta e_3) \dd x \Big| \\
  & \leq \|w\|_{\dot H^{\frac{3}{2}}} \Big(\|\big((w + e^{t\Delta} v_0)\cdot\nabla( w + e^{t\Delta} v_0 )\|_{\dot H^{-\frac{1}{2}}}
  + \|\theta\|_{\dot H^{-\frac{1}{2}}} \Big) \\
  & \leq \|w\|_{\dot H^{\frac{3}{2}}} \Big( \|w\|_{\dot H^1}^2 + \|e^{t\Delta} v_0\|_{\dot H^1}^2 + \|\theta\|_{L^{\frac{3}{2}}} \Big),
\end{align*}
which, together with the use of interpolation and Cauchy-Schwartz inequality, allows us  to find
\begin{eqnarray*}
  \frac{\dd}{\dd t} \|w\|_{\dot H^{\frac{1}{2}}}^2 + \|w\|_{\dot H^{\frac{3}{2}}}^2
  &\leq& C_1 \|w\|_{\dot H^{\frac{1}{2}}}^2 \|w\|_{\dot H^{\frac{3}{2}}}^2 + C_1 \|e^{t\Delta}v_0\|_{\dot H^1}^4
  + C_2 \|\theta_0\|_{L^{\frac{3}{2}}}^2 ,
\end{eqnarray*}
where $C_1,C_2>0$.
Therefore, as long as
\begin{equation}\label{eq:v0-the0-cd1}
  C_1 \int_0^T\|e^{t\Delta} v_0\|_{\dot H^1}^4 \dd t + C_2 \|\theta_0\|_{L^{\frac{3}{2}}}^2 T \leq \frac{1}{4C_1},
\end{equation}
we may use the continuity method to get that
\begin{eqnarray*}
  \|w\|_{L^\infty_T(\dot H^{\frac{1}{2}})}^2 + \|w\|_{L^2_T(\dot H^{\frac{3}{2}})}^2 \leq \frac{1}{2C_1}.
\end{eqnarray*}
Then, the classical estimate for the heat operator $e^{t\Delta} v_0$ in \eqref{eq:e-tDel-es} allows us to write that
\begin{equation}\label{eq:H1/2-T-1}
  \|v\|_{L^\infty_T(\dot H^{\frac{1}{2}})}^2 + \|v\|_{L^2_T(\dot H^{\frac{3}{2}})}^2 \leq \frac{1}{C_1} + 2\|v_0\|_{\dot H^{\frac{1}{2}}}^2.
\end{equation}
By noticing that $\|e^{t\Delta} v_0\|_{L^4_T(\dot H^1)} \leq \|v_0\|_{\dot H^{\frac{1}{2}}}$ (which is a consequence of \eqref{eq:e-tDel-es}) we see that there exists a positive real number $M$, which depends on $v_0$, such that
\begin{equation}\label{eq:v0h-es1}
  \| e^{t\Delta}v_0^\textrm{h} \|_{L^4_T(\dot H^1)}^4 \leq \|v_0^\textrm{h}\|_{\dot H^{\frac{1}{2}}}^4 \leq \frac{1}{64 C_1^2},\quad \textrm{where}\;\; v_0^\textrm{h}= \mathcal{F}^{-1}\big(\mathds{1}_{\{|\xi|>M\}} \widehat{v_0}(\xi)\big).
\end{equation}
Therefore, one finds
\begin{align*}
  C_1 \|e^{t\Delta} v_0\|_{L^4_T(\dot H^1)}^4 +  C_2 \|\theta_0\|_{L^{\frac{3}{2}}}^2 T &\leq 8 C_1 \|e^{t\Delta} v_0^\textrm{h}\|_{L^4_T(\dot H^1)}^4 + 8 C_1 \|e^{t\Delta} (v_0-v_0^\textrm{h})\|_{L^4_T(\dot H^1)}^4 +  C_2 \|\theta_0\|_{L^{\frac{3}{2}}}^2 T  \\
  &\leq \frac{1}{8C_1} + 8 C_1 M^2 \|e^{t\Delta} (v_0 - v_0^\textrm{h})\|_{L^4_T(\dot H^{\frac{1}{2}})}^4 +  C_2 \|\theta_0\|_{L^{\frac{3}{2}}}^2 T \\
  &\leq \frac{1}{8C_1} + \big(8 C_1 M^2\|v_0 \|_{\dot H^{\frac{1}{2}}}^4 +  C_2 \|\theta_0\|_{L^{\frac{3}{2}}}^2\big) T .
\end{align*}
Hence, as longs as
\begin{equation}\label{eq:T-cond}
  T\leq \frac{1}{8C_1 \big(8 C_1 M^2\|v_0 \|_{\dot H^{\frac{1}{2}}}^4 +  C_2 \|\theta_0\|_{L^{\frac{3}{2}}}^2\big)},
\end{equation}
we have local existence of a solution $(v,\theta)$ which satisfies the estimate \eqref{eq:H1/2-T-1}. \\

We now state the following result on the regularity estimates of $v$ whose proof is in the appendix.
\begin{lemma}\label{lem:L2Linf}
   Assume that $(v,\theta)$ is a smooth solution to the 3D Boussinesq system \eqref{eq:BousEq} on $[0,T]$ satisfying \eqref{es0.thet} with $s >3$. Assume that, 
   one has
\begin{align}\label{eq:v-H1/2es}
  \|v\|_{L^\infty_T(\dot H^{\frac{1}{2}})}^2 + \|v\|_{L^2_T(\dot H^{\frac{3}{2}})}^2 \leq C_0 E(T),
  \quad \textrm{where}\quad E(T):=  \|v_0\|_{\dot H^{\frac{1}{2}}}^2 +  \|\theta_0\|_{L^1\cap L^s}^2 T + 1 .
\end{align}
Then, there exists a positive constant $C$ such that 
\begin{align}\label{es.L2Linf}
  \|v\|_{L^2_T(L^\infty)} \leq C E(T)^{\frac{3}{2}} ,
\end{align}
and for all $q >3$,
\begin{align}\label{es.uni0}
  \|v\|_{L^\infty_T(\dot B^{-1+\frac3q}_{q,\infty})}+\|v\|_{\widetilde{L}^1_T(\dot B^{1+\frac3q}_{q,\infty})}\leq C E(T)^2.
\end{align}
\end{lemma}

The {\it{a priori}} estimates \eqref{es0.thet}, \eqref{eq:H1/2-T-1} and Lemma \ref{lem:L2Linf},
are enough  to show the existence ({\it{e.g.}} by using a standard approximation process).
first solve the Cauchy problem \eqref{eq:BousEq} with frequency localized initial data
$(\theta_{0,n},v_{0,n}) := (S_n \theta_0, S_n v_0)$, $n\in \N$,
where $S_n$ is defined by \eqref{eq:Del-Sj}. It is clear that $(\theta_{0,n},v_{0,n})$ belongs to $H^s(\R^3)$ for all $s$. Then, it follows from  \cite{DanPa08B}
that we obtain a local (unique) smooth solution $(\theta_n, v_n, \nabla p_n)$ to the system \eqref{eq:BousEq} associated with
$(\theta_{0,n}, v_{0,n})$. Moreover, the \textit{a priori} estimates below ensure that $(\theta_n,v_n)$ satisfies
\eqref{eq:v-the-es} on $[0,T]$ uniformly in $n$. Hence by using Rellich compactness theorem for instance (see \cite{PGLR1}), one can pass to the limit $n\rightarrow \infty$ (up to a subsequence) to show that there exist functions $(\theta,v,\nabla p)$ satisfying \eqref{eq:v-the-es}
which are solutions to the 3D Boussinesq system \eqref{eq:BousEq} in the sense of distribution. \\

$\bullet$ Step 2: local uniqueness. \\

Let $(\theta_1, v_1, \nabla p_1)$ and $(\theta_2, v_2, \nabla p_2)$ be two solutions of the 3D Boussinesq system \eqref{eq:BousEq} with the same initial data and satisfying \eqref{eq:v-the-es}.
We shall use the following uniqueness result.
\begin{lemma}\label{lem:uniq}
 Assume that for some $p\in [1,6)$ and $i=1,2$, we have
\begin{align}\label{eq:uniq-assum}
  \theta_i \in L^\infty_T (\dot B^{-1+\frac{3}{p}}_{p,\infty}),\quad v_i \in L^\infty_T(\dot B^{-1+\frac{3}{p}}_{p,\infty}) \cap
  \widetilde{L}^1_T(\dot B^{1+\frac{3}{p}}_{p,\infty}).
\end{align}
There exists a constant $c>0$ depending on $p$ such that if there is an $r\in (1,\infty]$ such that
\begin{align}\label{eq:uniq-cd}
  \|v_1\|_{\widetilde{L}^1_T(\dot B^{1+\frac{3}{p}}_{p,\infty})}
  + \|v_2\|_{\widetilde{L}^r_T(\dot B^{-1 + \frac{2}{r}+ \frac{3}{p}}_{p,\infty})} \leq c,
\end{align}
then $(\theta_1,v_1,\nabla p_1) \equiv (\theta_2, v_2,\nabla p_2)$ on $\R^3\times [0,T]$.
\end{lemma}

\begin{proof}[Proof of Lemma \ref{lem:uniq}]

When $r=\infty$, Lemma \ref{lem:uniq} is nothing but  Theorem 3 in \cite{DanPa08}.
It suffices to prove the lemma for the remaining values of $r$, namely $r\in (1,\infty)$. In this case, we may still extend the uniqueness result of Danchin and Pa\"icu \cite{DanPa08}.
Indeed, we may use the following useful estimate
\begin{align}\label{eq:uniq-modi}
  \sup_{q\in\Z} \int_0^t 2^{q(-2+ \frac{3}{p}-\eta)-\varepsilon_q(\tau)} \|\dot\Delta_q\divg (\delta v\otimes v_2)\|_{L^p} \dd \tau
  \leq C \|v_2\|_{\widetilde{L}^r_t(\dot B^{-1+\frac{2}{r} +\frac{3}{p}}_{p,\infty})}
  \sup_{q\in\Z} \Big\|2^{q(-\frac{2}{r}+\frac{3}{p} -\eta)-\varepsilon_q(\tau)} \|\dot\Delta_q \delta v\|_{L^p} \Big\|_{L^{r'}_t},
\end{align}
and the interpolation inequality
\begin{align}
  & \quad \sup_{q\in\Z} \Big\|2^{q(-\frac{2}{r}+\frac{3}{p} -\eta)-\varepsilon_q(\tau)} \|\dot\Delta_q \delta v\|_{L^p} \Big\|_{L^{r'}_t} \nonumber \\
  & \leq C \Big( \sup_{\tau\in [0,t],q\in\Z} 2^{-q(-2 + \frac{3}{p}-\eta)-\varepsilon_q(\tau)}
  \|\dot\Delta_q \delta v\|_{L^p} + \sup_{q\in\Z}\int_0^t 2^{q(\frac{3}{p}-\eta)-\varepsilon_q(\tau)}
  \|\dot\Delta_q\delta v\|_{L^p} \dd \tau\Big), 
\end{align}
where $\delta v =v_1 -v_2$, $0<\eta< -1 + 3\min\{\frac{2}{p},1\}$, $r'= \frac{r}{r-1}$ is the dual index of $r$ and
$$\varepsilon_q(t) := C \sum_{q_1\leq q} 2^{q_1(1+\frac{3}{p})} \|\big(\sum_{|k|\leq N_0} \dot\Delta_{q_1+k} v_1\big)\|_{L^1_t(L^p)}.$$
 To get the inequality \eqref{eq:uniq-modi}  it suffices to apply
Lemma \ref{lem:ProdEs-uni} with $s_1= -1 + \frac{2}{r} + \frac{3}{p}$, $s_2 =-\frac{2}{r}+ \frac{3}{p} -\eta$,
$s_1 + s_2 = -1 + \frac{6}{p} -\eta $ and such that they satisfy \eqref{eq:s1s2-cd}.

Using \eqref{eq:uniq-modi} and following the same strategy as the proof of Theorem 3 in \cite{DanPa08}, we can show the uniqueness result under the smallness condition \eqref{eq:uniq-cd}.

\end{proof}

Now, with Lemma \ref{lem:uniq} proved,
we see that \eqref{eq:v-the-es} implies the condition \eqref{eq:uniq-assum}.
Then, by letting $r=2$ and using the absolute continuity of the Lebesgue integral, we see that
there exists a small time $T_1>0$ depending only on $\|v_i\|_{\widetilde{L}^1_T(\dot B^{1+ 3/p}_{p,\infty})}$
and $\|v_i\|_{\widetilde{L}^2_T(\dot B^{3/p}_{p,\infty})}$ such that \eqref{eq:uniq-cd}
is satisfied for $r=2$ and $p\in (3,6)$. Therefore, Lemma \ref{lem:uniq} gives the uniqueness on $\R^3\times [0,T_1]$.
Repeating this process on $[T_1,2T_1]$, $[2T_1,3T_1]$...etc, and after a finite time we can conclude the uniqueness on the whole set $\R^3\times [0,T]$. \\

$\bullet$ Step 3: global existence of solution under the smallness condition \eqref{eq:data-cond}. \\

Using a result of Danchin and Pa\"icu \cite{DanPa08B} (see Theorem 1.4), we know that under the smallness condition \eqref{eq:data-cond}, there exists an absolute constant $C_0>0$ such that for any $T>0$,
\begin{equation}\label{eq:vL3inf-es}
  \|v\|_{L^\infty_T(L^{3,\infty})} \leq C \big( \|v_0\|_{L^{3,\infty}} + \|\theta_0\|_{L^1} \big) \leq C_0 c_* .
\end{equation}

Taking the inner product of both sides of the first equation in $\eqref{eq:BousEq}$ with $\Lambda v$, we get
\begin{align*}
  \frac{1}{2}\frac{\dd}{\dd t} \|v\|_{\dot H^{\frac12}}^2 + \| v\|_{\dot H^{\frac32}}^2
  & \leq \| v\|_{\dot H^{\frac32}} \|v\cdot\nabla v\|_{\dot H^{-\frac12}} +
  \| v\|_{\dot H^{\frac32}}\|\theta\|_{\dot H^{-\frac12}}.
\end{align*}
Then,  using classical paradifferential calculus together with the fact that  $L^{3,\infty}(\R^3)\hookrightarrow \dot B^{-1+ 3/r}_{r,\infty}(\R^3)$ for $r>3$, we have that
\begin{align*}
  \|v\cdot\nabla v\|_{\dot H^{-\frac{1}{2}}} \leq C \|T_v v\|_{\dot H^{\frac{1}{2}}} + C \|R(v, v)\|_{\dot H^{\frac{1}{2}}}
  \leq C  \| v\|_{\dot B^{-1}_{\infty,\infty}} \| v\|_{\dot H^{\frac{3}{2}}} \leq C \|v\|_{L^{3,\infty}} \|v\|_{\dot H^{\frac{3}{2}}},
\end{align*}
Therefore, by using \eqref{eq:vL3inf-es} one gets
\begin{align}\label{es.v.Hfr12}
  \frac{1}{2}\frac{\dd}{\dd t} \|v\|_{\dot H^{\frac12}}^2 + \| v\|_{\dot H^{\frac32}}^2
  &\leq C\| v\|_{\dot H^{\frac32}}^2 \| v\|_{L^{3,\infty}} + \| v\|_{\dot H^{\frac32}}\|\theta_0\|_{L^{\frac 32}}\nonumber\\
  &\leq \| v\|_{\dot H^{\frac32}}^2 \Big( 2CC_0 c_* + \frac{1}{4} \Big) + C \|\theta_0\|^2_{L^{\frac{3}{2}}}.
\end{align}
By choosing $c_*>0$ in \eqref{eq:data-cond} so that $2C C_0 c_* \leq \frac{1}{4},$

we obtain that 
\begin{align}\label{es.v.Hfr12b}
  \|v\|_{L^\infty_T(\dot H^{\frac12})}^2 + \| v\|_{L^2_T(\dot H^{\frac32})}^2 \leq \| v_0\|^2_{\dot H^{\frac12}} + C T\|\theta_0\|^2_{L^{\frac{3}{2}}}.
\end{align}

By using the blow-up criterion from Step 1, we easily conclude the global existence result.

\qed

\section{Propagation of the  $C^{1,\gamma},\, W^{2,\infty},\, C^{2,\gamma}$ regularity of the temperature fronts}\label{sec:C1gam}

The goal of this section is to prove the persistence of the regularity of the temperature front for the 3D Boussinesq equation \eqref{eq:BousEq}. We shall respectively show that the $C^{1,\gamma}$, $W^{2,\infty}$ and $C^{2,\gamma}$ of the temperature front is preserved along the evolution.
Before going any further, let us introduce a new quantity $\Gamma$ (in the spirit of the so-called  Alinhac's good unknown).

Let $\Omega=(\Omega^1,\Omega^2,\Omega^3) = \nabla\wedge v$ be the vorticity of the fluid, where the notation $\wedge $ stands for the wedge operation, that is,
\begin{equation*}
  \Omega=\nabla\wedge v=(\partial_2v^3-\partial_3v^2, \partial_3v^1-\partial_1v^3, \partial_1v^2-\partial_2v^1)^t.
\end{equation*}
Applying the operator $\nabla\wedge $ to the equation $\eqref{eq:BousEq}$ gives the vorticity equation:
\begin{equation}\label{eq.Omeg}
  \partial_t \Omega + v\cdot \nabla \Omega - \Delta \Omega =\Omega\cdot \nabla v+ (\partial_2 \theta, -\partial_1\theta, 0)^t.
\end{equation}
Note that
\begin{equation*}
  \partial_t \Omega + v\cdot \nabla \Omega - \Delta \big(\Omega - \Lambda^{-2} (\partial_2 \theta, -\partial_1\theta,0)^t \big) = \Omega\cdot\nabla v,
\end{equation*}
where $\Lambda = (-\Delta)^{1/2}$.

Let us set $\mathcal{R}_{-1,j}= \partial_j \Lambda^{-2} $, $j=1,2$,
and
\begin{align}\label{eq:op-R-1}
  \mathcal{R}_{-1} = (\mathcal{R}_{-1,2}, - \mathcal{R}_{-1,1},0)^t,
\end{align}
then $ \mathcal{R}_{-1} \theta = \Lambda^{-2} \nabla \wedge (\theta e_3) = \Lambda^{-2} (\partial_2 \theta, -\partial_1\theta,0)^t$ and the vector-valued quantity $\mathcal{R}_{-1}\theta$ satisfies
\begin{align*}
  \partial_t \mathcal{R}_{-1}\theta + v\cdot \nabla \mathcal{R}_{-1}\theta = - [\mathcal{R}_{-1},v\cdot\nabla]\theta,
\end{align*}
with
\begin{align}\label{def:R-1comm}
  [\mathcal{R}_{-1}, v\cdot\nabla]\theta : = \big([\mathcal{R}_{-1,2},v\cdot\nabla]\theta,  -[\mathcal{R}_{-1,1},v\cdot\nabla]\theta,0\big)^t .
\end{align}
By introducting a new unknown $\Gamma = (\Gamma^1,\Gamma^2,\Gamma^3)$ defined as
\begin{equation}\label{Gamma}
  \Gamma:= \Omega - \mathcal{R}_{-1}\theta, \end{equation}
we see  that $\Gamma$ verifies  
\begin{equation}\label{eq.Gamm}
  \partial_t \Gamma + v\cdot\nabla \Gamma -\Delta \Gamma =\Omega\cdot \nabla v + [\mathcal{R}_{-1},v\cdot\nabla]\theta. \\
\end{equation}

\subsection{Propagation of the $C^{1,\gamma}$ regularity of the temperature fronts} 
We start by proving the regularity of the velocity  $v$. More precisely, we are going to prove that
if $\theta_0\in L^1\cap L^\infty(\R^3)$ and $v_0\in H^1\cap W^{1,p}(\R^3)$, $p>2$, then for all $r\geq 1$, one has
\begin{equation}\label{eq:v-gam-es1}
  \|v\|_{L^\infty_T(H^1\cap W^{1,p})} + \|v\|_{\widetilde{L}^1_T (B^{\min\{2,3-\frac{3}{p}\}}_{\infty,\infty})} +  \|\Gamma\|_{\widetilde{L}^r_T(B^{\frac{2}{r}}_{p,\infty})}  \leq C e^{C E(T)^3} ,
\end{equation}
where $E(T)$ has been defined in \eqref{eq:v-H1/2es}.

Assume that the above control holds, then, for the temperature front data $$\theta_0 (x)= \theta^*_1(x) \mathds{1}_{D_0}(x) + \theta_2^*(x) \mathds{1}_{D_0^c}(x) \in L^1\cap L^\infty(\R^3)$$ and $v_0\in H^1\cap W^{1,3}(\R^3)$,
by using \eqref{X-C1gam-es} together with the fact that
\begin{align}\label{es:v-L1Cgam}
  v\in \widetilde{L}^1_T(B^2_{\infty,\infty}) \subset L^1_T(B^{1+\gamma}_{\infty,1}) \subset L^1_T(C^{1,\gamma}), \quad \gamma\in (0,1),
\end{align}
we would get that $\psi_t^{\pm 1}(x)\in L^\infty(0,T; C^{1,\gamma}(\R^3))$,
which clearly implies that
\begin{equation*}
  \partial D_0\in C^{1,\gamma}\quad \Longrightarrow \quad \partial D(t)= \psi_t(\partial D_0)\in L^\infty(0,T;C^{1,\gamma}).
\end{equation*}

Let us now prove \eqref{eq:v-gam-es1}. Assume $v_0\in H^1$, we first prove the control in $H^1$ of the velocity $v$.
To do so, we multiply the vorticity equation \eqref{eq.Omeg} with $\Omega$ then integrate with respect to the space variable, one finds that
\begin{align}\label{es1.v.H1}
  \frac{1}{2}\frac{\dd}{\dd t} \|\Omega\|_{L^2}^2 + \| \nabla \Omega\|_{L^2}^2  &\leq \|\Omega \|_{L^6} \|\nabla v\|_{ L^3}\|\Omega \|_{L^2}+\| \nabla\Omega \|_{L^2}\|\theta\|_{L^2}\nonumber\\
  &\leq C\big(\|\nabla v\|_{\dot H^{\frac12}}^2 \|\Omega\|_{L^2}^2+\|\theta_0\|^2_{L^2}\big)+\frac 12 \|\nabla\Omega\|_{L^2}^2,
\end{align}
which, together with the use of Gr\"onwall's inequality and \eqref{eq:v-H1/2es}, gives
\begin{align}\label{es2.v.H1}
  \|\Omega\|_{L^\infty_T(L^2)}^2 + \| \nabla \Omega\|_{L^2_T(L^2)}^2 \leq C\big(\| \Omega_0\|_{ L^2}^2+\|\theta_0\|^2_{L^2}T\big)e^{C\int_0^T\|  v\|_{\dot H^{3/2}}^2\dd t}  \leq C e^{C E(T)}.
\end{align}
Then, combining this control with \eqref{es.v.L2} and interpolation, we see that for all $\rho \in [2,\infty]$,
\begin{equation}\label{es:v-H1}
  \|v\|_{L^\infty_T (H^1)}^2 + \|\nabla v\|_{L^2_T (H^1)}^2 + \|v\|_{L^\rho_T( H^{1+ \frac{2}{\rho}})}^2 \leq C e^{C E(T)}.
\end{equation}

Then, we deal with the estimates in  $L^1_T (B^\gamma_{\infty,1})$ of $\Omega$ where $\gamma\in (0,\frac{1}{2})$. 
By using the smoothing estimate \eqref{es.sm1} with $u\equiv0$
and the embedding $H^1\hookrightarrow B^{\frac 12+\gamma}_{2,1}\hookrightarrow B^{\gamma-1}_{\infty,1}$
for all $\gamma\in (0,\frac12)$, one obtains from equation \eqref{eq.Omeg} that for all $\gamma\in (0,\frac{1}{2})$,
\begin{align}\label{es0.Lip}
  \|\Omega\|_{L^1_T ( B^{\gamma}_{\infty,1})}&\lesssim(1+T)\Big(  \|v_0\|_{H^1}+\int_0^T\|(v\cdot\nabla\Omega, \Omega\cdot\nabla v,\nabla\theta)\|_{B^{\gamma-2}_{\infty,1}}\dd t\Big)\nonumber\\
  &\leq C (1+T)\Big(  \|v_0\|_{H^1}+\int_0^T\|v\otimes\Omega\|_{H^1} \dd t + \int_0^T \|\nabla\theta\|_{B^{\gamma-2}_{\infty,1}}\dd t\Big).
\end{align}
Then, by using the fact that $\|\nabla\theta\|_{B^{\gamma-2}_{\infty,1}}\lesssim \|\theta\|_{L^\infty} \leq C \|\theta_0\|_{L^\infty}$ and
\begin{align*}
  \|v\otimes\Omega\|_{H^1}&\leq \|v\otimes\Omega\|_{L^2}+\|\nabla(v\otimes\Omega)\|_{L^2}\\
  &\lesssim \|v\|_{L^6}\|\Omega\|_{L^3}+\|\nabla v\|_{L^6}\|\Omega\|_{L^3}+\|v\|_{L^\infty}\|\nabla\Omega\|_{L^2}
  \lesssim \|\nabla v\|^2_{H^1} + \|v\|_{L^2}^2 ,
\end{align*}
we get from  \eqref{es0.Lip} and \eqref{es:v-H1} that, for all $\gamma\in[0,\frac 12)$,
\begin{align}\label{es:Omeg-Lip}
  \|\Omega\|_{L^1_T ( B^{\gamma}_{\infty,1})}\lesssim(1+T)\Big(  \|v_0\|_{H^1} + \|\nabla v\|^2_{L^2_T(H^1)} + \|v\|^2_{L^2_T(L^2)}
  +\|\theta_0\|_{L^\infty}T\Big)\leq C e^{CE(T)}.
\end{align}
Futhermore, we have that for all $\gamma\in [0,\frac{1}{2})$,
\begin{equation}\label{es.v.Lip}
  \|v\|_{L^1_T (B^{1+\gamma}_{\infty,1})} \lesssim \|\Delta_{-1} v\|_{L^1_T (L^\infty)} + \|\Omega\|_{L^1_T (B^\gamma_{\infty,1})}
  \leq C  e^{C E(T)} .
\end{equation}

Now, assume that $v_0\in H^1\cap W^{1,p}$ with $p>2$, we want to control the  ${L^p}$ norm of $\Gamma$.
Multiplying both sides of the equation \eqref{eq.Gamm} by $|\Gamma|^{p-2}\Gamma$
and integrating in the space variable and then doing an integration by parts, we find
\begin{eqnarray*}
   \frac{1}{p}\frac{\dd }{\dd t}\|\Gamma\|_{L^p}^p + (p-1)\int_{\R^3} |\nabla \Gamma|^2 |\Gamma|^{p-2} \dd x &\leq&    (p-1) \Big| \int_{\R^3} (\Omega \otimes v) : (\nabla \Gamma) |\Gamma|^{p-2} \dd x \Big| + \|[\mathcal{R}_{-1},v\cdot\nabla]\theta\|_{L^p} \|\Gamma\|_{L^p}^{p-1} \\
  &\leq& (p-1) \Big(\int_{\R^3} |\nabla \Gamma|^2 |\Gamma|^{p-2} \dd x \Big)^{1/2} \Big(\int_{\R^3} |\Omega|^2 |v|^2 |\Gamma|^{p-2} \dd x \Big)^{1/2} \\  && \ +  \ \|[\RR_{-1}, u\cdot\nabla]\theta\|_{L^p} \|\Gamma \|_{L^p}^{p-1} \\
  &\leq&  (p-1)\|\Gamma \|_{L^p}^{p-2} \|\Omega\|_{L^p}^2\|v \|_{L^\infty}^2  + \|\Gamma \|_{L^p}^{p-1} \|[\RR_{-1}, v\cdot\nabla]\theta \|_{L^p}  \\
  && \ + \ \frac 12(p-1)\int_{\R^3} |\nabla \Gamma|^2 |\Gamma|^{p-2}  \dd x.
\end{eqnarray*}
Then, since  $\Omega=\Gamma+\mathcal{R}_{-1}\theta$ and $\RR_{-1}\theta = (\RR_{-1,2}\theta, -\RR_{-1,1}\theta,0)^t$ we find that
\begin{eqnarray}\label{es.Gam.Lp}
  \frac{1}{2}\frac{\dd }{\dd t}\|\Gamma\|_{L^p}^2 &\leq& (p-1) \|\Omega\|_{L^p}^2 \|v\|_{L^\infty}^2 + \|\Gamma \|_{L^p}\|[\RR_{-1}, v\cdot\nabla]\theta \|_{L^p} \\
  &\leq&  \|\Gamma\|_{L^p}^2 \big(2(p-1)\|v \|_{L^\infty}^2+ 1 \big)  + 2(p-1) \|\RR_{-1}\theta \|_{L^p}^2\|v  \|_{L^\infty}^2 + \|[\RR_{-1}, v\cdot\nabla]\theta \|_{L^p}^2. \nonumber
\end{eqnarray}
We then use the Hardy-Littlewood-Sobolev inequality which allows us to get that, for all $p>2$,
\begin{align}\label{es.Hd}
  \|\RR_{-1}\theta\|_{L^p} \lesssim\|\Lambda^{-1}\theta\|_{L^p}\lesssim \| \theta\|_{L^{\frac {3p}{p+3}}}\lesssim \| \theta_0\|_{L^{\frac{3p}{p+3}}}.
\end{align}
Therefore, using Lemma \ref{lem:comm-es}, we obtain that
\begin{align}\label{es1.R-1}
  \|[\RR_{-1}, v\cdot\nabla]\theta\|_{L^p} &\lesssim \|[\RR_{-1}, v\cdot\nabla]\theta\|_{B^1_{p,\infty}} \nonumber\\
  &\lesssim \|\Omega\|_{L^p}\|\theta\|_{L^\infty}+\|v\|_{L^2}\|\theta\|_{L^2}\nonumber\\
  &\lesssim (\|\Gamma\|_{L^p}+\|\RR_{-1}\theta\|_{L^p}) \|\theta\|_{L^\infty}+\|v\|_{L^2}\|\theta\|_{L^2}\nonumber\\
  &\lesssim (\|\Gamma\|_{L^p}+\| \theta_0\|_{L^{\frac {3p}{p+3}}}) \|\theta_0\|_{L^\infty} + (\|v_0\|_{L^2}+ T \|\theta_0\|_{L^2}) \|\theta_0\|_{L^2} .
\end{align}
Using \eqref{es.Hd} and \eqref{es1.R-1} in  \eqref{es.Gam.Lp}, one finds 
\begin{align*}
  \frac{\dd }{\dd t}\|\Gamma(t)\|_{L^p}^2 \leq  C\big(1+ \|\theta_0\|_{L^\infty}^2 + \|v(t) \|_{L^\infty}^2\big) \big(\|\Gamma(t)\|_{L^p}^2 +  \|\theta_0 \|_{L^{\frac{3p}{p+3}}}^2\big)
  + C (1+T)^2 \|(v_0,\theta_0) \|_{L^2}^2 \|\theta_0\|_{L^2}^2  .
\end{align*}
Hence, Gr\"onwall's inequality and the control of $v$ in $L^2L^{\infty}$ given by \eqref{es.L2Linf}, imply that
\begin{align}\label{es:Gam.Lp}
  \|\Gamma\|_{L^\infty_T(L^p)}\leq & C \big(1 + T^2 + \|v\|^2_{L^2_T(L^\infty)}\big) \exp\Big\{C\|v\|^2_{L^2_T(L^\infty)} + C(1+T)\Big\}
  \leq C e^{CE(T)^3},
\end{align}
where $C>0$ depends on $p$ and on the norms of $(v_0,\theta_0)$. Note that we have  used that
\begin{align*}
  \|\Gamma_0\|_{L^p} \leq \|\Omega_0\|_{L^p} + \|\RR_{-1}\theta_0\|_{L^p} \leq C \|v_0\|_{W^{1,p}} + C\|\theta_0\|_{L^{\frac{3p}{p+3}}}\leq C.
\end{align*}
Moreover, we get the control in $L^p$ of $\Omega$ from \eqref{es:Gam.Lp} and \eqref{es.Hd}, that is, 
\begin{align}\label{es.Ome.Lp}
  \|\Omega\|_{L^\infty_T (L^p)} \leq \|\Gamma\|_{L^\infty_T (L^p)} + \|\RR_{-1}\theta\|_{L^\infty_T (L^p)}  \leq  C e^{CE(T)^3}.
\end{align}
By taking advantage of the high/low frequency decomposition one finds that for all $p>2$,
\begin{align}\label{es.v.W1p}
  \|v\|_{L^\infty_T (W^{1,p})} &\leq  \|\Delta_{-1} v\|_{L^\infty_T (W^{1,p})} + \|(\mathrm{Id}-\Delta_{-1})v\|_{L^\infty_T (W^{1,p})} \nonumber\\
  & \leq C \|v\|_{L^\infty_T (L^2)} + C \|\Omega\|_{L^\infty_T (L^p)} \nonumber \\
  &\leq C  e^{CE(T)^3}.
\end{align}

Recalling the equation verified by $\Gamma$ (see \eqref{eq.Gamm}) and using the smoothing effect given by \eqref{TD-sm-es} and Lemma \ref{lem:comm-es},
one obtains that, for all $p>2$ and $r\geq 1$,
\begin{align}\label{es.Gam.fr}
   \sup_{j\in \N} 2^{\frac{2}{r}j}\|\Delta_j \Gamma\|_{L^r_T (L^p)} & \lesssim \|\Gamma_0\|_{L^p} + \int_0^T \|\nabla v\|_{L^p} \|\Gamma\|_{B^0_{\infty,\infty}} \dd t \ + \  \|[\RR_{-1},v\cdot\nabla]\theta \|_{L^1_T(L^p)} + \|\Omega\cdot\nabla v\|_{L^1_T (L^p)}  \nonumber \\
   &\lesssim  \|\Omega_0\|_{L^p} + \|\RR_{-1}\theta_0\|_{L^p} + \ \|v\|_{L^\infty_T ( W^{1,p}) } \big(\|\Omega\|_{L^1_T (L^\infty )}  + \|\theta\|_{L^1_T(L^1\cap L^\infty)} \big) \nonumber \\
    & \quad + \ T \|\Omega\|_{L^\infty_T(L^p)} \|\theta\|_{L^\infty_T(L^\infty)} + T \|(v,\theta)\|_{L^\infty_T(L^2)}  \nonumber \\
   & \leq C  e^{C E(T)^3},
\end{align}
where in the last inequality we have used \eqref{es:Omeg-Lip}, \eqref{es.v.W1p} and the fact that $\|\RR_{-1}\theta\|_{B^0_{\infty,\infty}}\leq C \|\theta\|_{L^1\cap L^\infty}$.
Then using \eqref{es:Gam.Lp} and \eqref{es.Gam.fr}, we find that for any $p>2$ and $r\geq 1$,
\begin{align}\label{es.Gam-Besov}
  \|\Gamma\|_{\widetilde{L}^r_T(B^{\frac{2}{r}}_{p,\infty})} &\leq C T^{\frac{2}{r}}\|\Delta_{-1}\Gamma\|_{L^\infty_T(L^p)} + \sup_{j\in \N} 2^{\frac{2}rj}\|\Delta_j \Gamma\|_{L^r_T (L^p)} \nonumber \\
  & \leq C  e^{CE(T)^3}.
\end{align}

Finally, by using \eqref{es.Gam-Besov} and the continuous embedding $B^2_{3,\infty}\hookrightarrow B^\gamma_{\infty,1}$ for $\gamma\in (0,1)$,
we may get a more precise estimates than \eqref{es:Omeg-Lip}-\eqref{es.v.Lip} for all $\gamma\in [0,1)$. More precisely, we have
\begin{align}\label{eq:v-BesEs}
  \|v\|_{\widetilde{L}^1_T(B^{\min\{2, 3-2/p\}}_{\infty,\infty})} &\leq \|\Delta_{-1} v\|_{L^1_T(L^\infty)}
  + C \|\Omega\|_{\widetilde{L}^1_T(B^{\min\{1, 2- 3/p\}}_{\infty,\infty})} \nonumber \\
  &\leq  T \|v\|_{L^\infty_T(L^2)} + C \|\Gamma\|_{\widetilde{L}^1_T(B^{2-3/p}_{\infty,\infty})}
  + C \|\RR_{-1}\theta\|_{L^1_T(B^1_{\infty,\infty})} \nonumber \\
  &\leq C (1+T)^2 + C \|\Gamma\|_{\widetilde{L}^1_T (B^2_{p,\infty})} + C \|\theta\|_{L^1_T(L^1\cap L^\infty)} \nonumber \\
  &\leq C  e^{CE(T)^3}.
\end{align}

Hence, using \eqref{es:v-H1}, \eqref{es.v.W1p}, \eqref{es.Gam-Besov} and \eqref{eq:v-BesEs}, we find that \eqref{eq:v-gam-es1} holds, as desired.
\qed

\subsection{Persistence of the $W^{2,\infty}$ regularity of the temperature}\label{sub.W2inf}
By using Lemma \ref{lem:flow}, in order to show the control of the   $ L^\infty_T(W^{2,\infty})$ norm of the temperature front,
it suffices to control the velocity $v$ in the space $L^1(0,T; W^{2,\infty})$.
In view of the Biot-Savart law and the relation $\Omega= \Gamma + \RR_{-1}\theta$ with $\RR_{-1}\theta = \Lambda^{-2} \nabla \wedge (\theta e_3)$, we see that
\begin{align}\label{eq.v.GamThe}
  \nabla v = (-\Delta)^{-1}\nabla \nabla \wedge \Omega =&  \nabla\Lambda^{-2} \nabla\wedge \Gamma + \nabla \Lambda^{-4} \nabla \wedge \nabla \wedge (\theta e_3) \nonumber\\
  =&  \nabla \Lambda^{-2} \nabla\wedge \Gamma + \Lambda^{-4} \nabla^2\partial_3  \theta+ (\Lambda^{-2} \nabla \theta)\otimes e_3 ,
\end{align}
where in the last line we have used the formula $\nabla \wedge \nabla \wedge f = \nabla (\Div f) - \Delta f$,
therefore we find that
\begin{align}\label{eq.nab2v}
  \nabla^2 v = \nabla^2 \Lambda^{-2}\nabla\wedge \Gamma + \nabla^3\partial_3  \Lambda^{-4} \theta+ (\nabla^2 \Lambda^{-2} \theta)\otimes e_3 .
\end{align}

Since $v_0\in H^1\cap W^{1,p}$ for some $p>3$, we find that, thanks to \eqref{eq:v-gam-es1}

and using the embedding $B^{\frac{2}{r}}_{p,\infty} \hookrightarrow B^1_{\infty,1}$ for all  $r\in [1, \frac{2p}{p+3})$, we have
\begin{align}\label{es.Gam.B1inf}
  \|\Gamma\|_{L^r_T(B^1_{\infty,1})}  \lesssim  \|\Gamma\|_{\widetilde{L}^r_T (B^1_{\infty,1})}
  \lesssim \|\Gamma\|_{\widetilde{L}^r_T(B^{\frac{2}{r}}_{p,\infty})} \leq C e^{C E(T)^3},
\end{align}
which readily gives
\begin{align}\label{es.Gam.Lip}
  \|\nabla^2 \Lambda^{-2}\nabla\wedge \Gamma\|_{L^r_T (L^\infty)}
  \lesssim \|\Gamma\|_{L^r_T(B^1_{\infty,1})} \leq C e^{C E(T)^3},
\end{align}
for all $r\in [1, \frac{2p}{p+3})$ and $p>3$.

Now, we are going to show that $\nabla^3\partial_3  \Lambda^{-4} \theta$ and $\nabla^2 \Lambda^{-2} \theta\otimes e_3 $
belong to $L^\infty(\R^3 \times [0,T] )$ for all $T<T^*$ and for all initial temperature front  \eqref{eq.th0}.
It suffices to focus on the control of $\nabla^3\partial_3  \Lambda^{-4} \theta$, since the control of the other term is similar.

We shall apply some striated estimates pioneered in some works of J-Y. Chemin \cite{Chem88,Chem91} and further developed by P. Gamblin and X. Saint-Raymond \cite{GSR95}.
 We start by recalling a fundamental expression formula of $\partial_j\partial_k$ applied to an admissible system of vector fields (see \cite{GSR95}).
\begin{lemma}\label{lem:stra-exp}
  Let $0<\gamma <1$ and let $\mathcal{W}=\{W^i\}_{1\leq i\leq N}$ be an admissible system of $C^\gamma$-vector fields.
Then, there exist functions $a_{j,k}\in C^\gamma(\R^3)$ and $b_{j,k}^{l,i}\in C^\gamma(\R^3)$ such that for all $g\in \mathcal{S}'(\R^3)$,
\begin{align}\label{eq:par-jk-decom}
  \partial_j \partial_k  g = \Delta( a_{j,k}\, g )
  + \sum_{l,i,\sigma} \partial_l \partial_\sigma \big( b_{j,k}^{l,i} W^i_\sigma\, g \big),
\end{align}
where $\|a_{j,k}\|_{L^\infty} \leq 1$ and
\begin{equation}\label{eq:ajk-b-es}
\begin{split}
  \|b_{j,k}^{l,i}\|_{C^\gamma} & \leq C N^3 \Big( \|[\mathcal{W}]^{-1}\|_{L^\infty} \sum_{1\leq i \leq N} \|W^i\|_{C^\gamma} \Big)^{19},
\end{split}
\end{equation}
where the constant $C$ depends only on $\gamma$.
\end{lemma}

\begin{remark}
  Note that $a_{j,k}$ and $b_{j,k}^{l,i}$ in Lemma \ref{lem:stra-exp} are constructed via a partition of unity from local expressions of the following form
\begin{align}
  a_{j,k}= \frac{(W^m\times W^n)_j (W^m\times W^n)_k}{|W^m\times W^n|^2},\quad b_{j,k}^{l,i} = \frac{P_{j,k}^{l,i}(W^m,W^n)}{|W^m\times W^n|^4},
\end{align}
where $P_{j,k}^{l,i}(W^m,W^n)$ are homogeneous polynomials of $(W^m,W^n)$ of degree $7$ and $W^m\times W^n $ does not vanish.
\end{remark}

The following result (see \cite{GSR95}) will be useful to deal with the second term in \eqref{eq:par-jk-decom}.
\begin{lemma}\label{lem:par-jk-Cgam}
  Under the assumptions of Lemma \ref{lem:stra-exp}, we have that
\begin{align}
  \|\Lambda^{-2}\big(\partial_j \partial_k g - \Delta (a_{j,k}g)\big)\|_{C^\gamma} \leq C N^3 \Big(\|[\mathcal{W}]^{-1}\|_{L^\infty} \sum_{1\leq i\leq N} \|W_i\|_{C^\gamma}\Big)^{19} \|g\|_{C^\gamma_{\mathrm{co}}},
\end{align}
where
\begin{align*}
  \|g\|_{C^\gamma_{\mathrm{co}}} := \|g\|_{L^\infty} + \|[\mathcal{W}]^{-1}\|_{L^\infty} \sum_{1\leq i\leq N} \big(\|W^i\|_{C^\gamma} \|g\|_{L^\infty} + \|\partial_\mathcal{W}g\|_{C^{-1,\gamma}}\big).
\end{align*}
\end{lemma}

Now, let $\mathcal{W}(t)=\{W^i(t)\}_{1\leq i\leq 5}$ be a family of divergence-free vector fields which verifies \eqref{eq:Wi}.
According to \eqref{es:v-L1Cgam} and Lemma \ref{lem:flow}, we know that for all $T<T^*$, $v\in L^1(0,T; C^{1,\gamma}(\R^3))$
and the particle trajectory $\psi^{\pm}_t$ belongs to  $L^\infty(0,T; C^{1,\gamma}(\R^3))$ and satisfies that
\begin{align*}
  \|\psi_t^{\pm}\|_{L^\infty_T(C^{1,\gamma})} \leq
  C e^{C\int_0^T \|\nabla v\|_{L^\infty} \dd \tau} \Big( 1 + \int_0^T \|\nabla v\|_{C^\gamma} \dd \tau \Big)
  \leq C e^{\exp\{C E(T)\}}.
\end{align*}
 Note that the latter is obtained by using \eqref{eq:v-gam-es1} and \eqref{es.v.Lip}.
Therefore, thanks to the formula \eqref{exp:Wit}, we have that $W^i \in L^\infty(0,T;C^\gamma(\R^3))$ where
\begin{align}\label{eq:W-Cgam}
  \|W^i\|_{L^\infty_T(C^\gamma)} \leq C \|W_0^i\cdot\nabla \psi_t\|_{L^\infty_T(C^\gamma)}
  \|\psi^{-1}_t\|_{L^\infty_T(C^\gamma)}  \leq C e^{\exp\{C E(T)\}}.
\end{align}
Moreover, thanks to \cite{GSR95} (see Corollary 4.3)
and \eqref{es.v.Lip}, one has the following control
\begin{eqnarray}\label{eq:W-1es}
  \|[\mathcal{W}(t)]^{-1}\|_{L^\infty} &\leq& C e^{C \|\nabla v\|_{L^1_t(L^\infty)}} \|[\mathcal{W}_0]^{-1}\|_{L^\infty}  \nonumber \\
  &\leq& C e^{\exp\{C E(T)\}}.
\end{eqnarray}

Then, we consider the control of $\partial_j \partial_k \partial_\lambda \partial_3 \Lambda^{-4} \theta$ in $L^\infty$,
for all $j,k,\lambda=1,2,3$. According to \eqref{eq:par-jk-decom}, we infer that
\begin{align}\label{eq:the-decom0}
  \partial_\lambda\partial_3 \Lambda^{-2} \theta = - a_{\lambda,3} \theta + \sum_{m,n,\lambda} \partial_m \partial_\nu \Lambda^{-2} \big( b_{\lambda,3}^{m,n} W^n_\lambda \theta\big) :=  - a_{\lambda,3} \theta + {I_1} ,
\end{align}
and
\begin{align}\label{eq:the-decom}
  \partial_j\partial_k \partial_\lambda \partial_3 \Lambda^{-4} \theta
  = - a_{j,k} (\partial_\lambda\partial_3 \Lambda^{-2} \theta )
  + \sum_{l,i,\sigma} \partial_l\partial_\sigma \Lambda^{-2} \big(b_{j,k}^{l,i} W_\sigma^i (\partial_\lambda \partial_3 \Lambda^{-2} \theta) \big)
  := {I_2} + {I_3}.
\end{align}
First note that $\partial_{\mathcal{W}}\theta$ satisfies that
\begin{align}\label{eq:par-W-the}
  \partial_t \partial_{\mathcal{W}} \theta + v \cdot \nabla \partial_\mathcal{W} \theta =0,\qquad
  \partial_\mathcal{W} \theta|_{t=0} = \partial_{\mathcal{W}_0}\theta_0,
\end{align}
then the estimates \eqref{eq:T-sm3}, \eqref{es.v.Lip} and Lemma \ref{lem:str-reg} imply that for $\mu= \min\{\mu_1,\mu_2\}$,
\begin{align}\label{es.pWthtn}
  \|\partial_\mathcal{W} \theta(t)\|_{C^{-1,\mu}} \leq \|\partial_{\mathcal{W}_0} \theta_0\|_{C^{-1,\mu}}
  e^{C\int_0^t \|\nabla v(\tau)\|_{L^\infty}\dd\tau} \leq C e^{\exp\{C E(T)\}} .
\end{align}
For the term ${I}_1$, by making use of Lemma \ref{lem:par-jk-Cgam},
and \eqref{eq:W-Cgam},\eqref{eq:W-1es}, \eqref{es.pWthtn}, we get that
\begin{align*}
  \|I_1\|_{C^\mu} \leq C \Big(\|[\mathcal{W}(t)]^{-1}\|_{L^\infty}
  \sum_{1\leq i\leq 5}\|W^i(t)\|_{C^\mu}\Big)^{19} \|\theta(t)\|_{C^\mu_{\mathrm{co}}}
  \leq C e^{\exp\{C E(T)\}} ,
\end{align*}
and the following control
\begin{eqnarray*}
  \|I_1\|_{B^0_{\infty,\infty}} &\leq& \|\Delta_{-1} {I}(t)\|_{L^\infty} + \sup_{q\in \N} \|\Delta_q \mathrm{I}(t)\|_{L^\infty} \\
  &\lesssim&  \sum_{m,n,\nu} \|b_{\lambda,3}^{m,n} W^n_\nu \theta(t)\|_{L^2\cap L^\infty} \\
  &\lesssim&  \sum_{m,n} \|b_{\lambda,3}^{m,n}\|_{L^\infty}
  \|W^n\|_{L^\infty} \|\theta(t)\|_{L^2\cap L^\infty} \\
  &\lesssim&  e^{\exp\{C E(T)\}},
\end{eqnarray*}
hence, by interpolation inequality, one finds that
\begin{eqnarray}\label{eq:I-es}
  \|I_1\|_{L^\infty_T(L^\infty)} \leq C \big(\|I_1\|_{L^\infty_T(B^0_{\infty,\infty})} + \|I_1\|_{L^\infty_T(C^\gamma)} \big)
  \lesssim  e^{\exp\{C E(T)\}}.
\end{eqnarray}
Therefore, we obtain
\begin{align}\label{es:par-mu3-the}
  \|\partial_\lambda \partial_3 \Lambda^{-2}\theta\|_{L^\infty_T(L^\infty)} + \|I_2\|_{L^\infty_T(L^\infty)}
  \lesssim   e^{\exp\{C E(T)\}}.
\end{align}
The control of ${I}_3$ is analogous to the proof of \eqref{eq:I-es}. Indeed, by using \eqref{es:par-mu3-the} together with the following estimate (which is an easy consequence of \eqref{eq:par-W-mDf} and \eqref{eq:par-W-the})
\begin{eqnarray*}
  \|\partial_\mathcal{W} (\partial_j \partial_k \Lambda^{-2} \theta)\|_{L^\infty_T(B^{\mu-1}_{\infty,\infty})}
  &\lesssim& \|\mathcal{W}\|_{L^\infty_T(C^\mu)}  \|\theta\|_{L^\infty_T(L^2\cap L^\infty)}
  +  \|\partial_\mathcal{W}\theta\|_{L^\infty_T(C^{-1,\mu})} \\
   &\lesssim&  e^{\exp\{C E(T)\}},
\end{eqnarray*}
we find that
\begin{align}
  \|I_3\|_{L^\infty_T(L^\infty)} \lesssim  e^{\exp\{C E(T)\}}.
\end{align}
Collecting all the above estimates allows us to conclude that $\|\nabla^3 \partial_3 \Lambda^{-4}\theta\|_{L^\infty_T(L^\infty)}
\lesssim  e^{\exp\{C E_0(T)\}}$.

Finally, the latter control and \eqref{es.Gam.Lip} give that, for all $r\in [1, \frac{2p}{p+3})$ and $ p>3$,
\begin{align}\label{eq:nab2v-es}
  \|\nabla^2v\|_{L^r_T (L^\infty)} \lesssim  e^{\exp\{C E(T)\}}.
\end{align}
Then using Lemma \ref{lem:flow}, we conclude that the particle trajectory $\psi^\pm_t \in L^\infty([0,T], W^{2,\infty})$ and that the $W^{2,\infty}$ regularity is preserved, that is
\begin{equation*}
  \partial D_0\in W^{2,\infty} \quad \Longrightarrow \quad \partial D(t)= \psi_t(\partial D_0)\in L^\infty([0,T], W^{2,\infty}).
\end{equation*}
\qed

\subsection{Propagation  of the $C^{2,\gamma}$ regularity of the temperature fronts}\label{sec.C2ga}

The goal of this subsection is to prove that the $C^{2,\gamma}$ regularity of the temperature front is preserved along the evolution. 

By using Lemma \ref{lem:sr-cond}, we see that it suffices to prove that \eqref{eq:targ-sr} holds for $k=2$, $\gamma\in (0,1)$.
In fact, in the sequel we shall prove an even stronger result, namely that
\begin{align}\label{eq:targ5}
  \mathcal{W} \in L^\infty(0,T; C^{1,\gamma}(\R^3)),
\end{align}
where the admissible conormal vector system $\mathcal{W}=\{W^i\}_{1\leq i\leq 5}$ verifies \eqref{eq:Wi}. \vskip1mm

 Applying the operator $\nabla^2$ to equation \eqref{eq:Wi} gives  that
\begin{align}\label{eq.nab2W}
  \partial_t (\nabla^2 \mathcal{W}) + v\cdot\nabla (\nabla^2 \mathcal{W}) = \partial_\mathcal{W} \nabla^2 v
  + 2 \nabla \mathcal{W} \cdot\nabla^2 v + \nabla^2 \mathcal{W}\cdot \nabla v - \nabla^2 v \cdot \nabla \mathcal{W}
  - 2 \nabla v \cdot \nabla^2\mathcal{W},
\end{align}
where $\mathcal{W}=\{W^i\}_{1\leq i\leq 5}$ and $\mathcal{W}\cdot\nabla = \{W^i\cdot\nabla \}_{1\leq i\leq 5}$
are both vector-valued. Thanks to \eqref{eq:T-sm2}, we find that for all $t\in [0,T]$,
\begin{align}\label{es.C2gn}
  \| \nabla^2 \mathcal{W}(t)\|_{B^{\gamma-1}_{\infty,\infty}} & \lesssim \| \nabla^2 \mathcal{W}_0\|_{B^{\gamma-1}_{\infty,\infty}} +  \int_0^t \|\nabla v(\tau)\|_{L^\infty} \|\nabla^2 \mathcal{W}(\tau)\|_{B^{\gamma-1}_{\infty,\infty}}\dd \tau  \nonumber \\
  & \quad + \int_0^t \|\partial_{\mathcal{W}} \nabla^2 v(\tau)\|_{B^{\gamma-1}_{\infty,\infty}} \dd \tau +  \int_0^t \|(\nabla \mathcal{W}\cdot\nabla^2 v, \nabla^2 v\cdot\nabla \mathcal{W})\|_{B^{\gamma-1}_{\infty,\infty}} \dd \tau  \nonumber \\
  & \quad +  \int_0^t \|(\nabla^2 \mathcal{W}\cdot\nabla v, \nabla v\cdot \nabla^2 \mathcal{W}) \|_{B^{\gamma-1}_{\infty,\infty}}\dd \tau.
\end{align}
Recalling that $\mathcal{W}_0 = \{W^i_0\}_{1\leq i\leq 5}$ given by \eqref{eq:Wi0} satisfies $\mathcal{W}_0\in C^{1,\gamma}= B^{1+\gamma}_{\infty,\infty}$,
we immediately see that $\|\nabla^2 \mathcal{W}_0\|_{B^{\gamma-1}_{\infty,\infty}} \lesssim \|\mathcal{W}_0\|_{B^{\gamma+1}_{\infty,\infty}} <\infty$.
Applying the product estimate \eqref{eq:prod-es} and \eqref{eq:W-Cgam}, \eqref{eq:nab2v-es} to the last two terms of \eqref{es.C2gn}, we get
\begin{equation}\label{es.C2gn1}
\begin{split}
  \int_0^t \|(\nabla \mathcal{W}\cdot\nabla^2 v, \nabla^2 v\cdot\nabla \mathcal{W})\|_{B^{\gamma-1}_{\infty,\infty}} \dd \tau  &\lesssim  \|\nabla^2 v\|_{L^1_t(L^\infty)}\|\nabla \mathcal{W}\|_{L^\infty_t(B^{\gamma-1}_{\infty,\infty})}\nonumber\\
  &\lesssim \|\nabla^2 v\|_{L^1_t(L^\infty)}\| \mathcal{W}\|_{L^\infty(C^\gamma)} \leq C e^{\exp \{CE(T)\}} ,
\end{split}
\end{equation}
and
\begin{align}\label{es.C2gn2}
  \int_0^t \|(\nabla^2 \mathcal{W}\cdot\nabla v, \nabla v\cdot \nabla^2 \mathcal{W})\|_{B^{\gamma-1}_{\infty,\infty}}\dd \tau
  \lesssim \int_0^t \|\nabla v(\tau)\|_{L^\infty}\|\nabla^2 \mathcal{W}(\tau) \|_{B^{\gamma-1}_{\infty,\infty}}\dd \tau.
\end{align}
In order to control the term $  \|\partial_{\mathcal{W}} \nabla^2 v\|_{L^1_t(B^{\gamma-1}_{\infty,\infty})}$ in \eqref{es.C2gn},
we use the identity \eqref{eq.nab2v} to find  that
\begin{eqnarray}\label{es.C2gn3}
  \|\partial_{\mathcal{W}} \nabla^2 v\|_{L^1_t(B^{\gamma-1}_{\infty,\infty})} &\lesssim&  \|\partial_{\mathcal{W}}(\nabla^2 \Lambda^{-2}\nabla\wedge \Gamma)\|_{L^1_t(B^{\gamma-1}_{\infty,\infty})} +\|\partial_{\mathcal{W}}(\nabla^3\partial_3  \Lambda^{-4} \theta)\|_{L^1_t(B^{\gamma-1}_{\infty,\infty})} \nonumber \\
  && \ + \ \|\partial_{\mathcal{W}}(\nabla^2 \Lambda^{-2} \theta )\|_{L^1_t(B^{\gamma-1}_{\infty,\infty})} \\
  &:=&  K_{1}+ K_{2}+ K_{3}. \nonumber
\end{eqnarray}
In order to control $K_1$, we use \eqref{eq:prod-es4} together with the fact that $\nabla \Delta_{-1}$ is a bounded operator on $L^\infty$, we find
\begin{align}\label{es.C2gn31}
  K_1\lesssim & \int_0^t \|\mathcal{W}(\tau)\|_{B^1_{\infty,1}}
  \|\nabla\Gamma(\tau)\|_{B^{\gamma-1}_{\infty,\infty}}  \dd \tau + \|\partial_{\mathcal{W}} \nabla\Gamma\|_{L^1_t(B^{\gamma-1}_{\infty,\infty})}\nonumber\\
  \lesssim & \int_0^t \|\mathcal{W}(\tau)\|_{C^{1,\gamma}} \|\Gamma(\tau)\|_{B^\gamma_{\infty,\infty}} \dd \tau + \|\partial_{\mathcal{W}} \nabla \Gamma\|_{L^1_t(B^{\gamma-1}_{\infty,\infty})} ,
\end{align}
It remains to control $K_2$ and $K_3$. To do so, it suffices to observe that
\begin{align}\label{es.C2gnK23}
  K_2+K_3 &\lesssim \int_0^t \|\mathcal{W}(\tau)\|_{B^1_{\infty,1}}
  \|\theta(\tau)\|_{B^{\gamma-1}_{\infty,\infty}} \dd \tau + \|\partial_{\mathcal{W}} \theta\|_{L^1_t(B^{\gamma-1}_{\infty,\infty})} \nonumber\\
  &\lesssim \|\mathcal{W}\|_{L^1_t(C^{1,\gamma})} \|\theta_0\|_{L^\infty}  + \|\partial_{\mathcal{W}} \theta\|_{L^1_t(B^{\gamma-1}_{\infty,\infty})}.
\end{align}
Note that Lemma \ref{lem:str-reg} now implies that $\partial_{\mathcal{W}_0}\theta_0\in C^{-1,\gamma}$ and therefore, following the same approach as the proof of \eqref{es.pWthtn}, one obtains
\begin{align}\label{es.pWthtn2}
  \|\partial_{\mathcal{W}}\theta(t)\|_{B^{\gamma-1}_{\infty,\infty}} \leq \|\partial_{\mathcal{W}_0} \theta_0\|_{B^{\gamma-1}_{\infty,\infty}} e^{C\int_0^t \|\nabla v\|_{L^\infty}\dd\tau}
  \lesssim e^{\exp\{CE(T)\}} .
\end{align}
Finally, by collecting all the above estimates  \eqref{es.C2gn}, \eqref{es.pWthtn2}, one finds that
\begin{align}\label{es.C2gn-2}
   \|\mathcal{W}(t)\|_{B^{\gamma+1}_{\infty,\infty}} + \|\partial_{\mathcal{W}} \nabla^2 v\|_{L^1_t(B^{\gamma-1}_{\infty,\infty})} &\lesssim \|\mathcal{W}\|_{L^\infty_t(L^\infty)} + \| \nabla^2 \mathcal{W}(t)\|_{B^{\gamma-1}_{\infty,\infty}} + \|\partial_{\mathcal{W}} \nabla^2 v\|_{L^1_t(B^{\gamma-1}_{\infty,\infty})} \nonumber \\
  &\lesssim e^{CE_1(T)^3} + \|\partial_{\mathcal{W}} \nabla \Gamma\|_{L^1_t(B^{\gamma-1}_{\infty,\infty})}\nonumber \\
  &\quad  + \int_0^t \|\mathcal{W}(\tau)\|_{B^{\gamma+1}_{\infty,\infty}} \big(\|\Gamma\|_{B^1_{\infty,\infty}} + \|\nabla v\|_{L^\infty} + 1 \big) \dd \tau.
\end{align}

We now study the control of the term $\|\partial_{\mathcal{W}} \nabla \Gamma\|_{L^1_t(B^{\gamma-1}_{\infty,\infty})}$ in \eqref{es.C2gn31}.
Since we have $[\partial_{\mathcal{W}}, \partial_t + v\cdot\nabla]=0$, it follows from equation \eqref{eq.Gamm} that
\begin{align}\label{eq.pXGa}
  \partial_t\partial_{\mathcal{W}} \Gamma + v\cdot\nabla \partial_{\mathcal{W}}\Gamma -\Delta \partial_{\mathcal{W}}\Gamma & = -[\Delta,\partial_{\mathcal{W}}]\Gamma +  \partial_{\mathcal{W}}(\Omega\cdot \nabla v)
  + \partial_{\mathcal{W}} \big([\mathcal{R}_{-1},v\cdot\nabla] \theta \big) \nonumber \\
  & = : F_{1,1}+F_{1,2}+F_{1,3}.
\end{align}
Thanks to the smoothing estimate \eqref{es.sm1}, we obtain that for all $ \gamma' \in (0,1-\frac{3}{p})$,
\begin{eqnarray}\label{es.pXGa}
 \|\partial_{\mathcal{W}} \Gamma(t)\|_{B^{\gamma'-1}_{\infty,1}} + \|\partial_{\mathcal{W}}\Gamma\|_{L^2_t(B^{\gamma'}_{\infty,1})} &+& \|\partial_{\mathcal{W}} \Gamma\|_{L^1_t(B^{\gamma'+1}_{\infty,1})} \leq C(1+t)  \Big(\|\partial_{\mathcal{W}_0} \Gamma_0\|_{B^{\gamma'-1}_{\infty,1}}  \\
  &+& \ \int_0^t\|\nabla v(\tau)\|_{L^\infty}   \|\partial_{\mathcal{W}} \Gamma(\tau)\|_{B^{\gamma'-1}_{\infty,1}} \dd\tau  +  \sum_{i=1}^3 \|F_{1,i}\|_{L^1_t(B^{\gamma'-1}_{\infty,1})} \Big). \nonumber
\end{eqnarray}
By using the identity $\Gamma_0 = \Omega_0 - \RR_{-1}\theta_0$ and the following embedding
\begin{align}\label{embd.Lp}
  L^p\hookrightarrow B^0_{p,\infty}\hookrightarrow B^{\gamma'-1+\frac 3p}_{p,1}\hookrightarrow B^{\gamma'-1}_{\infty,1}, \quad \textrm{for all}\,\,0<\gamma'<1-\frac{3}{p},
\end{align}
we find
\begin{align}\label{es.pXGa0}
  \|\partial_{\mathcal{W}_0} \Gamma_0\|_{B^{\gamma'-1}_{\infty,1}}\lesssim \, & \|\partial_{\mathcal{W}_0} \nabla v_0\|_{L^p} + \|\partial_{\mathcal{W}_0}\RR_{-1}\theta_0\|_{L^p}\nonumber\\
  \lesssim \, & \|\partial_{\mathcal{W}_0} v_0\|_{W^{1,p}} + \|\nabla \mathcal{W}_0\|_{L^\infty} \|v_0\|_{W^{1,p}} + \|\mathcal{W}_0\|_{L^\infty} \| \theta_0\|_{L^p} <\infty.
\end{align}
By using the product estimate \eqref{eq:prod-es}, we get that
\begin{align}\label{es.F11}
  \|F_{1,1}\|_{L^1_t(B^{\gamma'-1}_{\infty,1})} &\leq \|\Delta{\mathcal{W}}\cdot\nabla \Gamma\|_{L^1_t(B^{\gamma'-1}_{\infty,1})} + 2 \|\nabla{\mathcal{W}}\cdot\nabla^2\Gamma\|_{L^1_t(B^{\gamma'-1}_{\infty,1})} \nonumber \\
  &\lesssim \int_0^t \Big(\|\Delta{\mathcal{W}}(\tau)\|_{B^{\gamma'-1}_{\infty,1}} \|\nabla\Gamma(\tau)\|_{L^\infty} +
  \|\nabla{\mathcal{W}}(\tau)\|_{L^\infty} \|\nabla^2 \Gamma(\tau)\|_{B^{\gamma'-1}_{\infty,1}} \Big)\dd\tau \nonumber\\
  &\lesssim \int_0^t\|\mathcal{W}(\tau)\|_{B^{\gamma'+1}_{\infty,1}} \|\Gamma(\tau)\|_{B^{\gamma'+1}_{\infty,1}} \dd\tau.
\end{align}
For the term $F_{1,2}$, we easily get that
\begin{align}\label{es.F12}
  \|F_{1,2}\|_{L^1_t(B^{\gamma'-1}_{\infty,1})} &\lesssim \|{\mathcal{W}}\|_{L^\infty_t(L^\infty)} \|\Omega\cdot \nabla v\|_{L^1_t(B^{\gamma'}_{\infty,1})} \nonumber \\
  &\lesssim e^{\exp\{C E(T)\}}  \int_0^t\|\Omega(\tau)\|_{B^{\gamma'}_{\infty,1}}\|\nabla v(\tau)\|_{B^{\gamma'}_{\infty,1}} \dd\tau \nonumber \\
  &\lesssim  e^{\exp\{C E(T)\}} \|\nabla v\|_{L^2_t (B^{\gamma'}_{\infty,1})}^2 .
\end{align}
Then, by making use of \eqref{eq.v.GamThe}, \eqref{es.Gam-Besov} together with the (continuous) embedding
$$\widetilde{L}^2_t(B^1_{p,\infty}) \hookrightarrow \widetilde{L}^2_t(B^{\gamma'}_{\infty,1}) \hookrightarrow L^2_t(B^{\gamma'}_{\infty,1})$$ for all $\gamma'\in (0,1-\frac{3}{p})$,
we find
\begin{align}\label{es.F12nv}
  \|\nabla v\|_{L^2_t(B^{\gamma'}_{\infty,1})} &\leq \|\Lambda^{-2}\nabla \nabla\wedge \Gamma\|_{L^2_t(B^{\gamma'}_{\infty,1})}
  + \|\nabla^2 \partial_3 \Lambda^{-4}\theta \|_{L^2_t(B^{\gamma'}_{\infty,1})} + \|\nabla \Lambda^{-2}\theta\|_{L^2_t(B^{\gamma'}_{\infty,1})} \nonumber \\
  &\lesssim  \| \Gamma\|_{L^2_t(L^p\cap B^{\gamma'}_{\infty,1})} + \| \Lambda^{-1} \theta\|_{L^2(L^p)} + \|\theta\|_{L^2_t(B^{\gamma'-1}_{\infty,1})} \nonumber \\
  &\lesssim  \| \Gamma\|_{\widetilde{L}^2_t(B^1_{p,\infty})} + \| \theta\|_{L^2_t (L^{\frac{3p}{p+3}})} + \|\theta\|_{L^2_t(L^\infty)} \nonumber \\
  &\lesssim  e^{C E(T)^3}.
\end{align}
This inequality applied to \eqref{es.F12} gives that
\begin{align}\label{es.F12b}
  \|F_{1,2}\|_{L^1_t(B^{\gamma'-1}_{\infty,1})}\lesssim e^{\exp\{C E(T)\}}.
\end{align}
Recalling that $[\RR_{-1},v\cdot\nabla]\theta$ is given by \eqref{def:R-1comm} and by using Lemma \ref{lem:comm-es} and \eqref{es.v.W1p},
we obtain that
\begin{align}\label{es.F13}
  \|F_{13}\|_{L^1_t(B^{\gamma'-1}_{\infty,1})}&\lesssim \|{\mathcal{W}}\|_{L^\infty_t(L^\infty)}\|\nabla[\RR_{-1},v\cdot\nabla]\theta
  \|_{L^1_t(B^{\gamma'-1}_{\infty,1})}\nonumber\\
  &\lesssim \|{\mathcal{W}}\|_{L^\infty_t(L^\infty)}\|[\RR_{-1},v\cdot\nabla]\theta
  \|_{L^1_t(B^{1}_{p,\infty})}\nonumber\\
  &\lesssim \|{\mathcal{W}}\|_{L^\infty_t(L^\infty)} \big(\|\nabla v\|_{L^1_t(L^p)}\|\theta\|_{L^\infty_t(L^\infty)}+t\|v\|_{L^\infty_t(L^2)}\|\theta\|_{L^\infty_t(L^2)}\big) \nonumber \\
  &\lesssim e^{\exp\{C E(T)\}}.
\end{align}
Hence, applying \eqref{es.pXGa0}, \eqref{es.F11} and \eqref{es.F12b}, \eqref{es.F13} to \eqref{es.pXGa}, we infer that
\begin{align}\label{es.pXGa-b}
  &\quad \|\partial_{\mathcal{W}} \Gamma(t)\|_{B^{\gamma'-1}_{\infty,1}} + \|\partial_{\mathcal{W}} \Gamma\|_{L^2_t(B^{\gamma'}_{\infty,1})}
  + \|\partial_{\mathcal{W}} \Gamma\|_{L^1_t(B^{\gamma'+1}_{\infty,1})} \nonumber \\
  & \lesssim (1+t)\Big( e^{\exp\{C E(T)\}} +  \int_0^t\|\nabla v\|_{L^\infty}   \|\partial_{\mathcal{W}} \Gamma(\tau)\|_{B^{\gamma'-1}_{\infty,1}} \dd\tau  
  + \int_0^t \|\mathcal{W}(\tau)\|_{B^{\gamma'+1}_{\infty,1}} \|\Gamma\|_{B^{\gamma'+1}_{\infty,1}} \dd\tau \Big).
\end{align}

Noticing that $\partial_{\mathcal{W}} \nabla f= \nabla \partial_{\mathcal{W}} f - \nabla \mathcal{W}\cdot\nabla f$,
and using the product estimate \eqref{eq:prod-es}, we may control the term $\|\partial_{\mathcal{W}} \nabla \Gamma\|_{L^1_t(B^{\gamma-1}_{\infty,\infty})}$ in  \eqref{es.C2gn-2} as follows
\begin{eqnarray}\label{es.pXnGa}
  \|\partial_{\mathcal{W}} \nabla \Gamma\|_{L^1_t(B^{\gamma-1}_{\infty,\infty})}
  &\lesssim& \|\partial_{\mathcal{W}} \Gamma\|_{L^1_t(B^{\gamma}_{\infty,\infty})}+\|\nabla \mathcal{W}\cdot\nabla \Gamma\|_{L^1_t(B^{\gamma-1}_{\infty,\infty})}\nonumber\\
  &\lesssim& \|\partial_{\mathcal{W}} \Gamma\|_{L^1_t(B^{\gamma'+1}_{\infty,1})} + \|\nabla \mathcal{W}\|_{L^\infty_t(B^{\gamma-1}_{\infty,\infty})} \| \nabla\Gamma\|_{L^1_t(L^\infty)}\nonumber\\
  &\lesssim& \|\partial_{\mathcal{W}} \Gamma\|_{L^1_t(B^{\gamma'+1}_{\infty,\infty})} + e^{\exp\{CE(T)\}},
\end{eqnarray}
where in the last step we have used \eqref{eq:W-Cgam} and the following estimate (in view of \eqref{es.Gam-Besov} and $B^2_{p,\infty}(\R^3) \hookrightarrow B^{1+\gamma'}_{\infty,1}(\R^3)$ for all $\gamma' \in (0,1- 3/p)$
\begin{align}\label{es.F1GL1}
 \| \nabla\Gamma\|_{L^1_t(L^\infty)}\lesssim \|\Gamma\|_{L^1_t(B^{\gamma'+1}_{\infty,1})}  \lesssim \|\Gamma\|_{\widetilde{L}^1_T(B^2_{p,\infty})}
  \lesssim  e^{CE(T)^3}.
\end{align}
Hence, by using the embedding $B^{\gamma+1}_{\infty,\infty}\hookrightarrow B^{\gamma'+1}_{\infty,1}$ for all $\gamma'\in (0,\gamma)$,
we collect the above estimates \eqref{es.C2gn-2}, \eqref{es.pXGa-b} and \eqref{es.pXnGa} to get that for all $\gamma'  \in (0,\min\{\gamma,1-3/p\})$,
\begin{align*}
  &\quad \|\mathcal{W}(t)\|_{B^{\gamma+1}_{\infty,\infty}} +  \|\partial_{\mathcal{W}} \Gamma(t)\|_{B^{\gamma'-1}_{\infty,1}} + \|\partial_{\mathcal{W}} \Gamma\|_{L^2_t(B^{\gamma'}_{\infty,1})} + \|\partial_{\mathcal{W}} \Gamma\|_{L^1_t(B^{\gamma'+1}_{\infty,1})}
  + \|\partial_{\mathcal{W}} \nabla^2 v\|_{L^1_t(B^{\gamma-1}_{\infty,\infty})} \nonumber \\
  & \lesssim e^{\exp\{CE(T)\}} + (1+t)\int_0^t \big(\|\mathcal{W}(\tau)\|_{B^{\gamma+1}_{\infty,\infty}} + \|\partial_{\mathcal{W}} \Gamma(\tau)\|_{B^{\gamma'-1}_{\infty,1}} \big)\big(\|\Gamma\|_{B^{\gamma'+1}_{\infty,1}} + \|\nabla v\|_{L^\infty} + 1 \big) \dd \tau,
\end{align*}
which, together with the use of Gr\"onwall's inequality and \eqref{es.v.Lip}, \eqref{es.F1GL1}, gives that
\begin{eqnarray}\label{es.C2gn-3}
  \| \mathcal{W}\|_{L^\infty_T(B^{\gamma+1}_{\infty,\infty})} &+& \|\partial_{\mathcal{W}} \Gamma\|_{L^\infty_T(B^{\gamma'-1}_{\infty,1})}
  + \|\partial_{\mathcal{W}} \Gamma\|_{L^2_t(B^{\gamma'}_{\infty,1})} + \|\partial_{\mathcal{W}} \Gamma\|_{L^1_T(B^{\gamma'+1}_{\infty,1})} + \|\partial_{\mathcal{W}} \nabla^2 v\|_{L^1_T(B^{\gamma-1}_{\infty,\infty})} \nonumber \\
  &\lesssim&  e^{\exp\{E(T)\}} \exp\Big\{CT +C(1+T) \|\nabla v\|_{L^1_T(L^\infty)} + (1+T)\|\Gamma\|_{L^1_T(B^{\gamma'+1}_{\infty,1})} \Big\}\nonumber \\
  &\lesssim&  e^{\exp \{C E(T)^3\}} .
\end{eqnarray}
So we have proved the persistence of the $C^{2,\gamma}$ regularity of the temperature front.

Besides, by following the same approach as the estimates \eqref{es.C2gn3}, \eqref{es.pWthtn2} and using \eqref{eq:prod-es4}, \eqref{es.F12nv}, \eqref{es.C2gn-3}, we obtain that for all $\gamma'\in(0, \min\{\gamma,1-3/p\})$,
\begin{eqnarray*}
   \|\partial_{\mathcal{W}}  \nabla^2 v\|_{L^2_T(B^{\gamma'-1}_{\infty,1})}
  &\lesssim&  \|\partial_{\mathcal{W}}(\nabla^2 \Lambda^{-2}\nabla\wedge \Gamma)\|_{L^2_T(B^{\gamma'-1}_{\infty,1})} + \|\partial_{\mathcal{W}}(\nabla^3\partial_3  \Lambda^{-4} \theta, \nabla^2 \Lambda^{-2} \theta )\|_{L^2_T(B^{\gamma'-1}_{\infty,1})} \nonumber \\
  &\lesssim& \|\mathcal{W}\|_{L^\infty_T(B^1_{\infty,1})} \big( \|\nabla \Gamma\|_{L^2_T(B^{\gamma'-1}_{\infty,1})}
  + \|\theta\|_{L^2_T(L^1\cap L^\infty)} \big) + \|\partial_{\mathcal{W}} \nabla \Gamma\|_{L^2_T(B^{\gamma'-1}_{\infty,1})} \\
  && \ + \ \|\partial_{\mathcal{W}}\theta\|_{L^2_T(B^{\gamma'-1}_{\infty,1})} \\
  &\lesssim& \|\mathcal{W}\|_{L^\infty_T(B^{\gamma+1}_{\infty,\infty})} \big(\|\Gamma\|_{\widetilde{L}^2_T(B^1_{p,\infty})} + T^{1/2}\|\theta_0\|_{L^1\cap L^\infty} \big) + \|\partial_{\mathcal{W}} \Gamma\|_{L^2_T(B^{\gamma'}_{\infty,1})} \\
  && \ + \ \|\partial_{\mathcal{W}_0}\theta_0\|_{B^{\gamma'-1}_{\infty,1}} e^{\exp\{C E(T)\}} \\
  &\lesssim&  e^{\exp\{C E(T)^3\}} ,
\end{eqnarray*}
and
\begin{eqnarray}\label{es.navL2tLip}
  \|\partial_{\mathcal{W}}\nabla v\|_{L^2_T(B^{\gamma'}_{\infty,1})} &\lesssim& \|\Delta_{-1} \partial_{\mathcal{W}}\nabla v\|_{L^2_T(L^\infty)} + \|\nabla (\partial_{\mathcal{W}} \nabla v)\|_{L^2_T(B^{\gamma'-1}_{\infty,1})} \nonumber \\
  &\lesssim& \|\mathcal{W}\|_{L^\infty_T(L^\infty)} \|\nabla v\|_{L^2_T(L^\infty)} + \|\nabla \mathcal{W}\|_{L^\infty_T(L^\infty)} \|\nabla v\|_{L^2_T(B^{\gamma'}_{\infty,1})} + \|\partial_{\mathcal{W}}\nabla^2 v\|_{L^2_T(B^{\gamma'-1}_{\infty,1})} \nonumber \\
  &\lesssim&  e^{\exp\{C E(T)^3\}}.
\end{eqnarray}

Using the notation \eqref{norm:BBsln2}, together with the estimates \eqref{eq:v-gam-es1}, \eqref{es:Gam.Lp},
\eqref{es.F12nv} and \eqref{es.F1GL1}--\eqref{es.navL2tLip}, we see that for all $\gamma\in (0,1)$  and  $\gamma'\in (0,\min \{\gamma,1-3/p\})$,
\begin{align}\label{es.el=1}
    &\|\WW \|_{L^\infty_T (\bb^{\gamma+1,0}_{\infty, \mathcal{W}})} + \|\nabla^2 v\|_{L^1_T (\bb^{\gamma-1,1}_{\infty,\mathcal{W}})}+ \|\nabla v\|_{L^2_T(\bb^{\gamma',1}_\mathcal{W})}
  + \|\Gamma\|_{L^\infty_T (\bb^{\gamma'-1,1}_{\mathcal{W}})} + \|\Gamma\|_{L^2_T (\bb^{\gamma',1}_{\mathcal{W}})} + \|\Gamma\|_{L^1_T (\bb^{\gamma'+1,1}_{\mathcal{W}})} \nonumber\\
  &=\|\WW \|_{L^\infty_T(B^{\gamma+1}_{\infty,\infty})} + \|(\nabla^2 v, \partial_{\mathcal{W}} \nabla^2 v)\|_{L^1_T(B^{\gamma-1}_{\infty,\infty})} + \|(\nabla v, \partial_{\WW}\nabla v) \|_{L^2_T(B^{\gamma'}_{\infty,1})}   +  \|(\Gamma, \partial_{\mathcal{W}} \Gamma)\|_{L^\infty_T(B^{\gamma'-1}_{\infty,1})} \nonumber  \\
  & \quad  + \  \|(\Gamma, \partial_{\WW}\Gamma) \|_{L^2_T(B^{\gamma'}_{\infty,1})}+ \|(\Gamma, \partial_{\WW}\Gamma) \|_{L^1_T(B^{\gamma'+1}_{\infty,1})}  \nonumber \\
  &\lesssim  e^{\exp\{C E(T)^3\}} .
\end{align}

\section{Propagation of the $C^{k,\gamma}$ regularity of temperature front with $k\geq 3$}\label{sec:Ck-gam}

In this section, we shall prove that the $C^{k,\gamma}$ regularity of the temperature front $\partial D(t)$
is preserved  for all  $[0,T]$ where $T<T^*$ and for all  $k\ge 3$ and $\gamma\in (0,1)$.
By using Lemma \ref{lem:sr-cond}, it suffices to prove that 
\begin{equation}\label{eq:targ6}
  \big(\partial_\WW^\ell \WW \big)(t,\cdot) \in L^\infty(0,T; C^\gamma(\R^3)),\quad \forall \ell\in \{0,1,\cdots,k-1\},
\end{equation}
where $\WW = \{W^i(t)\}_{1\leq i\leq 5}$ is a family of divergence-free tangential vector fields
$W^i(t)$ which satisfies equation \eqref{eq:Wi}.

We start by proving the following statement. For all $k\geq 3$, $\gamma\in(0,1)$ and  $\gamma'\in (0,\min\{\gamma,1-3/p\})$, one has
\begin{align}\label{eq:Targ7}
  \|\WW\|_{L^\infty_T (\bb^{\gamma+1,k-2}_{\infty,\WW})} + \|\nabla v\|_{L^1_T (\bb^{\gamma,k-1}_{\infty,\WW})} + \|\nabla v\|_{L^2_T(\bb^{\gamma',k-1}_\WW)} + \|\Gamma\|_{L^\infty_T (\bb^{\gamma'-1,k-1}_\WW)} & \nonumber \\
  + \|\Gamma\|_{L^2_T (\bb^{\gamma',k-1}_\WW)} + \|\Gamma\|_{L^1_T (\bb^{\gamma'+1,k-1}_\WW)} &\leq H_{k-1}(T),
\end{align}
where $ H_{k-1}(T)< \infty$ is an upper bound depending on $T$ and $k-1$.
Assume for a while that \eqref{eq:Targ7} is proved,  then it is not difficult to see that
\begin{eqnarray*}
  \sum_{\ell=0}^{k-1}\|\partial_\WW^\ell \WW\|_{L^\infty_T(B^{\gamma}_{\infty,\infty})}
  &\leq& \|\WW\|_{L^\infty_T(\bb^{\gamma+1,k-2}_{\infty,\WW})}
  + \|  \partial_\WW^{k-1} \WW \|_{L^\infty_T(B^{\gamma}_{\infty,\infty})} \\
  &\leq& \|\WW\|_{L^\infty_T(\bb^{\gamma+1,k-2}_{\infty,\WW})}
  + C \|\WW\|_{L^\infty_T(B^{\gamma}_{\infty,\infty})}
  \|\partial_\WW^{k-2}\WW\|_{L^\infty_T (B^{\gamma+1}_{\infty,\infty})} \\
  &\lesssim& e^{\exp\{C E(T)\}} \|\WW\|_{L^\infty_T(\bb^{\gamma+1,k-2}_{\infty,\WW})} \\
  &\lesssim&  e^{\exp\{C E(T)\}} H_{k-1}(T) < \infty,
\end{eqnarray*}
which would implies the desired estimate \eqref{eq:targ6}.

So let us prove \eqref{eq:Targ7}. To do so, we shall use an induction method. Assume that for each $\ell\in \{1,2,\cdots,k-2\},$ the following estimate holds:
\begin{align}\label{assup.el}
   \|\WW\|_{L^\infty_T (\bb^{\gamma+1,\ell-1}_{\infty,\WW})} + \|\nabla v\|_{L^1_T (\bb^{\gamma,\ell}_{\infty,\WW})}
  + \|\nabla v\|_{L^2_T(\bb^{\gamma',\ell}_\WW)}   \|\Gamma\|_{L^\infty_T (\bb^{\gamma'-1,\ell}_\WW)} + \|\Gamma\|_{L^2_T (\bb^{\gamma',\ell}_\WW)}
  + \|\Gamma\|_{L^1_T (\bb^{\gamma'+1,\ell}_\WW)} \leq H_\ell(T) .
\end{align}
Then, we want to prove that this inequality is also true at the rank $\ell+1$, that is,
\begin{align}\label{es.el+1}
  & \|\WW\|_{L^\infty_T (\bb^{\gamma+1,\ell}_{\infty,\WW})} + \|\nabla v\|_{L^1_T (\bb^{\gamma,\ell+1}_{\infty,\WW})}
  + \|\nabla v\|_{L^2_T(\bb^{\gamma',\ell+1}_\WW)} \nonumber\\
  & + \|\Gamma\|_{L^\infty_T (\bb^{\gamma'-1,\ell+1}_\WW)} + \|\Gamma\|_{L^2_T (\bb^{\gamma',\ell+1}_\WW)}
  + \|\Gamma\|_{L^1_T (\bb^{\gamma'+1,\ell+1}_\WW)} \leq H_{\ell+1}(T).
\end{align}

The case $\ell=1$ in \eqref{assup.el} corresponds to \eqref{es.el=1}.
We also notice that under the condition \eqref{assup.el}, we have
$\|\WW\|_{L^\infty_T (\bb^{\gamma+1,\ell-1}_{\infty,\WW})} \lesssim H_\ell(T)$,
so that Lemma \ref{lem:prod-es2} (with $\sigma=\gamma$) and Lemma \ref{lem:Bes-equ}
can be applied by replacing $k$ with $\ell$.

Now, our main goal is to prove \eqref{es.el+1} under the assumption that \eqref{assup.el} holds.
We first get a control  of $\partial_{\WW}^\ell \nabla^2 \WW$ in $L^\infty_T(B^{\gamma-1}_{\infty,\infty})$.
Note that $\partial_\WW^\ell $ can be any $\partial_{W^1}^{\ell_1}\cdots \partial_{W^5}^{\ell_5}$
with $\ell_1 +\cdots +\ell_5 = \ell$, and $\WW = \{W^i\}_{1\leq i\leq 5}$ is vector-valued.
In view of equation \eqref{eq.nab2W} and the fact that $[\partial_\WW, \partial_t+v\cdot\nabla]=0$, we see that
\begin{equation}\label{eq.pelXp2X}
\begin{split}
  \partial_t (\partial_\WW^\ell \nabla^2 \WW) + v\cdot\nabla (\partial_\WW^\ell\nabla^2 \WW) & = \partial_\WW^{\ell+1} \nabla^2 v
  + 2\partial_\WW^\ell  (\nabla \WW \cdot\nabla^2v) + \partial_\WW^\ell  (\nabla^2 \WW\cdot\nabla v) \\
  & \quad -\partial_\WW^\ell  (\nabla^2v\cdot\nabla \WW) - 2 \partial_\WW^\ell  (\nabla v\cdot\nabla^2 \WW)
  :=  \sum_{i=1}^5 F_{\ell,i} .
\end{split}
\end{equation}

Applying the estimate \eqref{eq:T-sm2} to the above transport equation, we get
\begin{equation}\label{es.pelXp2X}
\begin{split}
  \|\partial_\WW^\ell \nabla^2\WW\|_{L^\infty_t (B^{\gamma-1}_{\infty,\infty})}
  &\lesssim \|  \partial_{\WW_0}^\ell\nabla^2 \WW_0\|_{B^{\gamma-1}_{\infty,\infty}}
  + \int_0^t \|\nabla v(\tau)\|_{L^\infty} \| \partial_\WW^\ell\nabla^2 \WW(\tau)\|_{B^{\gamma-1}_{\infty,\infty}} \dd \tau  \\
  & \quad +  \sum_{i=1}^5\int_0^t \|F_{\ell,i}\|_{B^{\gamma-1}_{\infty,\infty}} \dd \tau .
\end{split}
\end{equation}
For the initial data, since $\mathcal{W}_0=\{W^i_0\}_{1\leq i\leq 5}$
belongs to $C^{k-1,\gamma}(\R^3)$, then, following the same idea as the proof of \eqref{ext.thet}, we get that
\begin{eqnarray}\label{es.Fel0}
  \|\partial_{\WW_0}^\ell\nabla^2 \WW_0\|_{B^{\gamma-1}_{\infty,\infty}}
 \lesssim \| \WW_0\|_{B^{\ell+1+\gamma}_{\infty,\infty}}
 \lesssim \|\mathcal{W}_0\|_{C^{k-1,\gamma}} <\infty,
\end{eqnarray}
where the hidden constant $C$ depends on $\|\WW_0\|_{W^{\ell-1,\infty}}$. Using Lemmas \ref{lem:prod-es2} and \ref{lem:Bes-equ}, we find that, for $i=2,4$
\begin{align}\label{es.Fel.2.4}
  \|F_{\ell,i}\|_{L^1_t(B^{\gamma-1}_{\infty,\infty})}
  &\lesssim \int_0^t \Big(\|\partial_\WW^\ell(\nabla \WW\cdot\nabla^2 v)\|_{B^{\gamma-1}_{\infty,\infty}} + \|\partial_\WW^\ell(\nabla^2 v\cdot\nabla \WW)\|_{B^{\gamma-1}_{\infty,\infty}} \Big) \dd \tau \nonumber \\
  &\lesssim \int_0^t \Big( \|\nabla \WW\cdot\nabla^2v\|_{\bb^{\gamma-1,\ell}_{\infty,\WW} }+ \|\nabla^2v\cdot\nabla \WW\|_{\bb^{\gamma-1,\ell}_{\infty,\WW}} \Big) \dd \tau\nonumber\\
  &\lesssim \int_0^t \|\nabla\WW \|_{\bb^{0,\ell}_\WW}\|\nabla^2v\|_{\bb^{\gamma-1,\ell}_{\infty,\WW}} \dd \tau \nonumber \\
  &\lesssim \int_0^t \|\WW(\tau)\|_{\bb^{\gamma+1,\ell}_{\infty,\WW}} \|\nabla v(\tau)\|_{\bb^{\gamma,\ell}_{\infty,\WW}}\dd \tau,
\end{align}
and analogously, for $i=3,5$
\begin{align}\label{es.Fel.3.5}
  \|F_{\ell,i}\|_{L^1_t(B^{\gamma-1}_{\infty,\infty})}
  \lesssim  \int_0^t \|\nabla^2\WW\|_{\bb^{\gamma-1,\ell}_{\infty,\WW}} \| \nabla v
  \|_{\bb^{0,\ell}_\WW}\dd \tau
  \lesssim  \int_0^t \|\WW(\tau)\|_{\bb^{\gamma+1,\ell}_{\infty,\WW}} \|\nabla v(\tau) \|_{\bb^{\gamma,\ell}_{\infty,\WW}}\dd \tau.
\end{align}

It remains to get a control of the term $F_{\ell,1}$. Since we have \eqref{eq.nab2v}, we see that
\begin{equation}\label{es.peln2v}
\begin{split}
  \| \partial_\WW^{\ell+1} \nabla^2 v\|_{L^1_t(B^{\gamma-1}_{\infty,\infty})} & \lesssim
  \| \partial_\WW^{\ell+1}\nabla^2 \Lambda^{-2}\nabla\wedge \Gamma\|_{L^1_t(B^{\gamma-1}_{\infty,\infty})} \\
  &\quad + \| \partial_\WW^{\ell+1}\nabla^3\partial_3  \Lambda^{-4} \theta \|_{L^1_t(B^{\gamma-1}_{\infty,\infty})}
  + \| \partial_\WW^{\ell+1}\nabla^2 \Lambda^{-2} \theta \|_{L^1_t(B^{\gamma-1}_{\infty,\infty})}.
\end{split}
\end{equation}
For the last two terms in \eqref{es.peln2v}, by applying the estimate \eqref{eq:prod-es2-2} with $m(D)=\nabla^3\partial_3  \Lambda^{-4}$ or $m(D)=\nabla^2 \Lambda^{-2}$, one gets
\begin{eqnarray}\label{pXel.th}
\| \partial_\WW^{\ell+1}\nabla^3\partial_3  \Lambda^{-4} \theta\|_{L^1_t(B^{\gamma-1}_{\infty,\infty})}
  + \|\partial_\WW^{\ell+1}\nabla^2 \Lambda^{-2} \theta \|_{L^1_t(B^{\gamma-1}_{\infty,\infty})} \nonumber 
  &\leq&  \| \nabla^3\partial_3  \Lambda^{-4} \theta\|_{L^1_t(\bb^{\gamma-1,\ell+1}_{\infty,\WW})} + \|\nabla^2 \Lambda^{-2} \theta \|_{L^1_t(\bb^{\gamma-1,\ell+1}_{\infty,\WW})}\nonumber\\
  &\lesssim& \int_0^t \big(1+\|\WW(\tau)\|_{\bb^{1,\ell}_\WW}\big)
  \big( \|\theta (\tau)\|_{\bb^{\gamma-1,\ell}_{\infty,\WW}} + \|\theta(\tau)\|_{L^2}\big) \dd \tau \nonumber  \\
  && \   + \ \|\theta\|_{L^1_t(\bb^{\gamma-1,\ell+1}_{\infty,\WW})} 
\end{eqnarray}
Recalling that $[\partial_\WW, \partial_t+v\cdot\nabla]=0$, we see that, for all $j\in\{0,1,\cdots,\ell+1\}$, $\partial_\WW^j \theta$ satisfies that
\begin{align}\label{eq.Pxelth}
  \partial_t\partial_\WW^j \theta + v\cdot\nabla\partial_\WW^j \theta = 0, \quad \partial_\WW^j \theta|_{t=0}=\partial_{\WW_0}^j \theta_0.
\end{align}
Hence, Lemma \ref{lem:str-reg} and \eqref{eq:T-sm3}, \eqref{es.v.Lip} allow us to get that, for all $j\in \{0,1,\cdots,\ell+1\}$,
\begin{align*}
  \|\partial_{\WW}^j \theta\|_{L^\infty_t(B^{\gamma-1}_{\infty,\infty})}\lesssim e^{\|\nabla v\|_{L^1_t(L^\infty)}}
  \|\partial_{\WW_0}^j \theta_0\|_{B^{\gamma-1}_{\infty,\infty}}\lesssim e^{\exp\{ C E(T)\}} ,
 \end{align*}
and then 
\begin{align}\label{es.pXel.th2}
  \| \theta\|_{L^\infty_t(\bb^{\gamma-1,\ell+1}_{\infty,\WW})} \leq \sum_{j=0}^{\ell+1} \|\partial_{\WW}^j \theta \|_{B^{\gamma-1}_{\infty,\infty}}
  \lesssim e^{\exp\{C E(T)\}}.
\end{align}
Hence, by applying \eqref{es.pXel.th2}  into \eqref{pXel.th}, one obtains
\begin{align}\label{pXel.th2}
  \| (\partial_\WW^{\ell+1}\nabla^3\partial_3  \Lambda^{-4} \theta,\partial_\WW^{\ell+1}\nabla^2 \Lambda^{-2} \theta ) \|_{L^1_t(B^{\gamma-1}_{\infty,\infty})}
  \lesssim \Big(1+\int_0^t \|\WW(\tau)\|_{\bb^{1,\ell}_\WW}\dd \tau\Big),
 \end{align}
where the constant depends on $H_\ell(T)$. For the first term of the right-hand side in \eqref{es.peln2v}, it suffices to use  Lemma \ref{lem:prod-es2} to find that
\begin{eqnarray}\label{es.pXelGa0}
  \| \partial_\WW^{\ell+1}\nabla^2 \Lambda^{-2}\nabla\wedge \Gamma\|_{L^1_t(B^{\gamma-1}_{\infty,\infty})}
 & \leq& \| \nabla^2 \Lambda^{-2}(\nabla \wedge \Gamma)\|_{L^1_t(\bb^{\gamma-1,\ell+1}_{\infty,\WW})}  \nonumber\\
  &\lesssim& 
  \int_0^t \big( 1+\|\WW(\tau)\|_{\bb^{1,\ell}_\WW}\big)\big( \|\nabla \Gamma(\tau) \|_{\bb^{\gamma-1,\ell}_{\infty,\WW}} \|\Gamma(\tau)\|_{L^p}\big)\dd \tau \nonumber \\
  && \ +  \ \| \nabla \Gamma \|_{L^1_t(\bb^{\gamma-1,\ell+1}_{\infty,\WW})}
\end{eqnarray}
and by noticing that
\begin{align}\label{com.ell}
  [\nabla,\partial_\WW^\ell ]f=\sum_{j=0}^{\ell-1}\partial_\WW^j ([\nabla,\partial_\WW]\partial_\WW^{\ell-1-j}f)
  =\sum_{j=0}^{\ell-1}\partial_\WW^j (\nabla\WW\cdot\nabla\partial_\WW^{\ell-1-j}f),
\end{align}
we find, by using Lemma \ref{lem:Bes-equ}, that
\begin{align}\label{es:nab-Gam}
  \|\nabla \Gamma\|_{\bb^{\gamma-1,\ell+1}_{\infty,\WW}} & = \|\nabla \Gamma\|_{\bb^{\gamma-1,\ell}_{\infty,\WW}}
  +  \|\partial_\WW^{\ell+1} \nabla \Gamma\|_{B^{\gamma-1}_{\infty,\infty}} \nonumber \\
  & \lesssim \|\Gamma\|_{\bb^{\gamma,\ell}_{\infty,\WW}} +  \|\partial_\WW^{\ell+1}\Gamma\|_{B^{\gamma}_{\infty,\infty}}
  + \|[\nabla,\partial_\WW^{\ell+1}]\Gamma\|_{B^{\gamma-1}_{\infty,\infty}} \nonumber \\
  & \lesssim \|\Gamma\|_{\bb^{\gamma,\ell+1}_{\infty,\WW}}
  + \sum_{j=0}^\ell \|\nabla \WW \cdot \nabla \partial_\WW^{\ell-j} \Gamma \|_{\bb^{\gamma-1,j}_{\infty,\WW}} \nonumber \\
  & \lesssim \|\Gamma\|_{\bb^{\gamma,\ell+1}_{\infty,\WW}} + \sum_{j=1}^\ell \|\nabla \WW\|_{\bb^{0,j}_\WW}
  \|\partial_{\WW}^{\ell-j} \Gamma\|_{\bb^{\gamma,j}_{\infty,\WW}}
  \lesssim \|\Gamma\|_{\bb^{\gamma,\ell+1}_{\infty,\WW}} +  \|\WW\|_{\bb^{1,\ell}_\WW} \|\Gamma\|_{\bb^{\gamma,\ell}_{\infty,\WW}},
\end{align}
which combined with \eqref{es.pXelGa0} and \eqref{es:Gam.Lp} give, up to a constant which depends only on ${H_\ell(T)}$
\begin{eqnarray}\label{es.pXelGa}
  \| \partial_\WW^{\ell+1}\nabla^2 \Lambda^{-2}\nabla\wedge \Gamma\|_{L^1_t(B^{\gamma-1}_{\infty,\infty})}
  &\lesssim &   \int_0^t \big(1+\|\WW\|_{\bb^{1,\ell}_\WW} \big)\big( \| \Gamma \|_{\bb^{\gamma,\ell}_{\infty,\WW}}+1\big)\dd \tau  \nonumber \\
  && \ + \ \Gamma \|_{L^1_t(\bb^{\gamma,\ell+1}_{\infty,\WW})} 
\end{eqnarray}
Collecting all the above estimates \eqref{es.pelXp2X}-\eqref{es.pXelGa} allows us to write that, up to a constant which depends on $H_\ell(T)$, we have
\begin{eqnarray}\label{es.peln2v}
 \|\partial_\WW^\ell \nabla^2\WW\|_{L^\infty_t (B^{\gamma-1}_{\infty,\infty})}
  + \| \partial_\WW^{\ell+1} \nabla^2 v \|_{L^1_t(B^{\gamma-1}_{\infty,\infty})} 
  &\lesssim& \|\Gamma \|_{L^1_t(\bb^{\gamma,\ell+1}_{\infty,\WW})} \\
  &+& \int_0^t
  \big( \| \Gamma(\tau) \|_{\bb^{\gamma,\ell}_{\infty,\WW}} +  \|\nabla v(\tau) \|_{\bb^{\gamma,\ell}_{\infty,\WW}} +1\big)
  \|\WW(\tau)\|_{\bb^{\gamma+1,\ell}_{\infty,\WW}}\dd \tau + 1. \nonumber
\end{eqnarray}

Now, we focus on the term $ \partial_\WW^{\ell+1} \Gamma$. By using equation \eqref{eq.Gamm} with respect to $\Gamma$  and the identity $[\partial_\WW,\partial_t+v\cdot\nabla]=0$, we obtain that
\begin{align}\label{eq.pXel+1Ga}
  \partial_t (\partial_\WW^{\ell+1} \Gamma) + v\cdot\nabla (\partial_\WW^{\ell+1} \Gamma) - \Delta (\partial_\WW^{\ell+1} \Gamma)
  & = - [\Delta,\partial_\WW^{\ell+1} ]\Gamma +  \partial_\WW^{\ell+1} (\Omega\cdot \nabla v)
  + \partial_\WW^{\ell+1} ([\mathcal{R}_{-1},v\cdot\nabla]\theta)  \nonumber \\
  & := \sum_{i=1}^{3} J_{\ell+1,i}.
\end{align}
Thanks to the smoothing estimate \eqref{es.sm1}, we have that for all $\gamma'\in(0,\min\{\gamma,1-\frac{3}{p}\})$,
\begin{align}\label{esB.el.Ga}
  & \quad \|\partial_\WW^{\ell+1} \Gamma\|_{L^\infty_t (B^{\gamma'-1}_{\infty,1})}
  + \|\partial_\WW^{\ell+1} \Gamma\|_{L^2_t  (B^{\gamma'}_{\infty,1})} + \|\partial_\WW^{\ell+1} \Gamma\|_{L^1_t  (B^{\gamma'+1}_{\infty,1})}  \nonumber \\
  &\lesssim (1+t)\bigg(\|\partial_{\WW_0}^{\ell+1} \Gamma_0\|_{ B^{\gamma'-1}_{\infty,1}} + \sum_{i=1}^3 \|F_{\ell+1,i}\|_{L^1_t  (B^{\gamma'-1}_{\infty,1})}
  + \int_0^t \|\nabla v(\tau)\|_{L^\infty} \|\partial_\WW^{\ell+1} \Gamma(\tau)\|_{ B^{\gamma'-1}_{\infty,1}} \dd \tau \bigg).
\end{align}
As for the initial data, we use the relation $\Gamma_0=\Omega_0-\RR_{-1}\theta_0$,
together with \eqref{com.ell} and Lemmas \ref{lem:prod-es2}, \ref{lem:Bes-equ}, we obtain that for any $\gamma'\in(0,\min\{\gamma,1-\frac{3}{p}\})$,
\begin{eqnarray}\label{es.Fel.ini}
 \Vert \partial^{\ell+1}_{\WW_0}\Gamma_0 \Vert_{B^{\gamma'-1}_{\infty,1}}&\lesssim& \Vert \partial^{\ell+1}_{\WW_0}\nabla v_0 \Vert_{B^{\gamma'-1}_{\infty,1}}
  + \Vert \partial^{\ell+1}_{\WW_0}\RR_{-1}\theta_0 \Vert_{B^{\gamma'-1}_{\infty,1}} \nonumber\\
  &\lesssim& \Vert \nabla\partial^{\ell+1}_{\WW_0}v_0 \Vert_{B^{\gamma'-1}_{\infty,1}}
  + \sum_{j=0}^\ell \Vert\nabla \WW_0 \cdot \nabla \partial_{\WW_0}^{\ell-j} v_0 \Vert_{\bb^{\gamma'-1,j}_{\WW_0}}
  + \|\WW_0 \cdot \nabla\RR_{-1} \theta_0\|_{\bb^{\gamma'-1,\ell}_{\WW_0}} \nonumber \\
  &\lesssim& \Vert\partial^{\ell+1}_{\WW_0}v_0 \Vert_{W^{1,p}}+\sum_{j=0}^\ell \|\nabla \WW_0 \Vert_{\bb^{0,j}_{\WW_0}}
  \Vert \nabla \partial_{\WW_0}^{\ell-j} v_0 \Vert_{\bb^{\gamma'-1,j}_{\WW_0}}
  + \|\WW_0 \Vert_{\bb^{0,\ell}_{\WW_0}} \|\nabla\RR_{-1} \theta_0 \Vert_{\bb^{\gamma'-1,\ell}_{\WW_0}} \nonumber\\
  &\lesssim& \sum_{j=0}^{\ell+1} \Vert\partial^j_{\WW_0}v_0 \Vert_{W^{1,p}}
  + \|\WW_0 \Vert_{\bb^{0,\ell}_{\WW_0}} \big(1+ \Vert\WW_0 \Vert_{\bb^{1,\ell-1}_{\WW_0}} \big)
  \big(\Vert \theta_0 \Vert_{\bb^{\gamma'-1,\ell}_{\WW_0}}+\|\theta_0 \Vert_{L^2} \big) \nonumber \\
  &\lesssim& \Vert v_0 \Vert_{W^{\ell+2,p}} + \sum_{j=0}^\ell  \Vert\partial^{j}_{\WW_0}\theta_0 \Vert_{C^{-1,\gamma}} + \Vert\theta_0 \Vert_{L^2} < \infty,
\end{eqnarray}
where the constant in each step depends on $\|\WW_0\|_{\bb^{\gamma+1,\ell}_{\infty,\WW_0}}$ (which is control by $ \|\WW_0\|_{C^{k-1,\gamma}})$
and  where in the last line we have used the estimate $\displaystyle \sum_{j=0}^{\ell+1} \|\partial^j_{\WW_0}v_0\|_{W^{1,p}} \lesssim C({\|\WW_0\|_{W^{\ell,\infty}}})
\|v_0\|_{W^{\ell+2,p}}$.
Then, by noticing that
\begin{align}\label{eq.com.Ga}
  J_{\ell+1,1} &= -[\Delta,\partial_\WW]\partial_\WW^\ell \Gamma - \partial_\WW [\Delta, \partial_\WW^\ell] \Gamma  \nonumber\\
  &= -[\Delta, \partial_\WW] \partial_\WW^\ell \Gamma - \partial_\WW([\Delta,\partial_\WW] \partial_\WW^{\ell-1} \Gamma)-\partial_\WW^2 ([\Delta,\partial_\WW^{\ell-1}])\Gamma )\nonumber\\
  &= -\sum_{j=0}^\ell \partial_\WW^j\big( [\Delta,\partial_\WW]\partial_\WW^{\ell-j} \Gamma\big) \nonumber \\
  &= -\sum_{j=0}^\ell \partial_\WW^j \big(\Delta \WW \cdot \nabla \partial_\WW^{\ell-j} \Gamma \big) -
  \sum_{j=0}^\ell \partial_\WW^j \big(2 \nabla \WW \cdot \nabla^2 \partial_\WW^{\ell-j} \Gamma \big),
\end{align}
and by  Lemmas \ref{lem:prod-es2} and \ref{lem:Bes-equ} again, we get that
\begin{eqnarray}\label{es.Fel1a}
  \| J_{\ell+1,1}\|_{B^{\gamma'-1}_{\infty,1}}
  &\lesssim& \sum_{j=0}^\ell \|\Delta \WW \cdot \nabla \partial_\WW^{\ell-j}  \Gamma \|_{\bb^{\gamma'-1,j}_\WW}
  + \sum_{j=0}^\ell \| \nabla \WW \cdot \nabla^2 \partial_\WW^{\ell-j}  \Gamma \|_{\bb^{\gamma'-1,j}_\WW} \nonumber \\
  &\lesssim& \sum_{j=0}^\ell \Big( \|\Delta \WW\|_{\bb^{\gamma'-1,j}_\WW} \|\nabla \partial_\WW^{\ell-j} \Gamma\|_{\bb^{0,j}_\WW}
  + \sum_{i=0}^j\|\nabla \WW\|_{\bb^{0,i}_\WW} \|\nabla^2 \partial_\WW^{\ell-j}  \Gamma\|_{\bb^{\gamma'-1,j-i}_\WW}\Big) \nonumber \\
  &\lesssim& \sum_{j=0}^\ell  \|\WW\|_{\bb^{\gamma'+1,j}_\WW} \|\Gamma\|_{\bb^{1,\ell}_\WW}
  + \sum_{j=0}^\ell\sum_{i=0}^j\| \WW\|_{\bb^{1,i}_\WW} \| \partial_\WW^{\ell-j}  \Gamma\|_{\widetilde{\bb}^{\gamma'+1,j-i}_\WW} \nonumber \\
  &\lesssim& \|\WW\|_{\bb^{\gamma'+1,\ell}_\WW} \| \Gamma\|_{\bb^{\gamma'+1,\ell}_\WW},
\end{eqnarray}
where in the last line we have used the following estimate
\begin{eqnarray*}
  \sum_{j=0}^\ell \sum_{i=0}^j\| \WW\|_{\bb^{1,i}_\WW} \|\partial_\WW^{\ell-j} \Gamma\|_{\widetilde{\bb}^{\gamma'+1,j-i}_\WW}
  &\lesssim& \sum_{j=0}^\ell \sum_{i=0}^j\| \WW\|_{\bb^{1,i}_\WW} \Big( \|\partial_\WW^{\ell-j} \Gamma\|_{{\bb}^{\gamma'+1,j-i}_\WW}+\|\partial_\WW^{\ell-j} \Gamma\|_{{\bb}^{1,j-i}_\WW}\| \WW\|_{{\bb}^{\gamma'+1,j-i}_\WW}\Big)\\
  &\lesssim& \sum_{i=0}^\ell  \| \WW\|_{\bb^{1,i}_\WW} \Big( \| \Gamma\|_{{\bb}^{\gamma'+1,\ell-i}_\WW}
  +\| \Gamma\|_{{\bb}^{1,\ell-i}_\WW}\| \WW\|_{{\bb}^{\gamma'+1,\ell-i}_\WW}\Big)\\
  &\lesssim& \| \WW\|_{\bb^{1,\ell}_\WW} \| \Gamma\|_{{\bb}^{\gamma'+1,0}_\WW} \big(1+ \| \WW\|_{{\bb}^{\gamma'+1,0}_\WW}\big)\\
  && \ + \ \| \WW\|_{\bb^{1,\ell-1}_\WW} \big(\| \Gamma\|_{{\bb}^{\gamma'+1,\ell}_\WW} + \| \Gamma\|_{{\bb}^{1,\ell}_\WW}\| \WW\|_{{\bb}^{\gamma'+1,\ell}_\WW}\big) \\
  &\lesssim& \| \Gamma\|_{{\bb}^{\gamma'+1,\ell}_\WW} \| \WW\|_{{\bb}^{\gamma'+1,\ell}_\WW},
\end{eqnarray*}
where the constant depends  on $\| \WW\|_{L^\infty_T(\bb^{1+\gamma,\ell-1}_{\infty,\WW})}$ which is control (up to a constant) by $H_\ell(T)$.
It follows from \eqref{es.Fel1a} that
\begin{align}\label{es:Fel1b}
  \|J_{\ell+1,1}\|_{L^1_t(B^{\gamma'-1 }_{\infty,1})} \lesssim \int_0^t \|\WW(\tau)\|_{\bb^{\gamma'+1,\ell}_\WW} \| \Gamma(\tau)\|_{\bb^{\gamma'+1,\ell}_\WW} \dd \tau.
\end{align}
Regarding the term $J_{\ell+1,2}$, using the fact that ${B^{\gamma'}_{\infty,1}}$ is a Banach algebra, together with H\"older's inequality, we see that
\begin{eqnarray}\label{es.Fe112}
  \|  J_{\ell+1,2}\|_{L^1_t(B^{\gamma'-1 }_{\infty,1})}
  &\lesssim& \|\WW\|_{L^\infty_t(L^\infty)} \int_0^t \|\partial_\WW^\ell (\Omega \cdot \nabla v)\|_{B^{\gamma' }_{\infty,1}} \dd \tau \nonumber \\
  &\lesssim& \|\WW\|_{L^\infty_t(L^\infty)} \sum_{j=0}^\ell\sum_{i=0}^j
  \int_0^t \|\partial_\WW^i \Omega\cdot \partial_\WW^{j-i}  \nabla v \|_{B^{\gamma' }_{\infty,1}}\dd\tau\nonumber\\
  &\lesssim& \sum_{j=0}^\ell\sum_{i=0}^j \int_0^t \|\partial_\WW^i \Omega(\tau)\|_{B^{\gamma' }_{\infty,1}}
  \| \partial_\WW^{j-i}  \nabla v(\tau) \|_{B^{\gamma' }_{\infty,1}}\dd\tau\nonumber\\
  &\lesssim& \Big(\sum_{j=0}^\ell \| \partial_\WW^{j}  \Omega \|_{L^2_t (B^{\gamma }_{\infty,1})}\Big)
  \Big(\sum_{j=0}^\ell \| \partial_\WW^{j}  \nabla v \|_{L^2_t (B^{\gamma' }_{\infty,1})}\Big) \nonumber\\
  &\lesssim&  \|\nabla v\|^2_{L^2_t(\bb^{\gamma',\ell }_\WW)} \nonumber \\
  &\leq& C({H_\ell(T)}) .
\end{eqnarray}
Note that $\nabla([\RR_{-1},v\cdot\nabla]\theta)=[\nabla\RR_{-1},v\cdot\nabla]\theta-\nabla v\cdot\nabla \RR_{-1}\theta,$ where $\RR_{-1}$ is given by \eqref{eq:op-R-1} (which a sum of pseudo-differential operators of order -1) and that
\begin{align*}
  [\nabla\RR_{-1},v\cdot\nabla]\theta : =
  \big([\nabla\mathcal{R}_{-1,2},v\cdot\nabla]\theta,  -[\nabla\mathcal{R}_{-1,1},v\cdot\nabla]\theta,0\big)^t .
\end{align*}
Therefore, one obtains
\begin{eqnarray}\label{es.Fel3}
  \| J_{\ell+1,3}\|_{L^1_t(B^{\gamma'-1 }_{\infty,1})}
  &\lesssim& \| \WW\|_{L^\infty_t(L^\infty)}\| \nabla\partial_\WW^\ell ([\RR_{-1},v\cdot\nabla]\theta) \|_{L^1_t(B^{\gamma'-1 }_{\infty,1})} \nonumber \\
  &\lesssim& \|\partial_\WW^\ell ([ \nabla\RR_{-1},v\cdot\nabla]\theta) \|_{L^1_t(B^{\gamma'-1 }_{\infty ,1})}
  + \|\partial_\WW^\ell (\nabla v\cdot\nabla \RR_{-1}\theta) \|_{L^1_t(B^{\gamma'-1 }_{\infty,1})}\nonumber \\
  && \ +\ \|[\nabla,\partial_\WW^\ell ]([ \nabla\RR_{-1},v\cdot\nabla]\theta) \|_{L^1_t(B^{\gamma'-1 }_{\infty,1})}\nonumber \\
  &:=& G_1 +  G_2 + G_3.
\end{eqnarray}
Using Lemma \ref{lem:prod-es2}, \eqref{es.v.Lip}, \eqref{es.pXel.th2} and the inductive assumption \eqref{assup.el}, one finds that
\begin{eqnarray}\label{es:G1}
   \|G_1\|_{L^1_t(B^{\gamma'-1 }_{\infty,1})} &\leq& \|[\nabla \RR_{-1},v\cdot\nabla]\theta\|_{L^1_t(\bb^{\gamma'-1,\ell }_{\WW})}\nonumber \\
  &\lesssim& \big(\|\nabla v\|_{L^1_t(\bb^{0,\ell }_\WW)} + \|v\|_{L^1_t(L^\infty)}\big) \|\theta \|_{L^\infty_t(\bb^{\gamma'-1,\ell }_\WW)} \nonumber \\
  &\lesssim& \big(\|\nabla v\|_{L^1_t(\bb^{\gamma,\ell }_{\infty,\WW})} + \|v\|_{L^1_t(L^\infty)}\big)
  \|\theta \|_{L^\infty_t(\bb^{\gamma-1,\ell }_{\infty,\WW})} \nonumber \\
   &\leq& C({H_\ell(T)}),
\end{eqnarray}
and
\begin{eqnarray}\label{es:G2}
   \|G_2\|_{L^1_t (B^{\gamma'-1 }_{\infty,1})} &\leq& \|\nabla v\cdot \nabla \RR_{-1} \theta\|_{L^1_t(\bb^{\gamma'-1,\ell }_{\WW})} \nonumber \\
   &\lesssim&  \|\nabla v\|_{L^1_t(\bb^{0,\ell }_{\WW})} \|\nabla\RR_{-1}\theta \|_{L^\infty_t(\bb^{\gamma'-1,\ell }_\WW)} \nonumber \\
   &\lesssim& \|\nabla v\|_{L^1_t (\bb^{\gamma,\ell }_{\infty,\WW})}  \big(1+\|\WW\|_{L^\infty_t(\bb^{1,\ell-1 }_\WW)} \big)
  \big(\|\theta \|_{L^\infty_t(\bb^{\gamma'-1,\ell }_\WW)} + \|\theta\|_{L^\infty_t (L^2)}\big) \nonumber \\
  &\leq& C({H_\ell(T)}),
\end{eqnarray}
Since we have \eqref{com.ell} and by Lemma \ref{lem:prod-es2}, we deduce that
\begin{eqnarray}\label{eq:G3-1}
  \|G_3 \|_{L^1_t(B^{\gamma'-1 }_{\infty,1})}
  &\lesssim& \sum_{j=0}^{\ell-1} \|\nabla\WW\cdot\nabla\partial_\WW^{\ell-1-j}[ \RR_{-1},v\cdot\nabla]\theta \|_{L^1_t(\bb^{\gamma'-1,j }_\WW)}\nonumber \\
  &\lesssim& \sum_{j=0}^{\ell-1} \|\nabla\WW\|_{L^\infty_t(\bb^{0,j }_\WW)}
  \|\nabla\partial_\WW^{\ell-1-j}[ \RR_{-1},v\cdot\nabla]\theta \|_{L^1_t(\bb^{\gamma'-1,j}_\WW)}\nonumber \\
  &\lesssim& \|\nabla\WW\|_{L^\infty_t(\bb^{1,\ell-1 }_\WW)}\sum_{j=0}^{\ell-1}
  \|\nabla\partial_\WW^{\ell-1-j}[ \RR_{-1},v\cdot\nabla]\theta \|_{L^1_t(\bb^{\gamma'-1,j }_\WW)}\nonumber \\
  &\lesssim& \|\nabla\WW\|_{L^\infty_t(\bb^{1,\ell-1 }_\WW)}
  \Big( \|\nabla\big([ \RR_{-1},v\cdot\nabla]\theta \big)\|_{L^1_t(\bb^{\gamma'-1,\ell-1 }_\WW)}\nonumber \\
  && \ + \sum_{j=0}^{\ell-2} \|[\nabla,\partial_\WW^{\ell-1-j}]\big([ \RR_{-1},v\cdot\nabla]\theta \big)\|_{L^1_t(\bb^{\gamma'-1,j }_\WW)}\Big).
\end{eqnarray}
Then {\it{via}} \eqref{com.ell} and the estimates \eqref{es:G1}-\eqref{es:G2}, we get that, up to a constant which depends only on $H_\ell(T)$, one has
\begin{eqnarray}
  \|G_3 \|_{L^1_t(B^{\gamma'-1 }_{\infty,1})}
  &\lesssim& \|[\nabla \RR_{-1},v\cdot\nabla]\theta\|_{L^1_t(\bb^{\gamma'-1,\ell-1 }_{\WW})}
  +  \|\nabla v\cdot \nabla \RR_{-1} \theta\|_{L^1_t(\bb^{\gamma'-1,\ell-1 }_{\WW})} \nonumber \\
  && \ + \sum_{j=0}^{\ell-2}\sum_{i=0}^{\ell-2-j}
  \|\partial_\WW^{i}\big(\nabla\WW\cdot\nabla \partial_\WW^{\ell-2-j-i}([ \RR_{-1},v\cdot\nabla]\theta )\big) \|_{L^1_t(\bb^{\gamma'-1,j }_\WW)} \nonumber \\
  &\lesssim& 1+\sum_{0\leq j+i\leq\ell-2}
  \|\nabla\WW\cdot\nabla \partial_\WW^{\ell-2-j-i}([ \RR_{-1},v\cdot\nabla]\theta ) \|_{L^1_t(\bb^{\gamma'-1,j+i }_\WW)} \nonumber \\
  &\lesssim& 1 + \|\nabla \WW\|_{L^\infty_t (\bb^{0,\ell-2 }_\WW)}
  \Big( \|\nabla\big([ \RR_{-1},v\cdot\nabla]\theta \big)\|_{L^1_t(\bb^{\gamma'-1,\ell-2 }_\WW)} \nonumber \\
  && \ + \sum_{0\leq j+i\leq\ell-2} \|[\nabla,\partial_\WW^{\ell-2-j-i}] \big([ \RR_{-1},v\cdot\nabla]\theta \big)\|_{L^1_t(\bb^{\gamma'-1,j+i }_\WW)}\Big). \nonumber
\end{eqnarray}
By repeating the above process we find that
\begin{align}\label{eq:G3}
  \|G_3\|_{L^1_t(B^{\gamma'-1}_{\infty,1})}\leq C (H_\ell(T)).
\end{align}
Hence collecting the above estimates \eqref{esB.el.Ga}, \eqref{es.Fel.ini}, \eqref{es:Fel1b}-\eqref{es.Fe112}, \eqref{es:G1}-\eqref{es:G2} and \eqref{eq:G3}, we conclude that
\begin{eqnarray}\label{esB.el.Ga}
  \|\Gamma(t)\|_{\bb^{\gamma'-1,\ell+1 }_\WW} &+& \|\Gamma\|_{L^2_t  (\bb^{\gamma',\ell+1 }_\WW)} +\|\Gamma\|_{L^1_t  (\bb^{\gamma'+1,\ell+1 }_\WW)} =  \|\partial_\WW^{\ell+1} \Gamma\|_{B^{\gamma'-1 }_{\infty,1}} + \|\Gamma\|_{ \bb^{\gamma'-1,\ell }_\WW}
  + \|\partial_\WW^{\ell+1} \Gamma\|_{L^2_t  (B^{\gamma' }_{\infty,1})} \nonumber \\
  && \hspace{5cm}  + \ \ \|\Gamma\|_{L^2_t  (\bb^{\gamma',\ell }_\WW)} + \|\partial_\WW^{\ell+1} \Gamma\|_{L^1_t  (B^{\gamma'+1 }_{\infty,1})}+  \|\Gamma\|_{L^1_t  (\bb^{\gamma'+1,\ell }_\WW)} \nonumber\\
  && \hspace{-1cm} \lesssim 1 + \ \int_0^t \big(\|\nabla v(\tau)\|_{L^\infty} + \|\Gamma(\tau)\|_{\bb^{\gamma'+1,\ell }_\WW}\big)
  \big(\|\WW(\tau)\|_{\bb^{\gamma'+1,\ell }_{\infty,\WW}} + \ \|\Gamma(\tau)\|_{\bb^{\gamma'-1,\ell+1 }_\WW} \big) \dd \tau,
\end{eqnarray}
where the hidden constant depends on $H_\ell(T)$. Then, since $\bb^{\gamma'+1,\ell+1}_\WW \subset \bb^{\gamma,\ell+1}_{\infty,\WW} $,
by putting \eqref{esB.el.Ga} into \eqref{es.peln2v}, we obtain that for all $\gamma'\in (0,\min\{\gamma,1-3/p\})$,
\begin{eqnarray}\label{esB.el.X}
  && \|\Gamma(t)\|_{\bb^{\gamma'-1,\ell+1}_\WW} + \|\Gamma\|_{L^2_t  (\bb^{\gamma',\ell+1 }_\WW)} + \|\Gamma\|_{L^1_t  (\bb^{\gamma'+1,\ell+1 }_\WW)}
   + \|\partial_\WW^{\ell+1} \nabla^2 v\|_{L^1_t  (B^{\gamma'+1 }_{\infty,\infty})} + \|\partial_\WW^\ell  \nabla^2\WW(t)\|_{B^{\gamma'-1 }_{\infty,\infty}}\nonumber \\
  &\lesssim& 1 +  \int_0^t \big(\|\nabla v(\tau)\|_{\bb^{\gamma,\ell }_{\infty,\WW}}+ \|\Gamma(\tau)\|_{\bb^{\gamma'+1,\ell }_\WW}+1\big)
  \big( \|\WW(\tau)\|_{\bb^{\gamma+1,\ell }_{\infty,\WW}} + \|\Gamma(\tau)\|_{\bb^{\gamma'-1,\ell+1 }_{\WW}} \big) \dd \tau,
\end{eqnarray}
where the constant depends only on ${H_\ell(T)}$. Following the same step as the proof of \eqref{eq.com.Ga}, \eqref{com.ell} and using Lemmas \ref{lem:prod-es2}, \ref{lem:Bes-equ}, we find that
\begin{eqnarray}
   \|[\nabla^2,\partial_\WW^\ell  ]\WW\|_{L^\infty_t (B^{\gamma'-1 }_{\infty,\infty})}
   &\lesssim& \sum_{j=0}^{\ell-1} \|\nabla^2 \WW \cdot \nabla \partial_\WW^{\ell-1-j}  \WW \|_{L^\infty_t(\bb^{\gamma-1,j }_{\infty,\WW})}
   + \sum_{j=0}^{\ell -1}\| \nabla \WW \cdot \nabla^2 \partial_\WW^{\ell-1-j} \WW\|_{L^\infty_t(\bb^{\gamma-1,j }_{\infty,\WW})} \nonumber \\
  &\lesssim& \sum_{j=0}^{\ell-1} \|\nabla^2 \WW \|_{L^\infty_t(\bb^{\gamma-1,j }_{\infty,\WW})}
  \|\nabla \partial_\WW^{\ell-1-j}  \WW \|_{L^\infty_t(\bb^{0,j }_{\WW})} \nonumber \\
  && \ +  \ \sum_{j=0}^{\ell -1}\| \nabla \WW\|_{L^\infty_t(\bb^{0,j }_{\WW})} \| \nabla^2 \partial_\WW^{\ell-1-j} \WW\|_{L^\infty_t(\bb^{\gamma-1,j }_{\infty,\WW})}\nonumber \\
  &\lesssim& \|\WW\|_{L^\infty_t(\bb^{\gamma+1,\ell-1 }_{\infty,\WW})} \| \WW\|_{L^\infty_t(\bb^{1,\ell-1 }_{\WW})} \nonumber \\
  &\leq& C(H_\ell(T)),
\end{eqnarray}
and
\begin{align}
  \|[\nabla,\partial_\WW^{\ell+1} ]\nabla v\|_{L^1_t (B^{\gamma'-1 }_{\infty,\infty})}
  \lesssim \sum_{j=0}^\ell  \|\nabla \WW \cdot \nabla \partial_\WW^{\ell-j}  \nabla v \|_{L^1_t(\bb^{\gamma-1,j }_{\infty,\WW})}
  \lesssim \int_0^t\|\WW(\tau)\|_{\bb^{1,\ell}_\WW} \| \nabla v(\tau)\|_{\bb^{\gamma,\ell }_{\infty,\WW}}\dd\tau, \nonumber
\end{align}
Thus,  by using Lemma \ref{lemp:Bes-prop} together with the frequency decomposition (high/low), we see that
\begin{eqnarray}\label{com.eln2.X1}
   \| \WW \|_{L^\infty_t(\bb^{\gamma+1,\ell }_{\infty,\WW})}
  &=& \| \partial_\WW^\ell  \WW\|_{L^\infty_t(B^{\gamma+1 }_{\infty,\infty})} + \|\WW\|_{L^\infty_t(\bb^{\gamma+1,\ell-1 }_{\infty,\WW})} \nonumber \\
  &\lesssim&  \|\nabla^2\partial_\WW^\ell   \WW \|_{L^\infty_t(B^{\gamma-1 }_{\infty,\infty})}
  + \|\Delta_{-1}\partial_\WW^\ell   \WW \|_{L^\infty_t(B^{\gamma-1 }_{\infty,\infty})} + H_\ell(T) \nonumber \\
  &\lesssim&  \|\partial_\WW^\ell  \nabla^2 \WW \|_{L^\infty_t(B^{\gamma-1 }_{\infty,\infty})} +
  \|[\nabla^2,\partial_\WW^\ell ]\WW \|_{L^\infty_t(B^{\gamma-1 }_{\infty,\infty})}
  + \|\WW\|_{L^\infty_t(L^\infty)} \|\partial_\WW^{\ell-1}  \WW \|_{L^\infty_t(B^{\gamma }_{\infty,\infty})} \nonumber\\
  && \ + \ H_\ell(T)\nonumber \\
  &\lesssim&  \|\partial_\WW^\ell  \nabla^2 \WW \|_{L^\infty_t(B^{\gamma-1 }_{\infty,\infty})} + 1
\end{eqnarray}
where the constant depends on $H_\ell(T)$ in the last inequality. Moreover, we have that
\begin{eqnarray}\label{com.elnv.1}
  \| \nabla v\|_{L^1_t(\bb^{\gamma,\ell+1 }_{\infty,\WW})} &=& \|\partial_\WW^{\ell+1} \nabla v\|_{L^1_t(B^{\gamma }_{\infty,\infty})} +\| \nabla v\|_{L^1_t(\bb^{\gamma,\ell }_{\infty,\WW})}\nonumber\\
  &\lesssim& \|\partial_\WW^{\ell+1} \nabla^2 v\|_{L^1_t(B^{\gamma-1 }_{\infty,\infty})}+
  \|[\nabla,\partial_\WW^{\ell+1} ]   \nabla v \|_{L^1_t(B^{\gamma-1 }_{\infty,\infty})}+ \|\Delta_{-1}\partial_\WW^{\ell+1} \nabla v\|_{L^1_t(B^{\gamma-1 }_{\infty,\infty})}+ H_\ell(T)\nonumber \\
  &\lesssim& \|\partial_\WW^{\ell+1} \nabla^2 v\|_{L^1_t(B^{\gamma-1 }_{\infty,\infty})}
  + \int_0^t\|\WW(\tau)\|_{\bb^{1,\ell }_\WW} \| \nabla v(\tau)\|_{\bb^{\gamma,\ell }_{\infty,\WW}}\dd\tau + 1.
\end{eqnarray}
where again the constant depends on $H_\ell(T)$. Therefore, by \eqref{esB.el.X} and \eqref{com.eln2.X1}, \eqref{com.elnv.1}, one finds  that
\begin{eqnarray}
  && \|\Gamma(t)\|_{\bb^{\gamma'-1,\ell+1 }_\WW} + \| \WW(t) \|_{\bb^{\gamma+1,\ell }_{\infty,\WW}}
  + \|\Gamma\|_{L^2_t  (\bb^{\gamma',\ell+1 }_\WW)}   + \|\Gamma\|_{L^1_t  (\bb^{\gamma'+1,\ell+1 }_\WW)} + \| \nabla v\|_{L^1_t(\bb^{\gamma,\ell+1 }_{\infty,\WW})} \nonumber \\
  &\leq& C + C\int_0^t \big(\|\nabla v(\tau)\|_{\bb^{\gamma,\ell }_{\infty,\WW}}+\|\Gamma(\tau)\|_{\bb^{\gamma'+1,\ell }_\WW}+1\big)
  \big( \|\WW(\tau)\|_{\bb^{\gamma+1,\ell }_{\infty,\WW}} + \|\Gamma(\tau)\|_{\bb^{\gamma'-1,\ell+1 }_{\WW}} \big) \dd \tau ,\nonumber
\end{eqnarray}
where $C>0$ depends on $H_\ell(T)$.
Then, Gr\"onwall's inequality and the induction assumption \eqref{assup.el} allow us to get that
\begin{eqnarray}\label{es:BelX-a}
   \|\Gamma\|_{L^\infty_T(\bb^{\gamma'-1,\ell+1 }_\WW)}&+& \| \WW \|_{L^\infty_T(\bb^{\gamma+1,\ell }_{\infty,\WW})}
  + \|\Gamma\|_{L^2_T  (\bb^{\gamma',\ell+1 }_\WW)} + \|\Gamma\|_{L^1_T  (\bb^{\gamma'+1,\ell+1 }_\WW)}  + \| \nabla v\|_{L^1_T(\bb^{\gamma,\ell+1 }_{\infty,\WW})}\nonumber \\
 &\leq& C\exp\big\{C\big( T +\|\nabla v\|_{L^1_T(\bb^{\gamma,\ell }_{\infty,\WW})} + \|\Gamma\|_{L^1_T(\bb^{\gamma'+1,\ell }_\WW)}\big)\big\}\nonumber \\
  &\leq& C \exp\{CH_\ell(T)\} \nonumber \\
  &\leq& H_{\ell+1}(T).
\end{eqnarray}

Now, it remains to control the term $\|\nabla v\|_{L^2_T(\bb^{\gamma',\ell+1 }_\WW)}$.
By using low/high frequency decomposition, we get
\begin{eqnarray}\label{es.el.navg}
   \| \partial_\WW^{\ell+1}  \nabla v \|_{L^2_T(B^{\gamma'}_{\infty,1})}
   &\lesssim& \|\Delta_{-1} \partial_\WW^{\ell+1} \nabla v \|_{L^2_T(L^\infty)} + \|\nabla \partial_\WW^{\ell+1}  \nabla v \|_{L^2_T(B^{\gamma'-1 }_{\infty,1})} \\
  &\lesssim& \|\partial_\WW^\ell \nabla v\|_{L^2_T(B^{\gamma' }_{\infty,1})} +
  \|[\nabla, \partial_\WW^{\ell+1}] \nabla v\|_{L^2_T(B^{\gamma'-1 }_{\infty,1})} + \ \|\partial_\WW^{\ell+1}\nabla^2 v\|_{L^2_T(B^{\gamma'-1 }_{\infty,1})}. \nonumber
\end{eqnarray}
Thanks to \eqref{com.ell} and Lemma \ref{lem:prod-es2}, we deduce that
\begin{eqnarray}\label{es.el.navg2}
   \| [\nabla,\partial_\WW^{\ell+1}]  \nabla v \|_{L^2_T(B^{\gamma'-1 }_{\infty,1})}
  &\lesssim& \sum_{j=0}^\ell \|\nabla\WW\cdot\nabla\partial_\WW^{\ell-j}\nabla v\|_{L^2_T(\bb^{\gamma'-1,j }_\WW)} \nonumber\\
  &\lesssim&  \sum_{j=0}^\ell \|\WW\|_{L^\infty_T(\bb^{1,j }_\WW)} \|\nabla\partial_\WW^{\ell-j}\nabla v\|_{L^2_T(\bb^{\gamma'-1,j }_\WW)} \nonumber \\
  &\lesssim& \|\WW\|_{L^\infty_T(\bb^{1,\ell }_\WW)} \|\nabla v\|_{L^2_T(\bb^{\gamma',\ell }_\WW)}.
\end{eqnarray}
By following the same line as the proof of \eqref{es.peln2v}, it follows from \eqref{es:BelX-a} that
\begin{eqnarray}\label{es.el.navg3}
  \|\partial_\WW^{\ell+1}  \nabla^2 v \|_{L^2_T(B^{\gamma'-1 }_{\infty,1})} &\leq&  \|\nabla^2 v\|_{L^2_T(\bb^{\gamma'-1,\ell+1 }_\WW)} \nonumber\\
  &\lesssim& \|\nabla^2 \Lambda^{-2}\nabla\wedge \Gamma\|_{L^2_T(\bb^{\gamma'-1,\ell+1 }_\WW)} + \|(\nabla^3\partial_3  \Lambda^{-4} \theta ,\nabla^2 \Lambda^{-2} \theta) \|_{L^2_T(\bb^{\gamma'-1,\ell+1 }_\WW)}\nonumber\\
  &\lesssim& \big(1+\|\WW\|_{L^\infty_T(\bb^{1,\ell }_\WW)} \big) \big(\|\nabla \Gamma \|_{L^2_T(\bb^{\gamma'-1,\ell+1 }_\WW)} +   \| \theta\|_{L^2_T(\bb^{\gamma'-1,\ell+1 }_\WW)} +1 \big) \nonumber \\
  &\lesssim& \big(1+\|\WW\|_{L^\infty_T(\bb^{\gamma+1,\ell }_{\infty, \WW})} \big)
  \big(\|\Gamma \|_{L^2_T(\bb^{\gamma',\ell+1 }_\WW)} +   \| \theta\|_{L^2_T(\bb^{\gamma-1,\ell+1 }_{\infty,\WW})} +1 \big).
\end{eqnarray}
Collecting all the estimates from \eqref{es.el.navg} to \eqref{es.el.navg3}, and using \eqref{assup.el}, \eqref{es.pXel.th2} and \eqref{es:BelX-a},
we find that
\begin{eqnarray}\label{es.nvL2tg-a}
  \|\nabla v\|_{L^2_T(\bb^{\gamma',\ell+1 }_\WW)}&=& \|\partial_\WW^{\ell+1}\nabla v\|_{L^2_T(B^{\gamma' }_{\infty,1})}+\|\nabla v\|_{L^2_T(\bb^{\gamma',\ell }_\WW)}\nonumber\\
  &\lesssim& \big(\|\WW\|_{L^\infty_T(\bb^{1+\gamma,\ell }_{\infty,\WW})}+1\big)\big(
  \|\nabla v\|_{L^2_T(\bb^{\gamma',\ell }_\WW)}+ \| \Gamma \|_{L^2_T(\bb^{\gamma',\ell+1 }_\WW)} + 1 \big) \nonumber\\
  &\lesssim&   \exp\{CH_\ell(T)\} \nonumber \\
  &\leq&   H_{\ell+1}(T) .
\end{eqnarray}
Therefore, \eqref{es.nvL2tg-a} and \eqref{es:BelX-a} give the inequality \eqref{es.el+1}, as desired.
The induction method finally implies the wanted estimates \eqref{eq:Targ7} and \eqref{eq:targ6}, hence the proof is complete.
\qed
\section{Appendix}\label{sec:append}

In this section we give the proof of Lemmas \ref{lem:parW-mDf} and \ref{lem:L2Linf}.

\begin{proof}[Proof of Lemma \ref{lem:parW-mDf}]
  Bony's decomposition gives that
\begin{eqnarray}\label{eq:parW-mDf1}
  \|\partial_\mathcal{W} m(D) f\|_{B^{-\epsilon}_{p,r}} \leq \|(T_{\mathcal{W}\cdot \nabla}) m(D) f\|_{B^{-\epsilon}_{p,r}}
  + \| T_{\nabla m(D) f} \cdot \mathcal{W}  \|_{B^{-\epsilon}_{p,r}}
  + \|R(\mathcal{W}\cdot, \nabla m(D) f)\|_{B^{-\epsilon}_{p,r}}.
\end{eqnarray}
Thanks to the spectrum support property of the dyadic blocks together with the fact that $\nabla \Delta_{-1} m(D)$ is a bounded operator on $L^p$ with $1\leq p\leq \infty$,
we have that for all $q\geq -1$,
\begin{eqnarray*}
   2^{-q\epsilon}\| \Delta_q (T_{\nabla m(D)f}\cdot \mathcal{W})\|_{L^p}
  &\leq& 2^{-q\epsilon} \sum_{j\in \N,|j-q|\leq 4} \|\Delta_q (\Delta_j \mathcal{W}\cdot S_{j-1}\nabla m(D) f)\|_{L^p} \\
  &\lesssim&  2^{-q \epsilon} \sum_{j\in \N, |j-q|\leq 4} 2^{-j\gamma} \|\mathcal{W}\|_{C^\gamma} \Big(\|\Delta_{-1} \nabla m(D) f\|_{L^p}
  + \sum_{0\leq j'\leq j} \|\Delta_{j'} \nabla m(D)f\|_{L^p} \Big) \\
  & \lesssim&  c_q \|\mathcal{W}\|_{C^\gamma} \|f\|_{B^{1-\epsilon-\gamma}_{p,r}},
\end{eqnarray*}
Moreover, using the divergence-free property of $\mathcal{W}$, one has that
\begin{eqnarray*}
  2^{-q \epsilon} \|\Delta_q (R(\mathcal{W}\cdot,\nabla m(D)f))\|_{L^p} 
  &\leq& 2^{-q\epsilon} \sum_{j\geq \max\{q-3,2\}} \|\Delta_q \divg (\Delta_j \mathcal{W} \, \widetilde{\Delta}_j m(D)f)\|_{L^p} \\
  &&  \ + \ 2^{-q\epsilon} \sum_{-1\leq j\leq 1} 1_{\{-1\leq q\leq 5\}} \|\Delta_q (\Delta_j \mathcal{W} \cdot \widetilde{\Delta}_j \nabla m(D) f)\|_{L^p} \\
  &\lesssim&  \sum_{j\geq \max\{q-3,2\}} 2^{q(1-\epsilon)}  \|\Delta_j\mathcal{W}\|_{L^\infty} \|\widetilde{\Delta}_j m(D) f\|_{L^p} \\
  && \ + \  \sum_{-1\leq j\leq 1}  1_{\{-1\leq q\leq 5\}} \|\Delta_j \mathcal{W}\|_{L^\infty} \|\widetilde{\Delta}_j \nabla m(D) f\|_{L^p} \\
  &\lesssim&  \sum_{j\geq \max\{q-4,1\}} 2^{q(1-\epsilon)} 2^{-j\gamma} \|\mathcal{W}\|_{C^\gamma} \|\Delta_j f\|_{L^p} \\
  && \ + \    \sum_{-1\leq j\leq 2}  1_{\{-1\leq q\leq 5\}} \|\mathcal{W}\|_{L^\infty} \|\Delta_j f\|_{L^p} \\
  &\lesssim&  c_q \|\mathcal{W}\|_{C^\gamma} \|f\|_{B^{1-\epsilon-\gamma}_{p,r}},
\end{eqnarray*}
where $\{c_q\}_{q\geq -1}$ is such that $\|c_q\|_{\ell^r}=1$. Hence, we immediately get that
\begin{equation}\label{eq:parW-mDf2}
  \| T_{\nabla m(D) f} \cdot \mathcal{W}  \|_{B^{-\epsilon}_{p,r}}
  + \|R(\mathcal{W}\cdot, \nabla m(D) f)\|_{B^{-\epsilon}_{p,r}} \leq C \|\mathcal{W}\|_{C^\gamma} \|f\|_{B^{1-\epsilon-\gamma}_{p,r}}.
\end{equation}
Note that there exists a bump function $\widetilde{\phi} \in C^\infty_c(\R^d)$ supported on an annulus and
$\widetilde{h}= \mathcal{F}^{-1}(m \widetilde{\phi})\in \mathcal{S}(\R^d)$ so that
\begin{align*}
  (T_{\mathcal{W}\cdot\nabla})m(D)f  = \sum_{j\in \N} S_{j-1} \mathcal{W} \cdot \nabla \Delta_j m(D) f
  = - \sum_{j\in\N} [m(D)\widetilde{\phi}(2^{-j}D), S_{j-1}\mathcal{W}\cdot]\nabla \Delta_j f + m(D) (T_{\mathcal{W}\cdot\nabla}) f,
\end{align*}
where $m(D) \widetilde{\phi}(2^{-j} D) = 2^{jd} \widetilde{h}(2^j\cdot)*$, and
\begin{align*}
  [m(D)\widetilde{\psi}(2^{-j}D), S_{j-1}\mathcal{W}\cdot]\nabla \Delta_j f(x)
  = \int_{\R^d} \widetilde{h}(y)  \big(S_{j-1}\mathcal{W}(x+2^{-j}y) - S_{j-1}\mathcal{W}(x)\big) \cdot \nabla \Delta_j f(x+2^{-j}y) \dd y.
\end{align*}
Then, we have that for all $q\geq -1$,
\begin{eqnarray*}
   2^{-q\epsilon} \|\Delta_q (T_{\mathcal{W}\cdot \nabla})m(D) f\|_{L^p} 
  &\leq& 2^{-q\epsilon} \sum_{j\in \N, |j-q|\leq 4} \|\Delta_q \big([m(D)\widetilde{\phi}(2^{-j}D),
  S_{j-1}\mathcal{W}\cdot\nabla] \Delta_j f \big)\|_{L^p} \\
 && \ + \ 2^{-q\epsilon} \big\|\Delta_q m(D) (T_{\mathcal{W}\cdot\nabla}) f \big\|_{L^p} \\
  &\lesssim&  2^{-q\epsilon} \sum_{j\in\N,|j-q|\leq 4} 2^{-j} \|\nabla S_{j-1}\mathcal{W}\|_{L^\infty} \|\nabla \Delta_j f\|_{L^p}
  +  c_q \|m(D) (T_{\mathcal{W}\cdot\nabla}) f\|_{B^{-\epsilon}_{p,r}} \\
  &\lesssim&  2^{-q\epsilon} \sum_{j\in\N,|j-q|\leq 4} \Big(\sum_{-1\leq j'\leq j-1} 2^{j'(1-\gamma)} \|\mathcal{W}\|_{C^\gamma}\Big)
  \|\Delta_j f\|_{L^p} \\
  &&   \ + \    c_q \sum_{0\leq j\leq 3}\|\Delta_{-1}m(D) \divg \big(S_{j-1} \mathcal{W} \, \Delta_j f\big)\|_{L^p}
  +   c_q \|(T_{\mathcal{W}\cdot\nabla}) f\|_{B^{-\epsilon}_{p,r}} \\
  &\lesssim&  c_q \big( \|\mathcal{W}\|_{C^\gamma} \|f\|_{B^{1-\epsilon-\gamma}_{p,r}} + \|(T_{\mathcal{W}\cdot \nabla}) f\|_{B^{-\epsilon}_{p,r}} \big).
\end{eqnarray*}
Following the same approach as the proof of \eqref{eq:parW-mDf2}, the above inequality gives
\begin{align}\label{eq:parW-mDf3}
  \|(T_{\mathcal{W}\cdot\nabla})m(D) f \|_{B^{-\epsilon}_{p,r}} & \lesssim \|\mathcal{W}\|_{C^\gamma}
  \|f\|_{B^{1-\epsilon-\gamma}_{p,r}}   + \|\partial_\mathcal{W} f\|_{B^{-\epsilon}_{p,r}}
  + \|(\Id - T_{\mathcal{W}\cdot\nabla}) f\|_{B^{-\epsilon}_{p,r}} \nonumber \\
  & \lesssim   \|\mathcal{W}\|_{C^\gamma}  \|f\|_{B^{1-\epsilon-\gamma}_{p,r}}   +  \|\partial_\mathcal{W} f\|_{B^{-\epsilon}_{p,r}} ,
\end{align}
which together with \eqref{eq:parW-mDf2} gives the desired estimate \eqref{eq:par-W-mDf}.
\end{proof}

\begin{proof}[Proof of Lemma \ref{lem:L2Linf}]
We follow the same approach as Chemin and Gallagher \cite{CheG06} 

in their study of the 2D Navier-Stokes equation with an external force. More precisely, we first establish the following control from \eqref{eq:v-H1/2es}, that is,  for all $\rho\in[2,\infty]$,
\begin{equation}\label{es.tec}
  \|v\|_{\widetilde{L}^\rho_T(\dot H^{\frac{1}{2}+ \frac{2}{\rho}})}^2 = \sum_j 2^{j(1+\frac 4\rho)}\|\dot\Delta_jv\|_{L^\rho_T(L^2)}^2\lesssim E(T)^2. \quad 
\end{equation}
By noticing that the equation verified by $v$ may be viewed as a nonhomogeneous heat equation, that is
\begin{equation}\label{eq.v.int}
  \partial_t v-\Delta v=-\mathbb{P}\divg(v\otimes v)+\mathbb{P}(\theta e_3),\quad v(0,x) = v_0(x),
\end{equation}
and using the smoothing effect estimate of the heat flow \eqref{eq:heat-Besov} (with $s=\frac{1}{2}$, $\rho_1=2$, $p=r=2$), we get
\begin{align}\label{es.t1}
  \sum_j2^{j(1+\frac4\rho)}\|\dot\Delta_jv\|_{L^\rho_T(L^2)}^2\lesssim \|v_0\|^2_{\dot H^{\frac12}} + \int^T_0\|\mathbb{P}(v\cdot\nabla v, \theta e_3)(t)\|_{\dot H^{-\frac12}}^2 \dd t.
\end{align}
Then, an easy computation gives that
\begin{align}\label{eq:v-H-1/2}
  \|\mathbb{P}(v\cdot\nabla v,\theta)(t)\|_{\dot H^{-\frac12}} \lesssim \|v\cdot\nabla v\|_{L^{\frac32}}+\|\theta\|_{L^\frac32}
  \lesssim \|\nabla v\|^2_{L^2}+\|\theta\|_{L^\frac32} \lesssim \| v\|_{\dot H^{\frac12}} \| v\|_{\dot H^{\frac32}}+\|\theta_0\|_{L^\frac32},
\end{align}
from which  it follows that
\begin{align*}
  \int_0^T\|\mathbb{P}(v\cdot\nabla v, \theta)(t)\|^2_{\dot H^{-\frac{1}{2}}} \dd t
  \lesssim \| v\|^2_{L^\infty_T(\dot H^{\frac12})}  \| v\|^2_{L^2_T(\dot H^{\frac32})} + \|\theta_0\|^2_{L^\frac32}T.
\end{align*}
Hence using this control in \eqref{es.t1} and \eqref{eq:v-H1/2es} leads to the desired estimate \eqref{es.tec}. \\

Now, we split the solution $v$ of equation \eqref{eq.v.int} into $v= h + w$ where $h$ and $w$ verify
\begin{equation}\label{eqh.v.int}
  h_t-\Delta h=\mathbb{P}(\theta e_3),\quad h(0,x) = v_0(x),
\end{equation}
and
\begin{equation}\label{eqw.v.int}
  w_t-\Delta w=-\mathbb{P}\divg(v\otimes v),\quad w(0,x) = 0.
\end{equation}
Duhamel's formula gives
\begin{equation*}\label{eqh.v.int'}
  h(t)=e^{t\Delta} v_0+\int_0^te^{(t-\tau)\Delta}\mathbb{P}(\theta(\tau) e_3) \dd \tau,
\end{equation*}
therefore, we have
\begin{equation*}\label{esh.v.int'}
  \|h\|_{L^2_T (L^\infty)} \leq \|e^{t\Delta} v_0\|_{L^2(\R^+;L^\infty)} + \int_0^T  \|e^{t\Delta} \mathbb{P}\theta(\tau) \|_{L^2(\R^+_t;L^\infty)}\dd \tau.
\end{equation*}
Since $\|f\|_{\dot B^{-1}_{\infty,2}} \approx \|e^{t\Delta} f\|_{L^2(\R^+;L^\infty)}$ (see {\it e.g.} \cite{BCD11} or \cite{PGLR1})
and $\dot H^{\frac 12}(\R^3) \hookrightarrow \dot B^{-1}_{\infty,2}(\R^3)$, we get
\begin{align}\label{esh.v.int''}
  \|h\|_{L^2_T(L^\infty)} \lesssim\|v_0\|_{\dot H^{\frac12}}+\int_0^T \| \mathbb{P}\theta(\tau) \|_{\dot B^{-1}_{\infty,2}} \dd \tau
  \lesssim\|v_0\|_{\dot H^{\frac12}}+\|\theta_0\|_{L^1\cap L^s}T \lesssim E(T),
\end{align}
where in the last inequality we have used that, for all $s\in (3,\infty]$, one has
\begin{align*}
  \|\mathbb{P} \theta(t)\|_{\dot B^{-1}_{\infty,2}} \lesssim \|\theta(t)\|_{\dot B^{\frac{3}{s}-1}_{s,2}} \lesssim  \|\theta(t)\|_{L^1\cap L^s} \lesssim  \|\theta_0\|_{L^1\cap L^s}.
\end{align*}

Then, we will try to get a control of the  $L^2_T(L^\infty)$ of $w$. Using  \eqref{eqw.v.int} and then  Bernstein's inequality and Plancherel's formula, one gets
\begin{align}\label{es.w}
  \|\dot\Delta_jw(t)\|_{L^\infty} \lesssim 2^{\frac{5}{2}j}\int_0^t e^{-c2^{2j}(t-\tau)} \|\dot\Delta_j(v\otimes v)(\tau)\|_{L^2} \dd \tau.
\end{align}
Using Bony's para-product decomposition for any tempered distributions $a$ and $b$
\begin{align*}
  \dot\Delta_j(ab)=\sum_{j'\ge j-4}\Delta_j(\dot S_{j'}a\dot\Delta_{j'}b) + \sum_{j'\ge j-4}\Delta_j(\dot\Delta_{j'}a\dot S_{j'+1}b),
\end{align*}
we see that by H\"older's inequality, one has
\begin{align*}%\label{es.w1}
  \|\dot\Delta_j(v\otimes v)\|_{L^{\frac {2\rho}{\rho+2}}_T(L^2)}
  \lesssim \sum_{j'\ge j-4}\|\dot S_{j'+1}v\|_{L^\rho _T(L^\infty)}\|\dot\Delta_{j'}v\|_{L^2_T(L^2)} .
\end{align*}
Bernstein's inequality, Young's inequality and estimate \eqref{es.tec} allow us to get that, for all $\rho\in(2,\infty]$,
\begin{eqnarray*}
  \|\dot S_{j'+1} v\|_{L^\rho_T (L^\infty)} &\lesssim&  2^{j'(1-\frac2\rho)} \sum_{k\leq j'}
  2^{k(\frac{1}{2} + \frac{2}{\rho})} \|\dot\Delta_k v\|_{L^\rho_T (L^2)} 2^{(k-j')(1-\frac{2}{\rho})} \nonumber\\
  &\lesssim&  c_{j'} 2^{j'(1-\frac2\rho)}\frac \rho{\rho-2} E(T),
\end{eqnarray*}
where $\{c_j\}_{j\in \Z}$ satisfies $\|c_j\|_{\ell^2(\Z)}=1$.
Applying Young's inequality in time variable to \eqref{es.w}, one finds that
\begin{eqnarray*}%\label{facile-}
  \|\dot\Delta_j w \|_{L^2_T(L^\infty)} &\lesssim& 2^{\frac 52j} \|e^{-c2^{2j}t}\|_{L^{\frac{\rho}{\rho-1}}([0,T])}
  \|\dot\Delta_j(v\otimes v)( t)\|_{L^{\frac {2\rho}{\rho+2}}(0,T;L^2)} \nonumber\\
  &\lesssim& \frac \rho{\rho-2} E(T) \sum_{j'\ge j-4}c_{j'}\|\dot\Delta_{j'}v\|_{L^{2}_T(L^2)} 2^{\frac32j'}
  2^{(j-j')(\frac 12+\frac2\rho)}.
\end{eqnarray*}
Hence by taking the $\ell^1(\Z)$-norm on $j\in \Z$, together with \eqref{eq:v-H1/2es}, we find that, for all $\rho\in (3,\infty)$,
\begin{align*}
  \|w\|_{L^2_T(L^\infty)} \leq \sum_{j\in\Z}\|\dot\Delta_jw(t)\|_{L^2_T(L^\infty)}
  \leq C\frac \rho{\rho-2}\|v\|_{L^2_T(\dot H^{\frac{3}{2}})} E(T) \lesssim  E(T)^{\frac{3}{2}}.
\end{align*}
This, combined with \eqref{esh.v.int''}, gives the desired control \eqref{es.L2Linf}.

In order to show \eqref{es.uni0}, by making use of Lemma \ref{lem:heat-Besov}, Minkowski's inequality and the  fact that the  Leray projector $\mathbb{P}$ is bounded in $L^p\,(1<p<\infty)$ (\cite{Cann}, \cite{PGLR1}),  we get from the equation \eqref{eq.v.int} that for every $q \geq 2$,
\begin{align}\label{es.uni00}
  \|v\|_{L^\infty_T(\dot B^{-1+\frac3q}_{q,\infty})}+\|v\|_{\widetilde{L}^1_T(\dot B^{1+\frac3q}_{q,\infty})} 
  &\lesssim \|v_0\|_{\dot B^{-1+\frac3q}_{q,\infty}}+\|\mathbb{P}(v\cdot\nabla v-\theta e_3)\|_{\widetilde{L}^1_T(\dot B^{-1+\frac3q}_{q,\infty})}\nonumber\\
  &\lesssim \|v_0\|_{\dot H^{\frac12}}+\|v\cdot\nabla v\|_{{L}^1_T(\dot B^{-1+\frac3q}_{q,\infty})}+\|\theta \|_{{L}^1_T(\dot B^{-1+\frac3q}_{q,\infty})} .
\end{align}

Since $\theta$ has a rough regularity, we consider $q>3$ and then
\begin{align}\label{es.uni01}
  \|\theta \|_{L^1_T(\dot B^{-1+\frac{3}{q}}_{q,\infty})}\lesssim  \|\theta\|_{L^1_T(\dot B^0_{3,\infty})}\lesssim \|\theta\|_{L^1_T(L^1\cap L^3)}\lesssim \|\theta_0\|_{L^1\cap L^3}T.
\end{align}
Thanks to \eqref{eq:prod-es} with $(p,r) = (q,\infty)$ and $q>3$, we get, by H\"older's and Bernstein's inequalities, that
\begin{equation}\label{es.uni02}
  \|v\cdot \nabla v\|_{L^1_T(\dot B^{-1+\frac3q}_{q,\infty})} \lesssim   \|v\|_{L^2_T(L^\infty)} \|\nabla v\|_{L^2_T(\dot B^{-1+\frac3q}_{q,\infty})} \lesssim   \|v\|_{L^2_T(L^\infty)} \|v\|_{L^2_T(\dot H^{\frac32})} .
\end{equation}
Finally, by using \eqref{es.uni01} and \eqref{es.uni02} into \eqref{es.uni00}, together with  \eqref{eq:v-H1/2es} and \eqref{es.L2Linf} one obtains  the desired control \eqref{es.uni0}.

\end{proof}

\section*{Acknowledgments}
%%%%%%%%%%%%%%%
\noindent     L. Xue was partially supported by National Key Research and Development Program of China (No. 2020YFA0712900) and National Natural Science Foundation of China (No. 11771043).

\vfill

\begin{flushleft}

\textbf{Omar Lazar}\\
College of Engineering and Technology,\\
American University of the Middle East,\\
Kuwait\\

and\\

Departamento de An\'alisis Matem\'atico \& IMUS,\\
Universidad de Sevilla,\\
Spain

\vspace{1cm}

\textbf{Yatao Li}\\
Laboratory of Mathematics and Complex Systems (MOE)\\
School of Mathematical Sciences \\
Beijing Normal University, \\
Beijing 100875, P.R. China,\\

\vspace{1cm}

\textbf{Liutang Xue}\\
Laboratory of Mathematics and Complex Systems (MOE)\\
School of Mathematical Sciences \\
Beijing Normal University, \\
Beijing 100875, P.R. China,\\

\end{flushleft}

\end{document}